\documentclass[12pt]{amsart}
\usepackage{amssymb}
\usepackage{amsmath}
\usepackage{xcolor}
\usepackage[normalem]{ulem}
\usepackage{cancel}

\usepackage{full page}

\allowdisplaybreaks

\newtheorem{thm}{Theorem}[section]
\newtheorem{prop}[thm]{Proposition}
\newtheorem{cor}[thm]{Corollary}
\newtheorem{lem}[thm]{Lemma}
\newtheorem{rem}[thm]{Remark}
\newtheorem{defn}[thm]{Definition}

\newtheorem{conj}[thm]{Conjecture}

\numberwithin{equation}{section}

\newcommand{\C}{\mathbb{C}}
\newcommand{\N}{\mathbb{N}}

\newcommand{\res}{\mbox{\rm Res}}
\newcommand{\wt}{\mbox{\rm wt}\ }

\newcommand{\nord}{\mbox{\scriptsize ${\circ\atop\circ}$}}

\begin{document}

\title[$A_2(V)$ for the Heisenberg and Virasoro VOAs]{The level two Zhu algebra  for the Heisenberg vertex operator algebra}

\author{Darlayne Addabbo}
\address{Department of Mathematics, University of Arizona, Tucson, AZ  85721.}\email{addabbo@math.arizona.edu}

\author{Katrina Barron}
\address{Department of Mathematics, University of Notre Dame, Notre Dame, IN 46556}
\email{kbarron@nd.edu}

%    General info

\date{March 18, 2023}

\thanks{D. Addabbo was supported by AMS-Simons Foundation Travel Grant; K. Barron was supported by Simons Foundation Collaboration Grant 282095.}


\begin{abstract}
We determine the level two Zhu algebra for the Heisenberg vertex operator algebra $V$ for any choice of conformal element.  We do this using only the following information for $V$: the internal structure of $V$; the level one Zhu algebra of $V$ already determined by the second author, along with Vander Werf and Yang; and the information the lower level Zhu algebras give regarding irreducible modules.  We are able to carry out this calculation of the level two Zhu algebra for $V$ with this minimal information by employing the general results and techniques for determining generators and relations for higher level Zhu algebras for a vertex operator algebra, as developed previously by the authors in {\it On generators and relations for higher level Zhu algebras and applications}, by Addabbo and Barron, {\it J. Algebra,} 2023.   In particular, we show that the level $n$ Zhu algebras for the Heisenberg vertex operator algebra become noncommutative at level $n=2$.  We also give a conjecture for the structure of the level $n$ Zhu algebra for the Heisenberg vertex operator algebra,  for any $n >2$.
\end{abstract}

\maketitle

%\tableofcontents

\section{Introduction}

The rank one Heisenberg vertex operator algebra, also called the free boson, is arguably the most fundamental vertex operator algebra.  Zero level Zhu algebras for a vertex operator algebra $V$ are important tools in studying the irreducible modules for $V$, and the higher level Zhu algebras are important tools for studying indecomposable non simple $V$-modules, and the number theoretic properties of the category of such $V$-modules.
In this paper, we completely determine the structure of the level two Zhu algebra for the Heisenberg vertex operator algebra, denoted by $M_a(1)$ for conformal element $\omega_a = \frac{1}{2} \alpha(-1)^2 {\bf 1} + a \alpha(-2) {\bf 1}$ for $a\in \mathbb{C}$,  using only the internal structure of $M_a(1)$, the level one Zhu algebra for $M_a(1)$ determined  previously by the second author  along with Vander Werf and Yang in \cite{BVY-Heisenberg}, and information about the irreducible modules for $M_a(1)$ already determined by the lower level Zhu algebras.  

The determination of the level two Zhu algebra for $M_a(1)$ using this minimal information about $M_a(1)$ is possible due to the results and techniques developed by the authors of this paper and given in \cite{AB-general-n}  for determining generators and relations for level $n$ Zhu algebras for vertex operator algebras in general.  In fact this paper serves as an important example for the use of the results and techniques developed in \cite{AB-general-n}, which builds on the initial techniques for determining generators and relations for the Heisenberg and Virasoro vertex operator algebra for the level one Zhu algebra as developed by the second author, as well as Vander Werf and Yang in \cite{BVY-Heisenberg} and \cite{BVY-Virasoro}. 

This determination of the generators and relations for the level two Zhu algebra of the Heisenberg vertex operator algebra has applications to the construction of higher level Zhu algebras for other vertex operator algebras, especially those which contain Heisenberg subalgebras, as well as applications to understanding the Zhu algebras (including level zero Zhu algebras) for orbifold vertex operator algebras related to the Heisenberg vertex operator algebra. 
 
\subsection{Background and motivation}

Given a vertex operator algebra $V$, in \cite{Z1}, \cite{Z2} Zhu defined an associative algebra, denoted by $A(V)$, and called ``the level zero Zhu algebra of $V$".  Zhu used this algebra to prove the modular invariance of graded characters of $V$ for the case when $V$ has a certain finiteness property called ``$C_2$-cofinite" and when $V$ is ``rational", meaning $V$ has semi-simple representation theory.   In \cite{Z1} Zhu proved that there is a  bijective correspondence between the isomorphism classes of irreducible $A(V)$-modules and irreducible $V$-modules.

Since their inception, level zero Zhu algebras have proven to be extremely powerful tools in the study of vertex operator algebras and have been used extensively throughout the literature. In \cite{W}, Wang determined the structure of the Zhu algebras for Virasoro vertex operator algebras and used this to prove a conjecture of Frenkel and Zhu \cite{FZ} that a Virasoro vertex operator algebra is rational if and only if its central charge  is of the form $c=1-\frac{6(p-q)^2}{pq}$ where $p$ and $q$ are relatively prime integers greater than or equal to $2$. 

In \cite{abe}, Abe used aspects of the level zero Zhu algebra to classify the irreducible representations of the $\mathbb{Z}_2$-orbifold of the symplectic fermions, denoted $SF^+$, and showed that there are indecomposible non irreducible modules for this vertex operator algebra, giving the first example of an irrational $C_2$-cofinite vertex operator algebra.  It had previously been conjectured that if a vertex operator algebra is $C_2$-cofinite then it is also rational.

The work of Abe on $SF^+$, which is denoted $\mathcal{W}(2)$ in $\mathcal{W}$-algebra notation, was later extended by Adamovi\'c and Milas  \cite{am1},  \cite{AM} who calculated the level zero Zhu algebras for the general family of triplet vertex algebras $\mathcal{W}(p)$ of central charge  $c = 1-\frac{6(p-1)^2}{p}$, $p\ge 2$, and used these Zhu algebras  to prove that these $C_2$-cofinite vertex operator algebras are irrational. In addition, they  classified their irreducible representations. 

Other important classes of vertex operator algebras were studied by Dong and Nagatomo in \cite{DN1} and \cite{DN2}, who used level zero Zhu algebras to classify the irreducible modules of $M(1)^+$ and $V_L^+$, respectively,  where $L$ is a positive definite even lattice,  and the $V^+$ notation for $V = M(1)$ or $V = V_L$ denotes the fixed point,  or orbifold, vertex operator algebra under an involution in the automorphism group of $V$.

More recently, in \cite{L} Linshaw calculated the level zero Zhu algebra for the Weyl vertex algebra of central charge $c=2$ (also called the $\beta\gamma$-system or the free bosonic ghost vertex algebra) in the setting of a conformal vertex algebra which is not a vertex operator algebra due to having infinite dimensional graded components.  In \cite{LM}, Laber and Mason use Zhu algebras to study $\mathbb{C}$-graded conformal vertex algebras, and in \cite{betagamma}, the second named author of the current paper, along with Batistelli, Orosz Hunziker, Pedi\'c, and Yamskulna, calculate the level zero Zhu algebra for certain $\mathbb{C}$-graded conformal Weyl vertex algebras under conformal flow and determine for which conformal elements these vertex algebras have semi-simple representation theory for their admissible modules.

While the level zero Zhu algebras have been used to prove many important results in the theory of  vertex operator algebras and their modules, if one wants to study irrational vertex operator algebras and their indecomposable modules, these algebras are  usually not enough  since in general there can be indecomposable non irreducible modules not detected by the level zero Zhu algebra. For this reason, in \cite{DLM}, Dong, Li, and Mason introduced an infinite family $\{A_n(V) \; | \; n\in \mathbb{N}\}$ of associative algebras for a vertex operator algebra $V$, where  $A_0(V)$ is the original level zero Zhu algebra. For $n>0$, these algebras are known as ``higher level Zhu algebras".  Dong, Li, and Mason presented several statements that generalize the results of Frenkel and Zhu from the level zero algebras to these higher level algebras, results that mainly focused on the semi-simple setting, e.g., the case of rational vertex operator algebras.

Since the level zero Zhu algebras play an important role in Zhu's modular invariance theorem for rational vertex operator algebras, it is natural  to expect that the higher level Zhu algebras should play a similar role for  irrational  vertex operator algebras. In \cite{Miyamoto2004} Miyamoto showed that this is indeed the case, proving that the higher level Zhu algebras can be used to construct ``graded pseudo-traces" for irrational $C_2$-cofinite vertex operator algebras.  These graded pseudo-traces are generalizations of the graded characters (or graded traces) studied by Zhu \cite{Z1}, \cite{Z2} and Miyamoto showed that the graded pseudo-traces for  irrational $C_2$-cofinite vertex operator algebras  have similar modular invariance properties.

In \cite{BVY}, the second author, along with Vander Werf, and Yang further studied higher level Zhu algebras, proving several important results about the relationship between the module categories of vertex operator algebras and the module categories of higher level Zhu algebras. They described settings for when the higher level Zhu algebras detect modules not seen by the the level zero Zhu algebra,  and more generally, when the level $n$ Zhu algebra detects modules not seen by the level $n-1$ Zhu algebra.  In addition, they provided necessary modifications of some statements made in \cite{DLM}.

Although  the higher level Zhu algebras are important objects in the study of modules for irrational vertex operator algebras and related questions of modular invariance of their graded pseudo-traces, determining the  structure of the higher level Zhu algebras is quite difficult. In \cite{BVY-Heisenberg} and \cite{BVY-Virasoro}, the second author, along with Vander Werf and Yang, constructed the first examples of higher level Zhu algebras, determining the structure of the level one Zhu algebras for the Heisenberg and Virasoro vertex operator algebras, respectively. In his PhD thesis  \cite{C}, \^Ceperi\'c calculated the level one Zhu algebra for the fixed point subalgebra of symplectic fermions, $SF^+$.

In \cite{AB-general-n} the authors proved many results helpful in determining the structure of higher level Zhu algebras. In particular, in \cite{AB-general-n}, the authors provide  guidelines for finding minimal generating sets for higher level Zhu algebras and prove  relations satisfied by these generators. They outline techniques on how to use  level $n-1$ Zhu algebras 
and the information given for modules induced at that level to determine relations for the generators of the level $n$ Zhu algebra.  These results are used heavily in the present paper.  In \cite{AB-general-n}, the authors also prove the necessity of a condition that is sometimes omitted in the definition of higher level Zhu algebras, and indicate why this fact is useful in calculating higher level Zhu algebras for certain vertex operator algebras.

In this paper, we present the first example of a higher level Zhu algebra greater than level one, describing $A_2(M_a(1))$ explicitly in terms of generators and relations. We expect that the insight obtained from this construction will motivate future constructions of higher level Zhu algebras, and will be useful in determining the relationship between the level zero Zhu algebras for certain fixed point vertex operator algebras under finite order automorphisms in the orbifold setting.

The main results of this paper were previously announced in \cite{AB-general-n} without proofs and details.

\subsection{Organization}

This paper is organized as follows, in Section 2, we follow \cite{DLM} and \cite{BVY} in defining the higher level Zhu algebras $A_n(V)$ for a vertex operator algebra $V$, the functor $\Omega_n$ from the category of $\mathbb{N}$-gradable $V$-modules to the category of $A_n(V)$-modules,  and the functor $L_n$ from the category of $A_n(V)$-modules to the category of $\mathbb{N}$-gradable $V$-modules, for $n \in \mathbb{N}$.  We then  recall some properties about these algebras and their relationship to the modules for $V$ from  \cite{DLM} and \cite{BVY}. 
 
In Section 3, we recall some general facts about the Heisenberg vertex operator algebra, its family of possible conformal elements $\omega_a$, and its irreducible modules.  Here we note that although the structure of the Zhu algebra does not depend on the choice of conformal element in the case of the Heisenberg algebra, in other settings, such as the case of the Weyl vertex algebras, the choice of conformal element drastically changes the structure of the Zhu algebra, cf. \cite{betagamma}.  Thus throughout the paper, to emphasize this fact, we continue to use the notation $M_a(1)$ to recall that the results we develop do apply to the Heisenberg vertex operator algebra with {\it any} choice of conformal element.   

 In Section 4, we give some general results about the generators and relations of the level $n$ Zhu algebra for the Heisenberg vertex operator algebra that are corollaries of general results from \cite{AB-general-n}.

In Section 5, we present the results from Section 3 on some generators and relations for the Zhu algebra in the special case when $n = 2$, i.e., for the special case in which we are determining the structure of the level two Zhu algebra for the Heisenberg vertex operator algebra.  The results of \cite{AB-general-n} allow us to determine an infinite generating set for $A_2(M_a(1))$, where each generator can be described in terms of only certain modes of $M_a(1)$. We then prove many relations and recursions satisfied by the elements of this infinite generating set, and use these recursions and relations to prove that $A_2(M_a(1))$ can be generated by only five elements. 
 
 In Section 6, we determine relations for the five-element generating set obtained in Section 5 using  methods introduced by the second author along with Vander Werf and Yang in \cite{BVY-Heisenberg}, \cite{BVY-Virasoro} and further developed by the authors of the current paper in  \cite{AB-general-n}.

In Section 7, we use the relations determined in Section 6 to determine the structure of $A_2(M_a(1))$ as an algebra.  In particular, we show that 
\[A_2(M_a(1))\cong A_1(M_a(1))\oplus (\mathbb{C}[x]\otimes M_2(\mathbb{C})).\]

In Section 8, we recall a conjecture previously stated in \cite{AB-general-n} that was inspired by the results of this paper,  namely that for $n \in \mathbb{Z}_+$
\[ A_n(M_a(1)) \cong A_{n-1} (M_a(1)) \oplus \left(\mathbb{C}[x] \otimes M_{p(n)}(\mathbb{C}) \right),\]
where $p(n)$ denotes the number of unordered partitions of $n$ into nonnegative integers, and $M_{p(n)}(\mathbb{C}$ denotes the algebra of $p(n) \times p(n)$ matrices.

The Appendix includes several computations necessary for the proof of Lemma \ref{lower-order-lemY} which gives the higher order terms of the relations satisfied by the five-element generating set for $A_2(M_a(1))$, reducing the determination of relations to computations which rely on the structure of the level one Zhu algebra $A_1(M_a(1))$ and is a key step in proving all the relations for the generators of $A_2(M_a(1))$.

{\bf Acknowledgments: } The first author was supported by an AMS-Simons Foundation Travel Grant, and the second author was supported by the Simons Foundation Collaboration Grant 282095.  We thank the AMS and Simons Foundation and greatly appreciate their support. 

\section{The algebras $A_n(V)$, and the functors $\Omega_n$ and $L_n$}\label{zhu-algebra-definition-section}

In this section, we recall the definition and some properties of the algebras $A_n(V)$ for $n \in \N$, first introduced in \cite{Z1} for $n = 0$, and then generalized to $n >0$ in \cite{DLM}.   We then recall the functors $\Omega_n$ and $L_n$ defined in \cite{DLM}, and we recall some results from \cite{BVY} and \cite{AB-general-n}.

For $n \in \N$, let $O_n(V)$ be the subspace of $V$ spanned by elements of the form
\begin{equation}\label{define-circ}
u \circ_n v =
\res_x \frac{(1 + x)^{\mathrm{wt}\, u + n}Y(u, x)v}{x^{2n+2}}
\end{equation}
for all homogeneous $u \in V$ and for all $v \in V$, and by elements of the form $(L(-1) + L(0))v$ for all $v \in V$. The vector space $A_n(V)$ is defined to be the quotient space $V/O_n(V)$.

That is, if we introduce the notations 
\begin{equation}\label{define-O-subsets} 
O^L (V) = \{ (L(-1) + L(0))v \, | \, v \in V \} \quad
\mathrm{and}
\quad O_n^\circ(V) = \mathrm{span} \{ u \circ_n v \; | \; u,v \in V\} ,
\end{equation} 
then
\begin{equation}\label{define-O}  
O_n(V) =  O^L (V) + O_n^\circ(V) .
\end{equation}

\begin{rem}\label{first-remark}{\em
For $n=0$, since $v \circ_0 \mathbf{1} = v_{-2} \mathbf{1} + (\mathrm{wt} \, v) v = L(-1) v + L(0) v$, it 
follows that $(L(-1) + L(0)) v \in O^\circ_0(V)$ for all $v \in V$ and thus $O_0(V) = O_0^\circ(V)$. 
However, we show in \cite{AB-general-n} that for $n > 0$, in 
general, $O^L (V) \not\subset O_n^\circ(V)$.  Thus for $A_n(V)$ to have a well-defined action on 
modules via zero modes,  it is necessary to not just define $O_n(V)$ to consist of $O_n^\circ(V)$ as is
 sometimes done, but to also include $O^L(V)$ as part of the definition; cf. \cite{DLM}, 
 \cite{Miyamoto2004}.  This is simply due to the fact that from the $L(-1)$-derivative property, the zero
 mode of $L(0)v$ which is $(\mathrm{wt} \, v)v_{\mathrm{wt} \, v - 1}$ is equal to the negative of the zero mode of $L(-1)v$.   In \cite{AB-general-n} we show that in general $O^L (V) \not\subset O_n^\circ(V)$, by
showing that, for instance, any time $V$ contains an element $u$ of weight one, or of weight zero 
such that $u_{-2} {\bf 1} \neq 0$,  then  $(L(-1) + L(0) )u \notin O_n^\circ(V)$ if $n >0$.  For instance, if $u$ generates a 
Heisenberg vertex subalgebra in $V$, or if $V$ has a nontrivial vacuum (weight zero) space, then $u$ will have this property.   Thus it is quite common for
a vertex operator algebra $V$ to have the property that $O^L (V) \not\subset O_n^\circ(V)$.  More
importantly, the presence of these $(L(-1) + L(0)) v \in O_n(V)$ is used heavily in the constructions of the higher level Zhu algebras in \cite{BVY-Heisenberg}, \cite{BVY-Virasoro},  and in this paper.}
\end{rem}

We define the following multiplication on $V$
\begin{equation}\label{*_n-definition}
u *_n v = \sum_{m=0}^n(-1)^m\binom{m+n}{n}\res_x \frac{(1 + x)^{\mathrm{wt}\, u + n}Y(u, x)v}{x^{n+m+1}},
\end{equation}
for $v \in V$ and homogeneous $u \in V$, and for general $u \in V$, $*_n$ is
defined by linearity.   It is shown in \cite{DLM} that with this multiplication, the subspace $O_n(V)$ of $V$ is a two-sided ideal of $V$, and $A_n(V)$ is an associative algebra, called the {\it level $n$ Zhu algebra}, with multiplicative identity $\mathbf{1} + O_n(V)$.  

The lemma below was proved by Zhu \cite{Z1} for the case $n=0$ and his proof extends easily to the $n>0$ case, as is done in \cite{AB-general-n}.

From \cite{DLM}, we have the following important results which will help in determining the structure of $A_2(M_a(1))$ given the structure of $A_1(M_a(1))$.  

\begin{lem}\cite{DLM} \label{DLM-commutator-lemma} Let $V$ be a vertex operator algebra. 

(i) The map
\begin{eqnarray}\label{surjection}
A_n(V) & \longrightarrow & A_{n-1}(V) \\
v + O_n(V) & \mapsto & v + O_{n-1}(V) \nonumber
\end{eqnarray}
is a surjective algebra homomorphism.

(ii) If $\omega$ is the conformal vector of $V$, then $\omega+O_n(V)$ is central in $A_n(V)$ for all $n \in \mathbb{N}$. 

(iii)  In general 
\begin{equation}\label{dlmcommutation}
u*_ n v-v*_n u \equiv_n \res_x Y(u,x)v(1+x)^{wtu-1}.
\end{equation}
\end{lem}

Next we recall the definitions of various $V$-module structures.  We assume the reader is familiar with the notion of weak $V$-module for a vertex operator algebra $V$ (cf. \cite{LL}).  

\begin{defn}\label{N-gradable-definition}
{\em An {\it $\mathbb{N}$-gradable weak $V$-module} (also often called an {\it admissible $V$-module} as in \cite{DLM}) $W$ for a vertex operator algebra $V$ is a weak $V$-module that is $\mathbb{N}$-gradable, $W = \coprod_{k \in \mathbb{N}} W(k)$, with $v_m W(k) \subset W(k + \mathrm{wt} v - m -1)$ for homogeneous $v \in V$, $m \in \mathbb{Z}$ and $k \in \N$, and without loss of generality, we can and do assume $W(0) \neq 0$, unless otherwise specified.  We say elements of $W(k)$ have {\it degree} $k \in \mathbb{N}$.

An {\it $\mathbb{N}$-gradable generalized weak $V$-module} $W$ is an $\mathbb{N}$-gradable weak $V$-module that admits a decomposition into generalized eigenspaces via the spectrum of $L(0) = \omega_1$ as follows: $W=\coprod_{\lambda \in{\C}}W_\lambda$ where $W_{\lambda}=\{w\in W \, | \, (L(0) - \lambda \, id_W)^j w= 0 \ \mbox{for some $j \in \mathbb{Z}_+$}\}$, and in addition, $W_{n +\lambda}=0$ for fixed $\lambda$ and for all sufficiently small integers $n$. We say elements of $W_\lambda$ have {\it weight} $\lambda \in \mathbb{C}$.

A {\it generalized $V$-module} $W$ is an $\mathbb{N}$-gradable generalized weak $V$-module where $\dim W_{\lambda}$ is finite for each $\lambda \in \mathbb{C}$.   

An {\it (ordinary) $V$-module} is a  generalized $V$-module such that  the generalized eigenspaces $W_{\lambda}$ are in fact eigenspaces, i.e., $W_{\lambda}=\{w\in W \, | \, L(0) w=\lambda w\}$.}
\end{defn}

We will often omit the term ``weak" when referring to $\mathbb{N}$-gradable weak and $\mathbb{N}$-gradable generalized weak $V$-modules.   

The term {\it logarithmic} is also often used in the literature to refer to $\mathbb{N}$-gradable weak generalized modules  or generalized modules. 

\begin{rem}{\em An $\mathbb{N}$-gradable $V$-module with $W(k)$ of finite dimension for each $k \in \mathbb{N}$ is not necessarily a generalized $V$-module since the generalized eigenspaces might not be finite dimensional. }
\end{rem}

%We define the {\it generalized graded dimension} of a generalized $V$-module $W = \coprod_{\lambda \in{\C}}W_\lambda$ to be
%\begin{equation}
%\mathrm{gdim}_q  \, W = q^{-c/24} \,  \sum_{\lambda \in \C}  \, ( \mathrm{dim} \ W_\lambda) \, q^\lambda.
%\end{equation}

Next, we recall the functors $\Omega_n$ and $L_n$, for $n \in \N$, defined and studied in \cite{DLM}.  Let $W$ be an $\mathbb{N}$-gradable $V$-module, and let
\begin{equation}\label{Omega}
\Omega_n(W) = \{w \in W \; | \; v_iw = 0\;\mbox{if}\; \wt v_i < -n \; 
\mbox{for $v\in V$ of homogeneous weight}\}.
\end{equation}
It was shown in \cite{DLM} that $\Omega_n(W)$ is an $A_n(V)$-module
via the action $o(v+O_n(V)) = v_{\mathrm{wt} \, v -1}$ for $v \in V$.   In particular, this action satisfies $o(u *_n v) = o(u)o(v)$ for $u,v \in A_n(V)$.

Furthermore, it was shown in \cite{DLM} and \cite{BVY} that there is a bijection between the isomorphism classes of irreducible $A_n(V)$-modules which cannot factor through $A_{n-1}(V)$ and the isomorphism classes of irreducible $\mathbb{N}$-gradable $V$-modules with nonzero degree $n$ component.  

In order to define the functor $L_n$ from the category of $A_n(V)$-modules to the category of $\mathbb{N}$-gradable $V$-modules, we need several notions, including the notion of the universal enveloping algebra of $V$, which we now define.  

Let
\begin{equation}
\hat{V} = \C[t, t^{-1}]\otimes V/D\C[t, t^{-1}]\otimes V,
\end{equation}
where $D = \frac{d}{dt}\otimes 1 + 1 \otimes L(-1)$. For $v \in V$, let $v(m) = v \otimes t^m +  D\C[t, t^{-1}]\otimes V \in \hat{V}$.  Then $\hat{V}$ can be given the structure of a $\mathbb{Z}$-graded Lie algebra as follows:  Define the degree of $v(m)$ to be $\wt v - m - 1$ for homogeneous $v \in V$, and define the Lie bracket on $\hat{V}$ by
\begin{equation}\label{bracket}
[u(j), v(k)] = \sum_{i = 0}^{\infty}\binom{j}{i}(u_iv)(j+k-i),
\end{equation}
for $u, v \in V$, $j,k \in \mathbb{Z}$.
Denote the homogeneous subspace of degree $m$ by $\hat{V}(m)$. In particular, the degree $0$ space of $\hat{V}$, denoted by $\hat{V}(0)$, is a Lie subalgebra.

Denote by $\mathcal{U}(\hat{V})$ the universal enveloping algebra of the Lie algebra $\hat{V}$.  Then $\mathcal{U}(\hat{V})$ has a natural $\mathbb{Z}$-grading induced from $\hat{V}$, and we denote by $\mathcal{U}(\hat{V})_k$ the degree $k$ space with respect to this grading, for $k \in \mathbb{Z}$.

We can regard $A_n(V)$ as a Lie algebra via the bracket $[u,v] = u *_n v - v *_n u$, and then the map $v( \mathrm{wt} \, v -1) \mapsto v + O_n(V)$ is a well-defined Lie algebra epimorphism from $\hat{V}(0)$ onto $A_n(V)$.

Let $U$ be an $A_n(V)$-module.  Since $A_n(V)$ is naturally a Lie algebra homomorphic image of $\hat{V}(0)$, we can lift $U$ to a module for the Lie algebra $\hat{V}(0)$, and then to a module for $P_n = \bigoplus_{p > n}\hat{V}(-p) \oplus \hat{V}(0) = \bigoplus_{p < -n} \hat{V}(p) \oplus \hat{V}(0)$ by letting $\hat{V}(-p)$ act trivially for $p\neq 0$.  Define
\[
M_n(U) = \mbox{Ind}_{P_n}^{\hat{V}}(U) = \mathcal{U}(\hat{V})\otimes_{\mathcal{U}(P_n)}U.
\]

We impose a grading on $M_n(U)$ with respect to $U$, $n$, and the $\mathbb{Z}$-grading on $\mathcal{U}(\hat{V})$, by letting $U$ be degree $n$, and letting $M_n(U)(i)$ be the $\mathbb{Z}$-graded subspace of $M_n(U)$ induced from $\hat{V}$, i.e., $M_n(U)(i) = \mathcal{U}(\hat{V})_{i-n}U$. 

For $v \in V$, define $Y_{M_n(U)}(v,x) \in (\mathrm{End} (M_n(U)))((x))$ by
\begin{equation}\label{define-Y_M}
Y_{M_n(U)}(v,x) = \sum_{m\in\mathbb{Z}} v(m) x^{-m-1}.
\end{equation} 

Let $W_{A}$ be the subspace of $M_n(U)$ spanned linearly by the coefficients of 
\begin{multline}\label{relations-for-M}
(x_0 + x_2)^{\mathrm{wt} \, v + n} Y_{M_n(U)}(v, x_0 + x_2) Y_{M_n(U)}(w, x_2) u \\ 
- (x_2 + x_0)^{\mathrm{wt} \, v + n} Y_{M_n(U)}(Y(v, x_0)w, x_2) u
\end{multline}
for $v,w \in V$, with $v$ homogeneous, and $u \in U$.  Set
\[ \overline{M}_n(U) = M_n(U)/\mathcal{U} (\hat{V})W_A .\]

It is shown in \cite{DLM} that if $U$ is an $A_n(V)$-module that does not factor through $A_{n-1}(V)$, then $\overline{M}_n(U) = \bigoplus_{i \in \mathbb{N}} \overline{M}_n(U) (i)$ is an $\mathbb{N}$-gradable $V$-module with $\overline{M}_n(U) (0)\neq 0$, and  as an $A_n(V)$-module,  $\overline{M}_n(U) (n) \cong U$. Note that the condition that $U$ itself does not factor though $A_{n-1}(V)$ is indeed a necessary and sufficient condition for $\overline{M}_n(U) (0)\neq 0$ to hold.  

It is also observed in \cite{DLM} that $\overline{M}_n(U)$ satisfies the following universal property:  For any weak $V$-module $M$ and any $A_n(V)$-module homomorphism $\phi: U \longrightarrow \Omega_n(M)$, there exists a unique weak $V$-module homomorphism $\Phi: \overline{M}_n(U) \longrightarrow M$, such that $\Phi \circ \iota = \phi$ where $\iota$ is the natural injection of $U$ into $\overline{M}_n(U)$. This follows from the fact that $\overline{M}_n(U)$ is generated by $U$ as a weak $V$-module, again with the possible need of a grading shift.

Let $U^* = \mbox{Hom}(U, \C)$.  As in the construction in \cite{DLM}, we can extend $U^*$ to $M_n(U)$ by first an induction to $M_n(U)(n)$ and then by letting $U^*$ annihilate $\bigoplus_{i \neq n} M_n(U)(i)$.  In particular, we have that elements of $M_n(U)(n) = \mathcal{U}(\hat{V})_0U$ are spanned by elements of the form 
\[o_{p_1}(a_1) \cdots o_{p_s}(a_s)U\]
where $s \in \mathbb{N}$, $p_1 \geq \cdots \geq p_s$, $p_1 + \cdots + p_s =0$, $p_i \neq 0$, $p_s \geq -n$, $a_i \in V$ and $o_{p_i}(a_i) = (a_i)(\mathrm{wt} \, a_i - 1 - p_i)$. Then inducting on $s$ by using Remark 3.3 in \cite{DLM} to reduce from length $s$ vectors to length $s-1$ vectors, we have a well-defined action of $U^*$ on $M_n(U)(n)$.  

Set
\[
J = \{v \in M_n(U) \, | \, \langle u', xv\rangle = 0 \;\mbox{for all}\; u' \in U^{*}, x \in \mathcal{U}(\hat{V})\}
\]
and
\[
L_n(U) = M_n(U)/J.
\]

\begin{rem}\label{L-a-V-module-remark} {\em It is shown in \cite{DLM}, Propositions 4.3, 4.6 and 4.7,  that if $U$ does not factor through $A_{n-1}(V)$, then $L_n(U)$ is a well-defined $\mathbb{N}$-gradable $V$-module with $L_n(U)(0) \neq 0$; in particular, it is shown that $\mathcal{U}(\hat{V})W_A \subset J$, for $W_A$ the subspace of $M_n(U)$ spanned by the coefficients of (\ref{relations-for-M}), i.e., giving the associativity relations for the weak vertex operators on $M_n(U)$.}
\end{rem}

We have the following theorem from \cite{BVY}:

\begin{thm}\label{mainthm}\cite{BVY}
For $n \in \N$, let $U$ be a nonzero $A_n(V)$-module such that if $n>0$, then $U$ does not factor through $A_{n-1}(V)$. Then $L_n(U)$ is an $\mathbb{N}$-gradable $V$-module with $L_n(U)(0) \neq 0$.  If we assume further that there is no nonzero submodule of $U$ that factors through $A_{n-1}(V)$, then $\Omega_n/\Omega_{n-1}(L_n(U)) \cong U$.
\end{thm}

One of the main reasons we are interested in Theorem \ref{mainthm} is what it implies for the question of when modules for the higher level Zhu algebras give rise to indecomposable non simple modules for $V$ not seen by the lower level Zhu algebras.

To this end, we recall the following two corollaries to Theorem \ref{mainthm} from \cite{BVY}:

\begin{cor}\label{mainthm-first-cor}\cite{BVY}  Suppose that for some fixed $n \in \mathbb{Z}_+$, $A_n(V)$ has a direct sum decomposition $A_n(V) \cong A_{n-1}(V) \oplus A'_n(V)$, for $A'_n(V)$ a direct sum complement to $A_{n-1}(V)$, and let $U$ be an $A_n(V)$-module.   If $U$ is trivial as an  $A_{n-1}(V)$-module, then   $\Omega_n/\Omega_{n-1} (L_n(U) ) \cong U$. 
\end{cor}

An example of this setting is given by the level one Zhu algebra for the Heisenberg vertex operator algebra as constructed in \cite{BVY-Heisenberg} as well as by the level two Zhu algebra for the Heisenberg vertex operator algebra as constructed in this paper.  An example of when a higher level Zhu algebra does not decompose into a direct sum with a lower level Zhu algebra as a direct summand was constructed in \cite{BVY-Virasoro} at level one and in that paper it is shown that there are new indecomposable modules that are induced at level one that are not detected by level zero.  

\begin{cor}\label{mainthm-cor}\cite{BVY}
Let $n \in \mathbb{Z}_+$ be fixed, and let $U$ be a nonzero indecomposable $A_n(V)$-module such that there is no nonzero submodule of $U$ that can factor through $A_{n-1}(V)$. Then $L_n(U)$ is a nonzero indecomposable $\mathbb{N}$-gradable $V$-module generated by its degree $n$ subspace, $L_n(U) (n) \cong U$, and satisfying 
\[\Omega_n/\Omega_{n-1} (L_n(U)) \cong L_n(U)(n) \cong U\] 
as $A_n(V)$-modules.

Furthermore if $U$ is a simple $A_n(V)$-module, then $L_n(U)$ is a simple $V$-module as well.
\end{cor}

\section{The Heisenberg vertex operator algebra, its level zero and level one Zhu algebras, and its irreducible modules}\label{Heisenberg-section}

We denote by $\mathfrak{h}$ a one-dimensional abelian Lie algebra spanned by $\alpha$ with a bilinear form $\langle \cdot, \cdot \rangle$ such that $\langle \alpha, \alpha \rangle = 1$, and by
\[
\hat{\mathfrak{h}} = \mathfrak{h}\otimes \C[t, t^{-1}] \oplus \C \mathbf{k}
\]
the affinization of $\mathfrak{h}$ with bracket relations
\[
[a(m), b(n)] = m\langle a, b\rangle\delta_{m+n, 0}\mathbf{k}, \;\;\; a, b \in \mathfrak{h},
\]
\[
[\mathbf{k}, a(m)] = 0,
\]
where we define $a(m) = a \otimes t^m$ for $m \in \mathbb{Z}$ and $a \in \mathfrak{h}$.

Set
\[
\hat{\mathfrak{h}}^{+} =\mathfrak{h} \otimes   t\C[t] \qquad \mbox{and} \qquad \hat{\mathfrak{h}}^{-} = \mathfrak{h} \otimes t^{-1}\C[t^{-1}].
\]
Then $\hat{\mathfrak{h}}^{+}$ and $\hat{\mathfrak{h}}^{-}$ are abelian subalgebras of $\hat{\mathfrak{h}}$.  Consider the induced $\hat{\mathfrak{h}}$-module given by 
\[
M(1) = U(\hat{\mathfrak{h}})\otimes_{U(\C[t]\otimes \mathfrak{h} \oplus \C c)} \C{\bf 1} \simeq S(\hat{\mathfrak{h}}^{-}) \qquad \mbox{(linearly)},
\]
where $U(\cdot)$ and $S(\cdot)$ denote the universal enveloping algebra and symmetric algebra, respectively, $\mathfrak{h} \otimes \C[t]$ acts trivially on $\mathbb{C}\mathbf{1}$ and $\mathbf{k}$ acts as multiplication by $1$.   Then $M(1)$ is a vertex operator algebra, often called the {\it vertex operator algebra associated to the rank one Heisenberg}, or the {\it rank one Heisenberg vertex operator algebra}, or the {\it one free boson vertex operator algebra}.  Here, the Heisenberg Lie algebra in question being precisely $\hat{\mathfrak{h}} \diagdown \mathbb{C}\alpha(0)$.

Any element of $M(1)$ can be expressed as a linear combination of elements of the form
\begin{equation}\label{generators-for-V}
\alpha(-k_1)\cdots \alpha(-k_j){\bf 1}, \quad \mbox{with} \quad  k_1 \geq \cdots \geq k_j \geq 1.
\end{equation}

The conformal element for $M(1)$ is given by 
\[
\omega = \frac{1}{2} \alpha(-1)^2{\bf 1}.
\]
However we can also shift the conformal element by 
\[ \omega_a = \frac{1}{2} \alpha(-1)^2{\bf 1} + a \alpha(-2) \mathbf{1},\]
for any $a \in \mathbb{C}$, and then $M(1)$ with the conformal element $\omega_a$ is a vertex operator algebra with central charge $c = 1 - 12 a^2$.  We denote this vertex operator algebra with shifted conformal elements by $M_a(1)$.    So, for instance $M(1) = M_0(1)$. 

Writing $Y (\omega_a,x) = \sum_{k \in \mathbb{Z}} L_a (k) x^{-k-2}$ we have that the representation of the Virasoro algebra corresponding to the vertex operator algebras structure $(M_a(1), Y, \mathbf{1}, \omega_a)$ is given as follows:  
For $n$ even
\begin{equation}\label{Ln-even}
L_{a} (-n) = \sum_{k = n/2 + 1}^\infty \alpha(-k)\alpha(-n+k) + \frac{1}{2} \alpha(-n/2)^2 + a (n-1) \alpha(-n),
\end{equation}
and for $n$ odd, we have
\begin{equation}\label{Ln-odd}
L_{a} (-n) = \sum_{k = (n+1)/2}^\infty \alpha(-k)\alpha(-n+k)  + a (n-1) \alpha(-n) .
\end{equation}

For this paper we only need the fact that $M_a(1)$ is simple and has infinitely many inequivalent irreducible modules, which can be easily classified (see \cite{LL}). (Indecomposable modules for $M_a(1)$ were classified by Milas \cite{M}.)  In particular,  for every highest weight irreducible $M_a(1)$-module $W$ there exists $\lambda \in \C$ such that
\begin{equation}\label{irreducibles}
W \cong M_a(1) \otimes_{\mathbb{C}} \mathbb{C}_\lambda,
\end{equation} 
where $\mathbb{C}_{\lambda}$  is the one-dimensional $\mathfrak{h}$-module such that $\alpha(0)$  acts as multiplication by $\lambda$. We denote these irreducible $M_a(1)$-modules by  $M_a(1, \lambda)= M_a(1)\otimes_{\mathbb{C}}  \mathbb{C}_{\lambda}$, so that $M_a(1, 0) = M_a(1)$.  If we let $v_\lambda \in \mathbb{C}_\lambda$ such that $\mathbb{C}_{\lambda} = \C v_{\lambda}$, then for instance,
\begin{eqnarray}
L_{a} (-1) v_\lambda &=& \alpha(-1)\alpha(0) v_{\lambda} \ = \ \lambda \alpha(-1)v_{\lambda}, \label {L(-1)-on-v-lambda}\\
L_{a} (0) v_\lambda &=& \left( \frac{1}{2}\alpha(0)^2  - a\alpha(0) \right) v_\lambda \ = \ \left( \frac{\lambda^2}{2} - a\lambda \right) v_{\lambda}. \label {L(0)-on-v-lambda}
\end{eqnarray}

Note that $M_a(1, \lambda)$ is admissible with 
\[M_a(1, \lambda) = \coprod_{k \in \mathbb{N}} M_a(1, \lambda)_k \]
and $M_a(1, \lambda)_0 \neq 0$ where $M_a(1, \lambda)_k$ is the eigenspace of eigenvectors of weight $k + \frac{\lambda^2}{2} - a\lambda$.  And since $M_a(1)$ is simple these $M_a(1, \lambda)$ are faithful $M_a(1)$-modules.  

\begin{rem}\label{a-not-mattering-remark}
{\em Note that for $v \in M_a(1)$, we have 
\begin{equation}\label{L(0)-a-condition} 
L_a(0) v = L_0(0) v
\end{equation}
and for $w \in M_a(1) \otimes_\mathbb{C}  \mathbb{C}_\lambda$, 
\begin{equation}\label{L(-1)-a-condition}
L_a(-1) w = L_0(-1)w .
\end{equation}
Thus in calculating $A_n(M_a(1))$, the choice of conformal vector, i.e. the choice of $a \in\mathbb{C}$ for $\omega_a$, does not change the structure of $A_n(M_a(1))$.   }
\end{rem}

The following is proved in \cite{FZ}:

\begin{thm}\label{A_0-theorem} \cite{FZ} 
As algebras 
\[A_0(M_a(1)) \cong \mathbb{C}[x,y]/(x^2-y) \cong \mathbb{C}[x]\] 
under the indentifications 
\[\alpha(-1)\mathbf{1} + O_0(M_a(1)) \longleftrightarrow x, \qquad \mbox{and} \qquad  \alpha(-1)^2\mathbf{1} + O_0(M_a(1)) \longleftrightarrow y.\]  
In addition, there is a bijection between isomorphism classes of irreducible admissible $M_a(1)$-modules and irreducible $\mathbb{C}[x]$-modules given by $M_a(1,\lambda) \longleftrightarrow \mathbb{C}[x]/(x - \lambda)$. 
\end{thm}
And from \cite{BVY-Heisenberg}, we have
\begin{thm}\label{A_1-theorem}\cite{BVY-Heisenberg}
Let $I$ be the ideal generated by the polynomial $(x^2-y)(x^2-y+2)$ in $\mathbb{C}[x, y]$.  
Then we have the following isomorphisms of algebras 
\begin{eqnarray}\label{first-characterization-level-one}
A_1(M_a(1)) &\cong& \mathbb{C}[x, y]/I \ \cong \ \mathbb{C}[x,y]/(x^2 - y) \oplus \mathbb{C}[x,y]/(x^2 - y + 2) \\
& \cong & A_0(M_a(1)) \oplus \mathbb{C}[x,y]/(x^2 - y+ 2)  \ \cong \  \mathbb{C}[x] \oplus \mathbb{C} [x] \label{second-characterization}
\end{eqnarray}
under the identifications
\[
\alpha(-1){\bf 1} + O_1(M_a(1)) \longleftrightarrow x + I, \quad \qquad  \alpha(-1)^2{\bf 1} + O_1(M_a(1)) \longleftrightarrow y + I.
\]
\end{thm}

\section{Generators and relations for the algebras $A_n(M_a(1))$ from \cite{AB-general-n}}\label{A_n-construction-section}

In this section we give some general results about the generators for the level $n$ Zhu algebra for the Heisenberg vertex operator algebra $M_a(1)$ that follow as a special case of the results from \cite{AB-general-n}.  Note that these results use only the internal structure of the vertex operator algebra itself, and no knowledge of the nature of the vertex operator algebra modules.  

Recall the notation (\ref{define-O-subsets}) and (\ref{define-O}).  To distinguish what relation we are using in our results below, we use the notation $u \approx v$ if $u \equiv v \, \mathrm{mod} \,O^L(V)$.  And we will write $u \sim_n v $ if  $u \equiv v  \, \mathrm{mod} \, O_n^\circ(V)$.  More broadly if $u$ is equivalent to $w$ modulo $O_n(V)$, we will write $u \equiv_n w$.  

We also use the following notation from \cite{AB-general-n}: For $v \in M_a(1)$ and $r \in \mathbb{N}$, let $F_r(v)$ denote the subspace of $M_a(1)$ spanned by elements of the form 
\[ \alpha(-k_1) \cdots \alpha(-k_m)v, \quad \mbox{for $k_1, \dots, k_m \in \mathbb{Z}_+$ with $m\leq r$.} \]

We have the following Corollary to Lemma 3.1, Proposition 3.2, and Lemma 3.3 in \cite{AB-general-n}.

\begin{cor}\label{recursion-corollary-n} \cite{AB-general-n}
For  $r \in \mathbb{N}$, and $ k_1, \dots, k_r \in \mathbb{Z}$, we have 
\begin{multline}\label{L-reduction}
\left(\sum_{j = 1}^r -k_j \right)\alpha(k_1) \alpha(k_2) \cdots \alpha(k_r) \mathbf{1} \\
\approx
 \sum_{j = 1}^r  k_j  \alpha(k_1) \alpha(k_2)\cdots \alpha(k_{j-1}) \alpha( k_j - 1)  \alpha(k_{j+1}) \cdots \alpha(k_r) \mathbf{1} .
\end{multline}
In particular, for $u,v \in M_a(1)$ and $j,k \in \mathbb{Z}$, $j \neq 0$, we have
\begin{equation}\label{L-reduction-special-case} 
u_{-j-1} v_{-k} {\bf 1} \approx - \left( \frac{\mathrm{wt} \, u  + \mathrm{wt} \, v  +j  +k  -2}{j} \right) u_{-j} v_{-k} {\bf 1}  - \frac{k}{j} u_{-j} v_{-k-1} {\bf 1} .
\end{equation}

Given $n \in \mathbb{N}$, for all $m\ge n+1$, and $v\in M_a(1)$,
\begin{equation}\label{recursion-equation}
\alpha(-m)v\sim_n (-1)^{m+1}\sum_{j=1}^{n+1}\binom{m-n-1}{j-1}\binom{m-n-j-1}{n + 1-j}\alpha(-n-j)v.
\end{equation}

Letting $Y^+(u,x) \in (\mathrm{End}(V))[[x]]$ and $Y^-(u,x)\in x^{-1}(\mathrm{End}(M_a(1)))[[x^{-1}]]$ denote the regular and singular parts, respectively, of $Y(u,x)$ for $u \in M_a(1)$, we have 
\begin{multline}\label{Y+}
Y^+(\alpha(-1) \mathbf{1} ,x)v \sim_n 
 \sum_{k=1}^{n}  \alpha(-k) v x^{k-1} \\
 +   \sum_{m\ge n+1} \sum_{k = n  + 1}^{2n + 1} (-1)^{m- 1}  \binom{m-n-1}{k-n  -1} \binom{m-k-1}{2n-k + 1 } \alpha(-k) v x^{m-1}  .
\end{multline}
\end{cor}

Note that from (\ref{generators-for-V}), we see that $M_a(1)$ is strongly generated by $\alpha(-1) {\bf 1}$, i.e. $M_a(1) = \langle \alpha (-1) {\bf 1} \rangle^1$, in the terminology and notation of \cite{AB-general-n}, and satisfies the permutation property from \cite{AB-general-n}.  Thus Proposition 3.6, Lemma 3.7, Theorem 3.8, Lemma 3.9, Corollary 3.13, Proposition 3.14, and Theorem 3.15 from \cite{AB-general-n} hold in this setting giving the following results:

\begin{cor}\label{reducing-general-generators-cor-heisenberg}\cite{AB-general-n}
For $n \in \mathbb{Z}_+$, $A_n(M_a(1))$ is generated by elements of the form
\[ \alpha(-2n)^{i_{2n}} \alpha(-2 n + 1)^{i_{2n -1}} \cdots \alpha( - 1)^{i_1 }\mathbf{1} + O_n(M_a(1)), \quad \mbox{and} \quad  \alpha(-1)^j {\bf 1} + O_n(M_a(1)),\]
for $i_1, i_2, \dots, i_{2n} \in \mathbb{N}$ with $0 \leq i_1 \leq n$, and $j \in \mathbb{Z}_+$.

Furthermore, for $v \in M_a(1)$ and $k \in \mathbb{Z}_+$, we have the following relations in $A_n(M_a(1))$:
\begin{equation}\label{-2n-1-reduction}
\alpha(-2n-1) v_{-k} {\bf 1} \approx -\left( \frac{ \mathrm{wt} \, v  + 2n + k - 1}{2n} \right) \alpha(-2n) v_{-k} {\bf 1}  - \frac{k}{2n} \alpha (-2n) v_{-k-1} {\bf 1}, 
\end{equation} 
and in particular
\begin{equation}\label{-2n-1-reduction-Heisenberg}
\alpha(-2n-1) \mathbf{1} \approx - \alpha(-2n) \mathbf{1}. 
\end{equation}
For $i_1, \dots, i_{2n} \in \mathbb{N}$ and $i_{2n+1}, k \in \mathbb{Z}_+$ 
\begin{equation}\label{reduce-equation-Heisenberg}
 \alpha(-2n-1)^{i_{2n + 1}} \alpha(-2n)^{i_{2n}} \alpha(-2n + 1)^{i_{2n-1}} \cdots \alpha(-1)^{i_1} v_{-k} {\bf 1}  \hspace{2in}
 \end{equation}
 \begin{eqnarray*}
&\equiv_n& \! \!  -\frac{1}{2n( i_{2n} + 1) } \Biggl(  \left(\mathrm{wt} \,  v_{-k} {\bf 1}  - 1 +   \sum_{j = 1}^{2n+1} j i_j  \right) \alpha(-2n - 1)^{i_{2n+1}-1} \alpha(-2 n)^{i_{2n}+1} \nonumber \\
& & \quad \alpha(-2n+1)^{i_{2n-1}}   \cdots \alpha( - 1)^{i_1 }  v_{-k} {\bf 1}  \nonumber \\
& & \quad - \, (2n+1) (i_{2n+1} -1)\left(  \sum_{j = 1}^{n+ 1} \binom{n + 1}{j-1}  \alpha( - n - j) \right) \alpha(-2n-1)^{i_{2n+1} -2} \alpha(-2n)^{i_{2n}+1} \nonumber \\
& & \quad  \alpha(-2n+1)^{i_{2n-1}}  \cdots \alpha(-1)^{i_1}  v_{-k} {\bf 1}  \nonumber \\
& & \quad  + \, (2n-1) i_{2n-1} \alpha(-2n-1)^{i_{2n + 1}-1} \alpha(-2n)^{i_{2n} + 2} \alpha(-2n + 1)^{i_{2n-1} - 1} \alpha(-2n + 2)^{i_{2n-2}} \nonumber \\
& & \quad \cdots \alpha(-1)^{i_1}  v_{-k} {\bf 1}  + \cdots  \nonumber  \\
& & \quad \cdots + 2 i_2 \alpha(-2n - 1)^{i_{2n+1}-1} \alpha(-2 n)^{i_{2n}+1} \alpha(-2n+1)^{i_{2n-1}}  \cdots \alpha(-4)^{i_4}  \alpha(-3)^{i_3 + 1} \nonumber\\
& & \quad \alpha(-2)^{i_2 - 1} \alpha( - 1)^{i_1 }  v_{-k} {\bf 1}  \nonumber \\
& & \quad + \, i_1 \alpha(-2n - 1)^{i_{2n+1}-1} \alpha(-2 n)^{i_{2n}+1} \alpha(-2n+1)^{i_{2n-1}} \cdots \alpha(-3)^{i_3} \alpha(-2)^{i_2 + 1} \alpha( - 1)^{i_1-1 }  v_{-k} {\bf 1}  \nonumber \\
&  & \quad + \,  k  \alpha(-2n - 1)^{i_{2n+1}-1} \alpha(-2 n)^{i_{2n}+1} \alpha(-2n+1)^{i_{2n-1}} \cdots  \alpha( - 1)^{i_1 } v_{-k-1} {\bf 1}  \Biggr)\nonumber
\end{eqnarray*}
where if $n=1$ this is to be interpreted as
\begin{eqnarray}\label{reduce-equation-n=1}
\lefteqn{ \ \ \ \alpha(-3)^{i_{3}} \alpha(-2)^{i_{2}} \alpha(-1)^{i_1}  v_{-k} {\bf 1}  }\\
&\equiv_1 & \! \! -\frac{1}{2( i_{2} + 1) } \Biggl( \biggl(  \mathrm{wt} \, v_{-k} {\bf 1}   - 1 +  \sum_{j = 1}^{3} i_j  j   \biggr) \alpha(-3)^{i_{3}-1} \alpha(-2)^{i_{2}+1} \alpha( - 1)^{i_1 }  v_{-k} {\bf 1}  \nonumber \\
& &  - 3 (i_{3} -1)\biggl(  \sum_{j = 1}^{2} \! \binom{2}{j-1}  \alpha(- j - 1) \biggr) \alpha(-3)^{i_{3} -2} \alpha(-2)^{i_{2}+1} \alpha(-1)^{i_1}  v_{-k} {\bf 1}  \nonumber \\
& &  + \, i_1 \alpha(-3)^{i_{3}-1} \alpha(-2)^{i_{2}+2} \alpha( - 1)^{i_1-1 } v_{-k} {\bf 1}    +  k \alpha(-3)^{i_{3}- 1} \alpha(-2)^{i_{2}+1} \alpha( - 1)^{i_1 }  v_{-k-1} {\bf 1}   \Biggr).\nonumber
\end{eqnarray}

Let $t\in \mathbb{Z}_+$, and $i_1,\dots, i_t \in \mathbb{N}$.  For $1\leq s \leq t$, set $p_s=\sum_{j=1}^{s}i_j$, $p_0 = 0$,  and $r=\sum_{j=1}^{t} j  i_j$.  We have the following formula for $\alpha(-t)^{i_{t}}\dots \alpha(-1)^{i_{1}}{\bf{1}}*_nv$ and $ v\in M_a(1)$
\begin{multline}\label{multiplication-formula-infinity} 
\alpha(-t)^{i_{t}}\cdots \alpha(-1)^{i_{1}}{\bf{1}}*_nv\\
= \sum_{m=0}^n \sum_{j=-m}^n (-1)^m\binom{m+n}{n}\binom{n+r}{j+m}\displaystyle\sum_{\substack{k_1,\dots ,k_{p_t}\in \mathbb{N}\\k_1+\cdots +k_{p_t}=n-j}} \! \! {}_{\circ}^{\circ}P_{i_1,\dots, i_{t}}(k_1,\dots, k_{p_t}) {}_{\circ}^{\circ} v +g_{i_1,\dots, i_{t}}(v), 
\end{multline}
where 
\begin{equation}\label{define-P}
P_{i_1,\dots, i_{t}}(k_1,\dots, k_{p_t})=\prod_{s=1}^{t}\prod_{l=1}^{i_{s}}\binom{k_{l+p_{s-1}}+s-1}{s-1} \alpha(-k_{l+p_{s-1}}-s),
\end{equation}
and $g_{i_1,\dots, i_{t}}(v)$ is a linear combination of elements of the form ${}_{\circ}^{\circ} \alpha(q_1) \cdots \alpha(q_{p_t}){}_{\circ}^{\circ}v$ for some $q_1, \dots, q_{p_t} \in \mathbb{Z}$ with at least one of the  $q$'s nonnegative.  In particular,  for $i, t \in \mathbb{Z}_+$, and $v \in M_a(1)$
\begin{multline}\label{multiplication-cor-formula}
\alpha(-t)^i \mathbf{1} *_n v =  \sum_{m=0}^n \sum_{j =0}^{m+n}  \sum_{\stackrel{p_1, p_2, \dots, p_i \in \mathbb{N}}{p_1 + p_2 + \cdots + p_i =  j} } (-1)^m \binom{m+n}{n}  \binom{n + it }{m + n - j} \\\left( \prod_{l=1}^i \binom{p_l + t - 1}{t - 1} \alpha(-p_l-t) \right)  v +  g_{i}(v) , 
\end{multline}
where  $g_{i}(v)$ is a linear combination of elements of the form  ${}_{\circ}^{\circ}\alpha(q_1)\cdots \alpha(q_{i}) {}_{\circ}^{\circ}v$ where at least one of the  $q$'s is nonnegative.
That is, modulo $g_{i}(v)$, the product $\alpha(-t)^i \mathbf{1} *_n v$ involves a linear combination of terms of the form $\alpha(-p_1-t) \alpha(-p_2-t)  \cdots \alpha(-p_i-t)  \, v$ where $\mathbf{p} = (p_1, p_2, \dots, p_i)$ is an ordered partition of $j$ into $i$ nonnegative integer parts for $0 \leq j \leq m+ n \leq 2n$. 

Furthermore, when $i = 1$ and $t = 1$, then $g_1(v) = 0$.  In addition,  if $n \geq t  + 1$, then we have 
\[g_1(v) =   \sum_{m=0}^{t - 1} \sum_{j=-1}^{m -t}(-1)^m\binom{m+n}{n}\binom{n+ t }{m + n - j }\binom{j  + t-1}{t-1} \alpha(-j  - t) v , \]
and more generally
\[ g_1(v) =  \sum_{m=0}^n\sum_{j=-1}^{-\infty}(-1)^m\binom{m+n}{n}\binom{n+ t}{m + n - j }\binom{j  + t-1}{t-1} \alpha(-j  - t)  v.\]

And finally, we have that for $i_1, i_2, \dots, i_{2n} \in \mathbb{N}$, 
\begin{multline}\label{reduce1-prop}
\binom{i_1 + n}{n} \sum_{m=0}^n  (-1)^m  \binom{i_1 }{m}  \alpha(-2n)^{i_{2n}} \alpha(-2 n + 1)^{i_{2n -1}} \cdots \alpha( - 2)^{i_2 } \alpha(-1)^{i_1}  \mathbf{1} \\
\equiv_n \alpha(-1)^{i_1} \mathbf{1} *_n \alpha(-2n)^{i_{2n}} \alpha(-2 n + 1)^{i_{2n -1}} \cdots \alpha( - 2)^{i_2 }\mathbf{1} + w',
\end{multline}
where $w'$ is a linear combination of terms of the form $\alpha(-2n)^{i_{2n}'} \cdots \alpha(-1)^{i_1'} \mathbf{1}$ where either $ i_1' + \cdots + i_{2n}' < i_1 + \cdots + i_{2n}$, or $ i_1' + \cdots + i_{2n}' = i_1 + \cdots + i_{2n}$ and  $i_1'< i_1$.    

Furthermore, the coefficient of the term on the lefthand side of Eqn.\ (\ref{reduce1-prop}) is zero if and only if  $0 < i_1  \leq n$, and is equal to 
\[ (-1)^n  \binom{i_1 + n}{n} \binom{i_1 - 1}{n} \]
if $i_1 >n$.
\end{cor}

\section{Generators and relations for the algebra  $A_2(M_a(1))$}

We first give the results from \cite{AB-general-n} summarized in Section \ref{A_n-construction-section} for $V = M_a(1)$ and $n \in \mathbb{N}$ applied to the setting of $V = M_a(1)$ and $n = 2$ to determine an initial generating set and relations for $A_2(M_a(1))$.  Then we use these relations and the definition of the higher level Zhu algebra to build further relations and to further reduce the set of generators for $A_2(M_a(1))$. This culminates in Theorem \ref{main-generators-thm}  where we give the minimal set of generators for $A_2(M_a(1))$ to be used in the subsequent section, Section \ref{relations-section} to determine the structure of $A_2(M_a(1))$.  

\subsection{Results from \cite{AB-general-n}}\label{AB-Corollaries-2-section}

It is immediate from Corollaries \ref{recursion-corollary-n} and 
\ref{reducing-general-generators-cor-heisenberg} applied to $n = 2$ that the following hold:

\begin{cor}\label{level-two-corollary} 
$A_2(M_a(1))$ is generated by elements of the form
\begin{equation}\label{generators-level-2}
 \alpha(-4)^{i_4} \alpha(-3)^{i_3} \alpha(-2)^{i_2}  \alpha( - 1)^{i_1 }\mathbf{1} + O_2(M_a(1)), \quad \mbox{and} \quad  \alpha(-1)^j {\bf 1} + O_2(M_a(1)),
 \end{equation}
for $i_1,  \dots, i_{4} \in \mathbb{N}$ with $0 \leq i_1 \leq 2$, and $j \in \mathbb{Z}_+$.

Furthermore, for $v \in M_a(1)$, we have the following relations in $A_2(M_a(1))$:

For all $m\geq 3$ 
\begin{eqnarray}\label{recursion-level-2}
\alpha(-m)v &\sim_2& (-1)^{m+1}\sum_{j=1}^{3}\binom{m-3}{j-1}\binom{m-j-3}{3-j}\alpha(-j-2)v\\
&\sim_2& (-1)^{m+1}\biggl(\frac{(m-4)(m-5)}{2} \alpha(-3) + (m-3)(m-5) \alpha(-4) \nonumber \\
& & \quad + \, \frac{(m-3)(m-4)}{2} \alpha(-5) \biggr) v. \nonumber
\end{eqnarray}

\begin{equation}\label{-5-reduction}
\alpha(-5) v \approx -\left( \frac{ 4  + \mathrm{wt} \, v}{4} \right) \alpha(-4) v - \frac{1}{4} \alpha (-4) v_{-2} {\bf 1}, 
\end{equation} 
and in particular
\begin{equation}\label{-5-reduction-Heisenberg}
\alpha(-5) \mathbf{1} \approx - \alpha(-4) \mathbf{1}. 
\end{equation}
For $i_1, \dots, i_{4} \in \mathbb{N}$ and $i_{5} \in \mathbb{Z}_+$ 
\begin{eqnarray}\label{reduce-equation-Heisenberg-level2}
\lefteqn{\ \ \ \ \ \ \ \ \ \alpha(-5)^{i_5} \alpha(-4)^{i_4} \alpha(-3)^{i_3} \alpha(-2)^{i_2} \alpha(-1)^{i_1} v  }\\
&\equiv_2& -\frac{1}{4( i_4+ 1) } \Biggl(  \left(\mathrm{wt} \, v - 1 +   \sum_{j = 1}^{5} j i_j  \right) \alpha(-5)^{i_5-1} \alpha(-4)^{i_4+1} \alpha(-3)^{i_3} \alpha(-2)^{i_2}  \alpha( - 1)^{i_1 } v \nonumber \\
& & \quad - \, 5 (i_5 -1)\left(  \sum_{j = 1}^{3} \binom{3}{j-1}  \alpha( - 2 - j) \right) \alpha(-5)^{i_5 -2} \alpha(-4)^{i_4+1}  \alpha(-3)^{i_3}  \alpha(-2)^{i_2}  \alpha(-1)^{i_1} v \nonumber \\
& & \quad  + \, 3 i_3 \alpha(-5)^{i_5-1} \alpha(-4)^{i_4 + 2} \alpha(-3)^{i_3- 1} \alpha(-2)^{i_2}  \alpha(-1)^{i_1} v   \nonumber  \\
& & \quad  + \, 2 i_2 \alpha(-5)^{i_5-1} \alpha(-4)^{i_4+1} \alpha(-3)^{i_3+1}  \alpha(-2)^{i_2 - 1} \alpha( - 1)^{i_1 } v \nonumber \\
& & \quad + \, i_1 \alpha(-5)^{i_5-1} \alpha(-4)^{i_4+1} \alpha(-3)^{i_3} \alpha(-2)^{i_2 + 1} \alpha( - 1)^{i_1-1 } v \nonumber \\
&  & \quad + \,  \alpha(-5)^{i_5-1} \alpha(-4)^{i_4+1} \alpha(-3)^{i_3} \alpha(-2)^{i_2}  \alpha( - 1)^{i_1 } v_{-2} {\bf 1}  \Biggr) .\nonumber
\end{eqnarray}

Let $t\in \mathbb{Z}_+$, and $i_1,\dots, i_t \in \mathbb{N}$.  For $1\leq s \leq t$, set $p_s=\sum_{j=1}^{s}i_j$, $p_0 = 0$,  and $r=\sum_{j=1}^{t} j  i_j$.  We have the following formula for $\alpha(-t)^{i_{t}}\dots \alpha(-1)^{i_{1}}{\bf{1}}*_2 v$, where $ v\in M_a(1)$
\begin{multline}\label{multiplication-formula-infinity-level2} 
\alpha(-t)^{i_{t}}\cdots \alpha(-1)^{i_{1}}{\bf{1}}*_2 v\\
= \sum_{m=0}^2 \sum_{j=-m}^2 (-1)^m\binom{m+2}{2}\binom{r+2}{j+m}\displaystyle\sum_{\substack{k_1,\dots ,k_{p_t}\in \mathbb{N}\\k_1+\cdots +k_{p_t}=2-j}} \! \! {}_{\circ}^{\circ}P_{i_1,\dots, i_{t}}(k_1,\dots, k_{p_t}) {}_{\circ}^{\circ} v +g_{i_1,\dots, i_{t}}(v), 
\end{multline}
where $P_{i_1,\dots, i_{t}}(k_1,\dots, k_{p_t})$ is defined by (\ref{define-P}),
and $g_{i_1,\dots, i_{t}}(v)$ is a linear combination of elements of the form  ${}_{\circ}^{\circ} \alpha(q_1) \cdots \alpha(q_{p_t}){}_{\circ}^{\circ}v$ for $q_1, \dots, q_{p_t} \in \mathbb{Z}$ with at least one of the  $q$'s nonnegative.   In particular,  for $i, t \in \mathbb{Z}_+$,  and $v \in M_a(1)$
\begin{multline}\label{multiplication-cor-formula-level2}
\alpha(-t)^i \mathbf{1} *_2 v =  \sum_{m=0}^2 \sum_{j =0}^{m+2}  \sum_{\stackrel{p_1, p_2, \dots, p_i \in \mathbb{N}}{p_1 + p_2 + \cdots + p_i =  j} } (-1)^m \binom{m+2}{2}  \binom{2 + it }{m + 2 - j} \\\left( \prod_{l=1}^i \binom{p_l + t - 1}{t - 1} \alpha(-p_l-t) \right)  v + g_{i}(v) , 
\end{multline}
where  $g_{i}(v)$ is a linear combination of elements of the form  ${}_{\circ}^{\circ}\alpha(q_1)\cdots \alpha(q_{i}) {}_{\circ}^{\circ}v$ where at least one of the  $q$'s is nonnegative.
That is, modulo $g_{i}(v)$, the product $\alpha(-t)^i \mathbf{1} *_n v$ involves a linear combination of terms of the form $\alpha(-p_1-t) \alpha(-p_2-t)  \cdots \alpha(-p_i-t)  \, v$ where $\mathbf{p} = (p_1, p_2, \dots, p_i)$ is an ordered partition of $j$ into $i$ nonnegative integer parts for $0 \leq j \leq m+ n \leq 4$. 

Furthermore, when $i = 1$ and $t = 1$, then $g_1(v) = 0$, and more generally if $i = 1$, and $t>1$, we have 
\begin{equation}\label{g-formula}
 g_1(v) = \sum_{j=-1}^{-\infty} \left(  \binom{2+ t}{ 2 - j }   - 3 \binom{2+ t}{3 - j } +  6 \binom{2+ t}{4 - j } \right)\binom{j  + t-1}{t-1} \alpha(-j  - t)  v .
 \end{equation}

And for $i_1, i_2, i_3, i_{4} \in \mathbb{N}$, we have
\begin{multline}\label{reduce1-prop-level2}
\sum_{m=0}^2  (-1)^m \binom{m+2}{2}  \binom{ i_1 + 2 }{m + 2 }  \alpha(-4)^{i_{4}} \alpha(-3)^{i_{3}}  \alpha( - 2)^{i_2 } \alpha(-1)^{i_1}  \mathbf{1} \\
\equiv_2 \alpha(-1)^{i_1} \mathbf{1} *_2 \alpha(-4)^{i_{4}} \alpha(-3)^{i_{3}}  \alpha( - 2)^{i_2 }\mathbf{1} + w',
\end{multline}
where $w'$ is a linear combination of terms of the form $\alpha(-4)^{i_{4}'} \cdots \alpha(-1)^{i_1'} \mathbf{1}$ where either $ i_1' + \cdots + i_{4}' < i_1 + \cdots + i_{4}$, or $ i_1' + \cdots + i_{4}' = i_1 + \cdots + i_{4}$ and  $i_1'< i_1$.    

Furthermore, the coefficient of the term on the lefthand side of Eqn.\ (\ref{reduce1-prop-level2}) is zero if and only if $0 < i_1  \leq 2$ and is equal to 
\[\frac{ (i_1^2 - 1) (i_1^2 - 4) }{4} \]
if $i_1 > 2$. 

Finally, letting $Y^+(u,x) \in (\mathrm{End}(V))[[x]]$ denote the regular part of $Y(u,x)$ for $u \in M_a(1)$, we have 
\begin{multline}\label{Y+-2}
Y^+(\alpha(-1) \mathbf{1} ,x) v 
%\sim_2 \alpha(-1) v  + \alpha(-2) v x   +   \sum_{m\ge 3} \sum_{k = 3}^{5} (-1)^{m- 1}  \binom{m-3}{k-3} \binom{m-k-1}{5-k }\alpha(-k) v x^{m-1} \\
 \sim_2 \alpha(-1) v  + \alpha(-2) v x  +   \sum_{m\ge 3} (-1)^{m- 1}\left(   \binom{m-4}{2 } \alpha(-3)   
  \right. \\
 \left. +   (m-3) (m-5)    \alpha(-4)   +  \binom{m-3}{2}  \alpha(-5)  \right) v x^{m-1}  .
\end{multline}
\end{cor}

Next we give some simplified formulas that follow from Corollary \ref{level-two-corollary} explicitly as these will be used later to reduce generators and determine relations for $A_2(M_a(1))$.  

\begin{cor}  
For $i, t \in \mathbb{Z}_+$, and $v \in F_r({\bf 1}) \subset M_a(1)$ 
\begin{multline}\label{multiplication-t-cor-formula-level2-expllicit}
\alpha(-t)^i \mathbf{1} *_2 v
=  \sum_{j =0}^{4}  \sum_{\stackrel{p_1, p_2, \dots, p_i \in \mathbb{N}}{p_1 + p_2 + \cdots + p_i =  j} }   \left( \binom{it+ 2}{ 2 - j}  - 3  \binom{it+ 2 }{3 - j} + 6  \binom{it + 2 }{4 - j}   \right) \\
\left( \prod_{l=1}^i \binom{p_l + t- 1}{t-1} \alpha(-p_l-t) \right)  v +   f_{i + r-2}({\bf 1}) , 
\end{multline}
for some $f_{i + r-2}({\bf 1}) \in F_{i + r-2}({\bf 1})$.  In particular,
\begin{multline}\label{multiplication-one-cor-formula-level2-expllicit}
\alpha(-1)^i \mathbf{1} *_2 v
=  \sum_{j =0}^{4}  \sum_{\stackrel{p_1, p_2, \dots, p_i \in \mathbb{N}}{p_1 + p_2 + \cdots + p_i =  j} }   \left( \binom{i+ 2}{ 2 - j}  - 3  \binom{i+ 2 }{3 - j} + 6  \binom{i+ 2 }{4 - j}   \right) \alpha(-p_1-1)  \\
\alpha(-p_2-1)  \cdots \alpha(-p_i-1)  \, v   +   f_{i + r-2}({\bf 1}) , 
\end{multline}
for some $f_{i + r-2}({\bf 1}) \in F_{i + r-2}({\bf 1})$, and 
\begin{equation}\label{multiplication-one-one}
\alpha(-1)   \mathbf{1} *_2 v 
%= \left( \binom{3}{ 2}  - 3  \binom{3 }{3} + 6  \binom{3 }{4}   \right) \alpha(-1)  v    +   \left( \binom{3}{ 1}  - 3  \binom{3 }{2} + 6  \binom{3 }{3}   \right) \alpha(-2)  v  +     \left( \binom{3}{ 0}  - 3  \binom{3 }{1} + 6  \binom{3 }{2}   \right) \alpha(-3)  v  +    \left( \binom{3}{ -1}  - 3  \binom{3 }{0} + 6  \binom{3 }{1}   \right) \alpha(-4)  v + \left( \binom{3}{ -2}  - 3  \binom{3 }{-1} + 6  \binom{3 }{0}   \right) \alpha(-5)  v + f_{0}(v) , 
%=\\
=  10 \alpha(-3)  v  + 15 \alpha(-4)  v + 6 \alpha(-5)  v,
\end{equation}
and 
\begin{multline}\label{multiplication-two-cor-formula-level2-explicit}
\alpha(-2)^i \mathbf{1} *_2 v =  \sum_{j =0}^{4}  \sum_{\stackrel{p_1, p_2, \dots, p_i \in \mathbb{N}}{p_1 + p_2 + \cdots + p_i =  j} } 
\left( \binom{2i+ 2}{ 2 - j}  - 3  \binom{2i+ 2 }{3 - j} + 6  \binom{2i + 2 }{4 - j}   \right)
 \\
\left( \prod_{l=1}^i (p_l + 1)\alpha(-p_l-2) \right)  v +   f_{i + r -2}({\bf 1}), 
\end{multline}
for some $f_{i+ r-2}({\bf 1}) \in F_{i + r-2}({\bf 1})$, and in the case $i = 1$, we have $f_{i + r - 2}({\bf 1}) = f_{r-1}(v) = 0$ in Eqn.\ (\ref{multiplication-two-cor-formula-level2-explicit}).   
\end{cor}

\begin{proof}  These formulas all follow from Eqn.\ \eqref{multiplication-cor-formula-level2} by first noting that $\binom{p}{q} = 0$ if $p\geq 0$ and $q>p$, and thus for any expression $X$,
\[ \sum_{j = 0}^2 \binom{ it + 2}{2-j}  X = \sum_{j = 0}^4 \binom{ it + 2 }{2-j} X \quad \mbox{and} \quad 
 \sum_{j = 0}^3 \binom{ it + 2}{3-j}  X = \sum_{j = 0}^4 \binom{it + 2}{3-j}X. \]
Then the fact that the $g_i(v)$ terms in Eqn.\ (\ref{multiplication-cor-formula-level2}) are in $f_{i + r - 2}({\bf 1})$ follow from the fact that $g_i(v)$ has at least one $\alpha(q)$ with $q\geq 0$ and the bracket relations for the Heisenberg algebra, so that the nonnegative mode acting on any positive mode in $v$ will decrease the number of modes by 2 or will be zero.
The last statement follows from the fact that $g_1(v)$ in Eqn.\ (\ref{g-formula}) reduces to $-\alpha(0)v = 0$.
\end{proof}

\subsection{Further reduction of generators arising from the \cite{AB-general-n} Corollaries}

We have been ordering commuting modes from smallest to largest weight, to emphasize that many of the results we have used thus far are special cases of results obtained in \cite{AB-general-n} where the modes are ordered in this way. From now on though, it will be useful to order commuting modes from largest to smallest weight and we will use this convention throughout the remainder of the paper.
In this section, applying the Corollaries of Section \ref{AB-Corollaries-2-section}, we further reduce the set of generators of $A_2(M_a(1))$.  

\begin{prop}\label{reducing-one-prop}
 $A_2(M_a(1))$ is generated by $v + O_2(M_a(1))$ for $v$ in 
\begin{equation}\label{generators-again}
 \{ \alpha(-1)^i \alpha(-2)^j \alpha(-3)^k \alpha(-4)^l \mathbf{1} \; | \; 0\leq i \leq 2, \ j, k, l \in \mathbb{N} \}, 
 \end{equation}
\end{prop}

\begin{proof} From Corollary \ref{level-two-corollary}, we have that 
$A_2(M_a(1))$ is generated by $v + O_2(M_a(1))$ for $v$ in 
\[
 \{ \alpha(-1)^i \alpha(-2)^j \alpha(-3)^k \alpha(-4)^l \mathbf{1}, \alpha(-1)^m \mathbf{1}  \; | \; 0\leq i \leq 2, \ j, k, l, m \in \mathbb{N} \}. 
 \]
Thus it remains to prove that $\alpha(-1)^i \mathbf{1}$ for $i >2$ is generated by elements in (\ref{generators-again}) mod $O_2(V)$.  We induct on the total degree of elements $\alpha(-1)^i \alpha(-2)^j \alpha(-3)^k \alpha(-4)^l \mathbf{1}$, given by $ i + j + k + l$.  

From (\ref{L-reduction})  with $k_1 = k_2 = \cdots = k_r = -1$, we have
\begin{equation}\label{reduce-alpha(-1)}
i\alpha(-1)^i \mathbf{1} \approx - \alpha(-1)^{i-1} \alpha(-2) \mathbf{1} .
\end{equation}
Thus $\alpha(-1)^3 \mathbf{1}$ is generated by elements in (\ref{generators-again}) mod $O_2(V)$, and therefore the result holds for all elements of degree 3 or less.  We make the inductive assumption that elements of total degree $i + j + k + l < r$ are generated by elements in  (\ref{generators-again}) mod $O_2(V)$.   Since all elements of total degree $r$ are generated by elements in (\ref{generators-level-2}) mod $O_2(V)$, it remains to prove the result for  $\alpha(-1)^r \mathbf{1}$.  

Using the multiplication formula (\ref{multiplication-one-cor-formula-level2-expllicit}) with $v = \alpha(-2)^j \mathbf{1}$, to observe that for $i \geq 4$ and $j \in \mathbb{N}$, we have
\begin{eqnarray*}
\lefteqn{\alpha(-1)^i \mathbf{1} *_2 \alpha(-2)^j \mathbf{1} }\\
&=&   \left( \binom{i+ 2}{ 2}  - 3  \binom{i+ 2 }{3} + 6  \binom{i+ 2 }{4}   \right) \alpha(-1)^i  
\alpha(-2)^j  \mathbf{1} \\
& &  + \left( \binom{i+ 2}{ 1}  - 3  \binom{i+ 2 }{2} + 6  \binom{i+ 2 }{3}   \right)  \! \!  \sum_{\stackrel{p_1, p_2, \dots, p_i \in \mathbb{N}}{p_1 + p_2 + \cdots + p_i =  1} }   \left( \prod_{l=1}^i \alpha(-p_l-1) \right)  \alpha(-2)^j \mathbf{1}  \\
& & +   \left( 1 - 3  \binom{i+ 2 }{1} + 6  \binom{i+ 2 }{2}   \right)  \! \! \sum_{\stackrel{p_1, p_2, \dots, p_i \in \mathbb{N}}{p_1 + p_2 + \cdots + p_i =  2} }  \left( \prod_{l=1}^i \alpha(-p_l-1) \right)   \alpha(-2)^j \mathbf{1} \\ 
& & +   \left(  - 3   + 6  \binom{i+ 2 }{1}   \right) \! \! \sum_{\stackrel{p_1, p_2, \dots, p_i \in \mathbb{N}}{p_1 + p_2 + \cdots + p_i =  3} }  \left( \prod_{l=1}^i \alpha(-p_l-1) \right)     \alpha(-2)^j \mathbf{1} \\
& & + 6 \! \! \sum_{\stackrel{p_1, p_2, \dots, p_i \in \mathbb{N}}{p_1 + p_2 + \cdots + p_i =  4} }  \left( \prod_{l=1}^i \alpha(-p_l-1) \right)    \alpha(-2)^j \mathbf{1}  
+  f_{i + j-2}(\mathbf{1}) \\
&=& \frac{(i^2-1) (i^2 - 4)}{4}   \alpha(-1)^i  
\alpha(-2)^j  \mathbf{1} + i (i + \frac{1}{2})(i-1)(i + 2) \alpha(-1)^{i-1} \alpha(-2)^{j+1} \mathbf{1}\\
& & + (3i^2 + 6i + 1) \left( i \alpha(-1)^{i - 1} \alpha(-2)^j \alpha(-3) \mathbf{1} + \binom{i}{2}\alpha(-1)^{i-2} \alpha(-2)^{j + 2} \mathbf{1}\right)  \\ 
& & + (6i + 9) \left( i \alpha(-1)^{i - 1} \alpha(-2)^j \alpha(-4) \mathbf{1} +  i (i - 1) \alpha(-1)^{i - 2} \alpha(-2)^{j + 1} \alpha(-3) \mathbf{1} \right.\\
& & \left. + \binom{i}{3}\alpha(-1)^{i - 3} \alpha(-2)^{j + 3} \mathbf{1} \right)  
 +  6  \biggl( i \alpha(-1)^{i-1} \alpha(-2)^j \alpha(-5) \mathbf{1}  \\
 & &  +  i ( i -1)  \alpha(-1)^{i-2} \alpha(-2)^{j+1} \alpha(-4) \mathbf{1} + \binom{i}{2} \alpha(-1)^{i - 2} \alpha(-2)^j \alpha(-3)^2 \mathbf{1} \\
& & \left. + \frac{i (i-1) (i-2) }{2} \alpha(-1)^{i-3} \alpha(-2)^{j+2} \alpha(-3) {\bf 1}  + \binom{i}{4} \alpha(-1)^{i - 4} \alpha(-2)^{j + 4} \mathbf{1}\right)  \\
& &  +  f_{i+j -2}( \mathbf{1}) ,  \\
&=&  \frac{(i^2-1) (i^2 -  4)}{4}   \alpha(-1)^i  
\alpha(-2)^j  \mathbf{1}  + w +  6i  \alpha(-1)^{i-1} \alpha(-2)^j \alpha(-5) \mathbf{1} 
+ f_{i+ j-2}( \mathbf{1}) ,
\end{eqnarray*}
for $w$ a linear combination of elements of the form $\alpha(-1)^{i'} \alpha(-2)^{j'} \alpha(-3)^{k'} \alpha(-4)^{l'} \mathbf{1}$ with $i' + j' + k' + l' = i + j$ and $i'<i$, and for some $f_{i+ j-2}(\mathbf{1}) \in F_{i+j-2}(\mathbf{1})$.

Using Eqn. (\ref{reduce-equation-Heisenberg-level2}), we have that 
\begin{multline*}
6\alpha(-1)^{i-1} \alpha(-2)^j \alpha(-5) \mathbf{1} \equiv_2 - \frac{3}{2} \left( (3 + i + 2j \alpha(-1)^{i-1} \alpha(-2)^j \alpha(-4) \mathbf{1} \right. \\
\left. + 2j   \alpha(-1)^{i-1} \alpha(-2)^{j-1} \alpha(-3) \alpha(-4) \mathbf{1} + (i-1) \alpha(-1)^{i-2} \alpha(-2)^{j+1} \alpha(-4) \mathbf{1} \right).
\end{multline*}
 Thus 
\begin{eqnarray}\label{reduce-alpha(-1)-2}
\alpha(-1)^i \alpha(-2)^j  \mathbf{1}  \equiv_2   \frac{4}{(i^2 - 1) (i^2 - 4)}  \alpha (-1)^i {\bf 1} *_2 \alpha(-2)^j \mathbf{1} + w' +  f_{i + j - 2} (\mathbf{1}) 
\end{eqnarray}
for $w'$ a linear combination of elements of the form $\alpha(-1)^{i'} \alpha(-2)^{j'} \alpha(-3)^{k'} \alpha(-4)^{l'} \mathbf{1}$ with $i' + j' + k' + l' = i + j$ and $i'<i$, and  $f_{i + j - 2} (\mathbf{1})  \in F_{i+j - 2}(\mathbf{1})$.  

Therefore by Eqn. (\ref{reduce-alpha(-1)}), we have that for $r > 3$,
\begin{eqnarray*}
\alpha(-1)^r \mathbf{1} &\approx& - \frac{1}{r} \alpha(-1)^{r-1} \alpha(-2) \mathbf{1} \\
&=&  - \frac{4}{(r-3)(r-2)r (r + 1)} \alpha (-1)^{r-1} {\bf 1}  *_2 \alpha(-2) \mathbf{1} + w' +  f'_{r - 2} (\mathbf{1}) ,
\end{eqnarray*}
for $w'$ a linear combination of elements of the form $\alpha(-1)^{i'} \alpha(-2)^{j'} \alpha(-3)^{k'} \alpha(-4)^{l'} \mathbf{1}$ with $i' + j' + k' + l' = r $ and $i'<r$ and $f_{r - 2} (\mathbf{1})  \in F_{r - 2}(\mathbf{1})$.  The result follows by the inductive assumption on $r$.  
\end{proof}

\begin{prop}\label{generators-prop}
$A_2(M_a(1))$ is generated by $v + O_2(M_a(1))$ for $v$ in 
\begin{equation}\label{generators-again2}
\{ \alpha(-1)^i \alpha(-2)^j \alpha(-3)^k \alpha(-4)^l \mathbf{1} \; | \; 0\leq i \leq 2, \ 0\leq j\leq 1,\  k, l \in \mathbb{N} \}.
 \end{equation}
\end{prop}

\begin{proof} 
From Proposition \ref{reducing-one-prop}, we need only prove that elements of the form 
\begin{equation}\label{form}
\alpha(-1)^i \alpha(-2)^j \alpha(-3)^k \alpha(-4)^l \mathbf{1}
\end{equation}
for $j>2$ are in (\ref{generators-again2}).  We again proceed by induction on the total degree $i + j + k + l = r$.  The result follows immediately for terms with total degree $r \leq 1$, and for elements of degree $i + j + k + l$ with $j \leq 1$.  Assume the result holds for elements of total degree $i + j + k + l <r$.  We need to show that elements of degree $r = i + j + k + 1$ with $j>1$ are of the form (\ref{generators-again2}) mod $O_2(V)$.

%We first note that if $j = 2$, then by Eqn. (\ref{multiplication-two-cor-formula-level2-explicit}), we have that for any $v \in M_a(1)$,
%\begin{eqnarray*}
%\alpha(-2)^2 \mathbf{1} *_2 v &=&  \sum_{j =0}^{4}  \sum_{\stackrel{p_1, p_2  \in \mathbb{N}}{p_1 + p_2 =  j} } 
%\left( \binom{6}{ 2 - j}  - 3  \binom{6 }{3 - j} + 6  \binom{6 }{4 - j}   \right)
% \\
%& &  (p_1 + 1)(p_2 + 1)\alpha(-p_1-2 ) \alpha(-p_2 - 2)   v + f_{1}(v) \\
%&=& 15\alpha(-2)^2v  + 162 \alpha(-2) \alpha(-3) v + 292 \alpha(-3)^2v + 219 \alpha(-2)\alpha(-4)v \\
%& & + 132 \alpha(-5)v + 198 \alpha(-3) \alpha(-4) v + 30 \alpha(-6)v + 48 \alpha(-3)\alpha(-5)v \\
%& & + 54 \alpha(-3)^2v + f_1(v)
%\end{eqnarray*}
From Eqn. (\ref{multiplication-two-cor-formula-level2-explicit}) we have that for $j\geq 2$ and $v \in F_r({\bf 1}) \subset M_a(1)$, 
\begin{eqnarray*}
\lefteqn{\left( \binom{2j+ 2}{ 2}  - 3  \binom{2j+ 2 }{3} + 6  \binom{2j + 2 }{4}   \right) \alpha(-2)^j v  }\\
&=& \alpha(-2)^j \mathbf{1} *_2 v\\
& &  -  
\left( \binom{2j+ 2}{ 1}  - 3  \binom{2j+ 2 }{2} + 6  \binom{2j + 2 }{3}   \right)
\! \!  \sum_{\stackrel{p_1, p_2, \dots, p_j \in \mathbb{N}}{p_1 + p_2 + \cdots + p_j =  1} }
\left( \prod_{l=1}^j (p_l + 1)\alpha(-p_l-2) \right)  v \\
& & -
\left( 1  - 3  \binom{2j+ 2 }{1} + 6  \binom{2j + 2 }{2}   \right)  \! \! \sum_{\stackrel{p_1, p_2, \dots, p_j \in \mathbb{N}}{p_1 + p_2 + \cdots + p_j =  2} } 
\left( \prod_{l=1}^j (p_l + 1)\alpha(-p_l-2) \right)  v\\
& & - \left(  - 3  + 6  \binom{2j + 2 }{1}   \right) \! \! 
 \sum_{\stackrel{p_1, p_2, \dots, p_j \in \mathbb{N}}{p_1 + p_2 + \cdots + p_j =  3} }  
\left( \prod_{l=1}^j (p_l + 1)\alpha(-p_l-2) \right)  v\\
& & - \  6  \! \! 
 \sum_{\stackrel{p_1, p_2, \dots, p_j \in \mathbb{N}}{p_1 + p_2 + \cdots + p_j =  4} } 
\left( \prod_{l=1}^j (p_l + 1)\alpha(-p_l-2) \right)  v
+  f_{j+ r-2}({\bf 1}) \\
&=&  \alpha(-2)^j \mathbf{1} *_2 v + w - 12j(4j + 3) \alpha(-5) \alpha(-2)^{j-1}v -30j \alpha(-6) \alpha(-2)^{j-1}v \\
& & - \ 48 j (j -1) \alpha(-5) \alpha(-3) \alpha(-2)^{j-2}v +   f_{j+r -2}({\bf 1} )
\end{eqnarray*}
where $w$ is a linear combination of terms of the form $\alpha(-2)^{j'} \alpha(-3)^{k'} \alpha(-4)^{l'}v$ with $j' + k' + l' = j$ and $j'<j$, and $f_{j+r-2}({\bf 1} ) \in F_{j+r-2}({\bf 1})$.  

Using the recursion Eqn.\ (\ref{recursion-level-2}), we have that for any $w \in V$
\[\alpha(-6)w \sim_2 - \left( \alpha(-3) + 3 \alpha(-4) + 3 \alpha(-5) \right)w \]
and substituting this into the previous equation as well as simplifying the coefficient of the term on the left, we have
\begin{eqnarray}\label{reducing-2-first}
\lefteqn{ (j+1)(2j+1) (j-1)(2j-1)  \alpha(-2)^jv  \ =  \ (j^2 - 1)(4j^2 -1)  \alpha(-2)^jv }  \\
&\sim_2&  \alpha(-2)^j \mathbf{1} *_2 v + w   - 12j(4j + 3)  \alpha(-5) \alpha(-2)^{j-1}v  \nonumber  \\
& & + 30j \left( \alpha(-3) + 3 \alpha(-4) + 3 \alpha(-5)\right) \alpha(-2)^{j-1}v  \nonumber \\
& & - \ 48 j(j-1) \alpha(-5) \alpha(-3) \alpha(-2)^{j-2}v +  f_{j+r -2}({\bf 1})  \nonumber \\
%&\sim_2 &   \alpha(-2)^j \mathbf{1} *_2 v - 12j(4j + 3) \alpha(-5) \alpha(-2)^{j-1}v + 90j\alpha(-5) \alpha(-2)^{j-1}v \nonumber \\
%& & - \ 48j(j-1) \alpha(-5) \alpha(-3) \alpha(-2)^{j-2}v + w' +   f_{j+r-2}({\bf 1})  \nonumber \\
& = &  \alpha(-2)^j \mathbf{1} *_2 v - 6j(8j - 9) \alpha(-5) \alpha(-2)^{j-1}v  \nonumber \\
& & - 48 j(j-1) \alpha(-5) \alpha(-3) \alpha(-2)^{j-2}v  + w' +  f_{j+r-2}({\bf 1}) \nonumber
\end{eqnarray}
where $w'$ is a linear combination of terms of the form $\alpha(-2)^{j'} \alpha(-3)^{k'} \alpha(-4)^{l'}v$ with $j' + k' + l' = j$ and $j'<j$, and  $f_{j+r-2}({\bf 1}) \in F_{j+r-2}({\bf 1})$. 

Letting $v = \alpha(-1)^i \alpha(-3)^k \alpha(-4)^l \mathbf{1}$, for $k, l \in \mathbb{Z}_+$ and $0\leq i \leq 2$, we use Eqn.\ (\ref{reduce-equation-Heisenberg-level2}) to reduce terms involving $\alpha(-5)$, noting that  if $i_5 = 1$, $i_4 = l$, $i_3 = k$,  $i_2 = j-1$, and $i_1 = i$ in Eqn.\ (\ref{reduce-equation-Heisenberg-level2}), we have 
\begin{eqnarray*}
%\lefteqn{\alpha(-5) \alpha(-2)^{j-1} \alpha(-1)^i \alpha(-3)^k \alpha(-4)^l \mathbf{1}  }\\
%&=& 
\lefteqn{\alpha(-5)   \alpha(-4)^l \alpha(-3)^k \alpha(-2)^{j-1} \alpha(-1)^i\mathbf{1} }\\
&\equiv_2& -\frac{1}{4(l+ 1) } \Bigl(  \left(4  + i + 2(j-1) + 3k + 4l \right) \alpha(-4)^{l+1} \alpha(-3)^{k} \alpha(-2)^{j-1}  \alpha( - 1)^{i } \mathbf{1}  \\
& & \quad  + \, 3 k \alpha(-4)^{l + 2} \alpha(-3)^{k- 1} \alpha(-2)^{j-1}  \alpha(-1)^{i} \mathbf{1}   \nonumber  \\
& & \quad  + \, 2 (j-1) \alpha(-4)^{l+1} \alpha(-3)^{k+1}  \alpha(-2)^{j - 2} \alpha( - 1)^{i } \mathbf{1}  \\
& & \quad + \, i \alpha(-4)^{l+1} \alpha(-3)^{k} \alpha(-2)^{j} \alpha( - 1)^{i-1 } \mathbf{1}  \Bigr) ,
\end{eqnarray*}
and if $i_5 = 1$, $i_4 = l$, $i_3 = k+1$,  $i_2 = j-2$, and $i_1 = i$ in Eqn.\ (\ref{reduce-equation-Heisenberg-level2}), we have 
\begin{eqnarray*}
%\lefteqn{\alpha(-5) \alpha(-3) \alpha(-2)^{j-2} \alpha(-1)^i \alpha(-3)^k \alpha(-4)^l \mathbf{1}}\\
%&=& 
\lefteqn{\alpha(-5) \alpha(-4)^{l} \alpha(-3)^{k+1} \alpha(-2)^{j-2} \alpha(-1)^{i} \mathbf{1} } \\
&\equiv_2& -\frac{1}{4( l + 1) } \Bigl(  \left(4  +  i + 2(j-2) + 3(k+1) + 4l  \right) \alpha(-4)^{l+1} \alpha(-3)^{k+1} \alpha(-2)^{j-2}  \alpha( - 1)^{i } \mathbf{1} \ \\
& & \quad  + \, 3 (k+1)  \alpha(-4)^{l + 2} \alpha(-3)^{k} \alpha(-2)^{j-2}  \alpha(-1)^{i} \mathbf{1}    \\
& & \quad  + \, 2 (j-2) \alpha(-4)^{l+1} \alpha(-3)^{k+2}  \alpha(-2)^{j-3} \alpha( - 1)^{i } \mathbf{1}  \\
& & \quad + \, i \alpha(-4)^{l+1} \alpha(-3)^{k+1} \alpha(-2)^{j-1} \alpha( - 1)^{i-1 } \mathbf{1} \Bigr) .
\end{eqnarray*}

Therefore  for $j >1$, and $i,j,k \in \mathbb{N}$, substituting these last two equations into Eqn.\ (\ref{reducing-2-first}), with $v = \alpha(-1)^i \alpha(-3)^k \alpha(-4)^l \mathbf{1}$, we have
\begin{eqnarray}\label{reducing-2}
\lefteqn{\ \ \ \ \ \alpha(-1)^i \alpha(-2)^j \alpha(-3)^k \alpha(-4)^l \mathbf{1} = \alpha(-2)^j (\alpha(-1)^i \alpha(-3)^k \alpha(-4)^l \mathbf{1}) } \\
%&\sim_2& \frac{1}{(j^2-1)(4j^2-1)} \Bigl(   \alpha(-2)^j \mathbf{1} *_2   \alpha(-1)^i \alpha(-3)^k \alpha(-4)^l \mathbf{1}   \nonumber\\
%& &  - 6j(8j - 9) \alpha(-5) \alpha(-2)^{j-1} \alpha(-1)^i \alpha(-3)^k \alpha(-4)^l \mathbf{1} \nonumber\\
%& &  - \ 48 j(j-1) \alpha(-5) \alpha(-3) \alpha(-2)^{j-2} \alpha(-1)^i \alpha(-3)^k \alpha(-4)^l \mathbf{1}  \nonumber \\
%& & + w'+  f_{i+ j + k + l -2}(  \mathbf{1})  \Bigr) \nonumber \\
&\equiv_2&  \frac{1}{(j^2-1)(4j^2-1)} \Bigl(   \alpha(-2)^j \mathbf{1} *_2   \alpha(-1)^i \alpha(-3)^k \alpha(-4)^l \mathbf{1}   \nonumber \\
& &   + 6j(8j - 9)  \frac{i}{4( l+ 1) }  \alpha(-4)^{l+1} \alpha(-3)^{k} \alpha(-2)^{j} \alpha( - 1)^{i-1 } \mathbf{1} 
\Bigr)   + w''+   f'_{i+j+k + l-2}( \mathbf{1}) \nonumber
\end{eqnarray}
where  $w''$ is a linear combination of terms of the form $\alpha(-1)^{i'} \alpha(-2)^{j'} \alpha(-3)^{k'} \alpha(-4)^{l'} \mathbf{1} $ with $i'+ j' + k' + l' = i + j + k +
 l$, $i'\leq i$, and $j'<j$, and $f'_{i+ j+ k + l-2}(\mathbf{1}) \in F_{i + j + k + l-2}(\mathbf{1})$.  
 
If $i \neq 0$, we apply Eqn.\ (\ref{reducing-2}) again for the second term to obtain
\begin{eqnarray*}
\lefteqn{\alpha(-1)^i \alpha(-2)^j \alpha(-3)^k \alpha(-4)^l \mathbf{1} }\\
&\equiv_2&  \frac{1}{(j^2-1)(4j^2-1)} \biggl(   \alpha(-2)^j \mathbf{1} *_2   \alpha(-1)^i \alpha(-3)^k \alpha(-4)^l \mathbf{1}  \nonumber \\
& &  +  6j(8j -9)  \frac{i}{4( l+ 1) } \frac{1}{(j^2 - 1)(4j^2-1)} \Bigl(   \alpha(-2)^j \mathbf{1} *_2   \alpha(-1)^{i-1} \alpha(-3)^k \alpha(-4)^{l+1} \mathbf{1}   \\
& &  + 6j(8j -9 ) \frac{1}{4( l+ 2) } (i-1) \alpha(-4)^{l+2} \alpha(-3)^{k} \alpha(-2)^{j} \alpha( - 1)^{i-2 } \mathbf{1}  \Bigr)
\biggr) + \tilde{w} +   f''_{i+j+k + l-2}( \mathbf{1})
\end{eqnarray*}
where  $\tilde{w}$ is a linear combination of terms of the form $\alpha(-1)^{i'} \alpha(-2)^{j'} \alpha(-3)^{k'} \alpha(-4)^{l'} \mathbf{1}$ with $i'+ j' + k' + l' = i + j + k +
 l$, $i'<i$, and $j'<j$, and $f''_{i+ j+ k + l-2}(\mathbf{1}) \in F_{i + j + k + l-2}(\mathbf{1})$.   And finally if $i = 2$, we apply  Eqn.\ (\ref{reducing-2}) again for the third term to obtain
\begin{equation*}
\alpha(-1)^i \alpha(-2)^j \alpha(-3)^k \alpha(-4)^l \mathbf{1} \equiv_2   \alpha(-2)^j \mathbf{1} *_2  v + \tilde{w}'  +   f'''_{i+j+k + l-2}( \mathbf{1})
\end{equation*}
where  $v$ is a linear combination of elements of the form $\alpha(-1)^{i'} \alpha(-3)^{k'} \alpha(-4)^{l'} \mathbf{1}$ with $i'\leq i$, and $i' + j' + k' = i + j + k$,  and $\tilde{w}'$ is a linear combination of terms of the form $\alpha(-1)^{i'} \alpha(-2)^{j'} \alpha(-3)^{k'} \alpha(-4)^{l'}\mathbf{1}$ with $i'+ j' + k' + l' = i + j + k + l$, $i'<1$, and $j'<j$, and $f'''_{i+ j+ k + l-2}(\mathbf{1}) \in F_{i + j + k + l-2}(\mathbf{1})$. 
 
 By induction on the total degree $i + j + k + l$ for $0\leq i\leq2$ and $j,k,l \in \mathbb{N}$, we have that elements in (\ref{generators-again}) are generated mod $O_2(V)$,  by elements in (\ref{generators-again2}) and $\alpha(-2)^j \mathbf{1}$ for $j \in \mathbb{Z}_+$.  
 
 It remains to show that  $\alpha(-2)^j \mathbf{1}$ for $j \in \mathbb{Z}_+$ are generated by elements in (\ref{generators-again2}) mod $O_2(V)$.

But setting $i = k = l = 0$ in Eqn.\ (\ref{reducing-2}), we have that 
\begin{equation*}
\alpha(-2)^j\mathbf{1} \equiv_2  \frac{1}{(j^2-1)(4j^2-1)} \left(   \alpha(-2)^j \mathbf{1} *_2    \mathbf{1} \right) + w''+  f_{j-2}( \mathbf{1}) 
\end{equation*}
where $w''$ is a linear combination of terms of the form $\alpha(-2)^{j'} \alpha(-3)^{k'} \alpha(-4)^{l'}\mathbf{1}$ with  $j' + k' + l' = j$  and $j'<j$, and  $f_{ j-2}(\mathbf{1}) \in F_{j -2}(\mathbf{1})$.   Since $\mathbf{1} + O_2(M_a(1))$ is a multiplicative identity in $A_2(M_a(1))$, we have that
 \begin{equation*}
\Bigl(1 -   \frac{1}{(j^2-1)(4j^2-1)} \Bigr)\alpha(-2)^j\mathbf{1} 
\equiv_2 w''+  f_{j-2}( \mathbf{1}) .
\end{equation*}
Since $j\geq 0$, the coefficient on the left is nonzero, and it follows that $\alpha(-2)^j\mathbf{1}$ is a linear combination of terms of the form $\alpha(-2)^{j'} \alpha(-3)^{k'} \alpha(-4)^{l'}\mathbf{1}$ with $j' + k' + l' =  j$ and $j'<j$, and elements in  $F_{ j -2}(\mathbf{1})$.  The result follows by induction on $j$.
\end{proof}

\subsection{Additional formulas and relations needed for the further reduction of generators}

Continuing along this path of reducing the set of generators, we next want to prove that $A_2(M_a(1))$ is generated by $v + O_2(M_a(1))$ for $v$ in 
\[ \{ \alpha(-1)^i \alpha(-2)^j \alpha(-3)^k \alpha(-4)^l \mathbf{1} \; | \; 0\leq i \leq 2, \ 0\leq j\leq 1,\  0\leq k \leq 1, \, l \in \mathbb{N} \}.\]
To do so, we will need several formulas arising from Eqn.\ \eqref{reduce-equation-Heisenberg-level2} as well as some relations given by elements of $O_2(M_a(1))$.  We develop these formulas in the following lemmas.

In the next two lemmas, we determine some relations given by elements of $O_2(M_a(1))$ that follow from Eqn.\ (\ref{Y+-2}), the $L(-1)$-derivative property in $M_a(1)$, and the definition of $\circ_2$.  

\begin{lem}\label{Y+-lemma} 
Recalling that $Y^+(u,x) \in (\mathrm{End}(V))[[x]]$ and $Y^-(u,x) \in x^{-1}(\mathrm{End}(V))[[x^{-1}]]$ denote the regular and singular terms for the vertex operator $Y(u,x)$ for $u \in V$, respectively, we have  that for $v \in M_a(1)$
\begin{eqnarray*}
Y^+(\alpha(-1){\mathbf{1}},x)v &\sim_2& \alpha(-1)v+\alpha(-2)vx+\alpha(-3)v\frac{x^2}{1+x}+(\alpha(-3)+\alpha(-4))v\frac{x^3}{(1+x)^2}\\
& & \quad + \,  (\alpha(-3)+2\alpha(-4)+\alpha(-5))v\frac{x^4}{(1+x)^3},\\
Y^+(\alpha(-2){\mathbf{1}},x) v &\sim_2& \alpha(-2)v +\alpha(-3) v \frac{2x+x^2}{(1+x)^2} +(\alpha(-3)+\alpha(-4)) v \frac{3x^2+x^3}{(1+x)^3}\\
& & \quad + \,  (\alpha(-3)+2\alpha(-4)+\alpha(-5))v\frac{4x^3+x^4}{(1+x)^4},\\
Y^+(\alpha(-3){\mathbf{1}},x) v &\sim_2& \alpha(-3)v \frac{1}{(1+x)^3}+(\alpha(-3)+\alpha(-4))v \frac{3x}{(1+x)^4}\\
& & \quad +\,  (\alpha(-3)+2\alpha(-4)+\alpha(-5))v \frac{6x^2}{(1+x)^5},\\
Y^+(\alpha(-4){\mathbf{1}},x)v &\sim_2& -\alpha(-3)v \frac{1}{(1+x)^4}+(\alpha(-3)+\alpha(-4))v \frac{(1-3x)}{(1+x)^5}\\
& & \quad + \,  (\alpha(-3)+2\alpha(-4)+\alpha(-5))v \frac{4x-6x^2}{(1+x)^6}.
\end{eqnarray*}

In addition, for $m \in \mathbb{Z}_+$
\begin{equation}\label{Y-}
 \mathrm{Res}_x x^{k + m-1} Y^-(\alpha(-m)\mathbf{1},x )= 0, \quad \mbox{for $k<0$,} 
 \end{equation}
 and for $m_1, m_2 \in \mathbb{Z}_+$
 \begin{equation}\label{Y-2}
 \mathrm{Res}_x x^{k + m_1 + m_2 -1} Y^-(\alpha(-m_1)\mathbf{1},x ) Y^-(\alpha(-m_2) \mathbf{1}, x) = 0, \quad \mbox{for $k<0$.} 
 \end{equation}
\end{lem}

\begin{proof}
From Eqn.\ (\ref{Y+-2}), and assuming any rational function in $x$ means the formal power series expanded in terms of positive coefficients in $x$, we have 
\begin{eqnarray*}
\lefteqn{Y^+(\alpha(-1){\mathbf{1}}, x)v}\\
&\sim_2& \alpha(-1)v+\alpha(-2)vx+\frac{\alpha(-3)v}{2}\sum_{m\ge3}(-1)^{m+1}(m-4)(m-5)x^{m-1}\\
& & \quad + \ \alpha(-4)v\sum_{m\ge3}(-1)^{m+1}(m-3)(m-5)x^{m-1} \\
& & \quad + \ \frac{\alpha(-5)v}{2}\sum_{m\ge3}(-1)^{m+1}(m-3)(m-4)x^{m-1}\\
&=&\alpha(-1)v+\alpha(-2)vx+\alpha(-3)v\frac{x^2(3x^2+3x+1)}{(1+x)^3} + \alpha(-4)v\frac{3x^4+x^3}{(1+x)^3}+\alpha(-5)v\frac{x^4}{(1+x)^3}\\
%&=&\alpha(-1)v+\alpha(-2)vx+(3\alpha(-3)+3\alpha(-4)+\alpha(-5))v\frac{x^4}{(1+x)^3}\\
%& & \quad + \ (3\alpha(-3)+\alpha(-4))v\frac{x^3}{(1+x)^3}+\alpha(-3)\frac{x^2}{(1+x)^3}\\
&= &\alpha(-1)v+\alpha(-2)vx+\alpha(-3)v\frac{x^2}{1+x}+(\alpha(-3)+\alpha(-4))v\frac{x^3}{(1+x)^2}\\
&& \quad + \ (\alpha(-3)+2\alpha(-4)+\alpha(-5))v\frac{x^4}{(1+x)^3}.
\end{eqnarray*}
The other formulas are obtained by using the $L(-1)$ derivative property, and the fact that $(k-1)!\alpha(-k) \mathbf{1} = L(-1)^k \alpha(-1)\mathbf{1}$.   The $L(-1)$ derivative property along with the fact that  $\mathrm{Res}_x x^{k} Y^-(\alpha(-1)\mathbf{1},x ) = \sum_{j\in \mathbb{N}} \alpha(j) x^{k-j-1}= 0$ for $k<0$ gives (\ref{Y-}) as well as (\ref{Y-2}).
\end{proof}

We now calculate some elements in $O_2(V)$ using the definition of $\circ_2$ given by \eqref{define-circ} and Lemma \ref{Y+-lemma} which allows one to target which $i, j, k, l \in \mathbb{N}$ give terms in $u \circ_2 v$ for $u = \alpha(-1)^i\alpha(-2)^j\alpha(-3)^k\alpha(-4)^l \mathbf{1}$ where the singular terms for $Y(u, x)$ do not contribute.  In particular, 
\begin{eqnarray}\label{start-of-O2-lemma}
\lefteqn{\alpha(-1)^i\alpha(-2)^j\alpha(-3)^k\alpha(-4)^l{\mathbf{1}}\circ_2v}\\
&=&
\res_x\frac{(1+x)^{i+2j+3k+4l+2}}{x^6}Y(\alpha(-1)^i\alpha(-2)^j\alpha(-3)^k\alpha(-4)^l{\mathbf{1}},x)v \nonumber \\
%&=& 
%\res_x\frac{(1+x)^{i+2j+3k+4l+2}}{x^6} {}_{\circ}^{\circ} \bigg((Y^-(\alpha(-1){\mathbf{1}}, x)+Y^+(\alpha(-1){\mathbf{1}}, x))^i \nonumber \\
%& & \quad (Y^-(\alpha(-2){\mathbf{1}}, x)  +  Y^+(\alpha(-2){\mathbf{1}}, x))^j(Y^-(\alpha(-3){\mathbf{1}}, x) +Y^+(\alpha(-3){\mathbf{1}}, x))^k \nonumber \\
%& & \quad (Y^-(\alpha(-4){\mathbf{1}}, x)  +  Y^+(\alpha(-4){\mathbf{1}}, x))^l\bigg){}_{\circ}^{\circ}v\nonumber\\
&=& 
\res_x\frac{(1+x)^{2}}{x^6} {}_{\circ}^{\circ} \Big(((1+ x)Y^-(\alpha(-1){\mathbf{1}}, x)+ (1+x)Y^+(\alpha(-1){\mathbf{1}}, x))^i \nonumber \\
& & \quad ((1+x)^2 Y^-(\alpha(-2){\mathbf{1}}, x)  +  (1 + x)^2 Y^+(\alpha(-2){\mathbf{1}}, x))^j  \nonumber \\
& & \quad ( ( 1 + x)^3 Y^-(\alpha(-3){\mathbf{1}}, x)  + \ (1 + x)^3 Y^+(\alpha(-3){\mathbf{1}}, x))^k  \nonumber \\
& & \quad  ((1 + x)^4 Y^-(\alpha(-4){\mathbf{1}}, x)  +  (1 + x)^4 Y^+(\alpha(-4){\mathbf{1}}, x))^l\Big){}_{\circ}^{\circ}v\nonumber.
%\\
%&\sim_2&
%\res_x\frac{(1+x)^{2}}{x^6}{}_{\circ}^{\circ} \bigg((1+x)(Y^-(\alpha(-1){\mathbf{1}}, x) +(1+x)(\alpha(-1)+\alpha(-2)x \\
%& & \quad +\ \alpha(-3)\frac{x^2(3x^2+3x+1)}{(1+x)^3} +\alpha(-4)\frac{3x^4+x^3}{(1+x)^3}+\alpha(-5)\frac{x^4}{(1+x)^3})\bigg)^i\\
%& & \quad \bigg((1+x)^2Y^-(\alpha(-2){\mathbf{1}}, x)+(1+x)^2\bigg( \alpha(-2) +(3\alpha(-3)+3\alpha(-4)\\
%& & \quad + \ \alpha(-5))\frac{4x^3+x^4}{(1+x)^4} + (3\alpha(-3)+\alpha(-4))\frac{3x^2}{(1+x)^4}+\alpha(-3)\frac{2x-x^2}{(1+x)^4}\bigg)\bigg)^j\\
%& & \quad \bigg((1+x)^3Y^-(\alpha(-3){\mathbf{1}}, x) +(1+x)^3\bigg((3\alpha(-3)+ 3\alpha(-4)+\alpha(-5))\frac{6x^2}{(1+x)^5}\\
%& & \quad + \ (3\alpha(-3)+\alpha(-4))\frac{3x-3x^2}{(1+x)^5}
%+ \alpha(-3)\frac{x^2-4x+1}{(1+x)^5}\bigg)\bigg)^k \\
%& & \quad \bigg((1+x)^4Y^-(\alpha(-4){\mathbf{1}}, x) +  (1+x)^4\bigg((3\alpha(-3)+3\alpha(-4)+\alpha(-5))\frac{4x-6x^2}{(1+x)^6} \\
%& & \quad  +  \ (3\alpha(-3) +  \alpha(-4))\frac{3x^2-6x+1}{(1+x)^6}+\alpha(-3)\frac{(-x^2+6x-3)}{(1+x)^6}\bigg)\bigg)^l{}_{\circ}^{\circ}v.
\end{eqnarray}
%where
%\[ {}_{\circ}^{\circ}  Y(u, x) Y(v,x)   {}_{\circ}^{\circ} = Y^+(u,x) Y(v,x) + Y(v, x) Y^-(u,x) \]
%for all $u,v \in V$. 
Thus we can use Lemma \ref{Y+-lemma} to determine that each regular term $Y^
+(\alpha(-m) \mathbf{1} , x)$ for $m = 1, 2, 3$ and $4$ in (\ref{start-of-O2-lemma}), modulo $O_2(V)$, is of the order $(1+ x)^{-2}$ and higher in $(1+x)$.  This and (\ref{Y-}) imply that, for instance, in the cases when $(i,j,k,l)=(1,0,1,0)$, $(2,0,0,0)$, $(0,2,0,0),$ and $(0,0,2,0),$ the singular terms can be ignored in this expression and give important elements in $O_2(V)$, and thus relations in $A_2(V)$, as we show in the proof of the following Lemma.

For convenience, we introduce the following shorthand
\begin{equation}\label{ABC-notation}
A=\alpha(-1)+\alpha(-2),\hspace{.1in} B=\alpha(-2)+\alpha(-3),\hspace{.1in} C=\alpha(-3)+\alpha(-4), \hspace{.1in}D=\alpha(-4)+\alpha(-5).
\end{equation}

\begin{lem}\label{first-ABC-lemma} 
The following are in $O^\circ_2(V) \subset O_2(V)$ for all $v\in M_a(1)$.
\begin{eqnarray}\label{O21}
(\alpha(-3)+\alpha(-4))(\alpha(-3)+2\alpha(-4)+\alpha(-5))v   &=&  C(C+D)v ,\\
\label{O??}
(\alpha(-3)+2\alpha(-4)+\alpha(-5))^2 v  &=& (C+D)^2 v,\\ 
\label{O22}
(\alpha(-4)+\alpha(-5))(\alpha(-3)+2\alpha(-4)+\alpha(-5))v  &=& D(C+D)v ,
\end{eqnarray}
\begin{multline}\label{O23}
\bigl( (\alpha(-2)+2\alpha(-3)+\alpha(-4))^2+2(\alpha(-2)+\alpha(-3))(\alpha(-3)+2\alpha(-4)+\alpha(-5)) \bigr)v \\
= \big( (B+C)^2 + 2B(C+D) \big) v, 
\end{multline}
\begin{multline}\label{O24}
\bigl((\alpha(-1)+\alpha(-2))(\alpha(-2)+3\alpha(-3)+3\alpha(-4)+\alpha(-5))+\\(\alpha(-2)+\alpha(-3))(\alpha(-2)+2\alpha(-3)+\alpha(-4)) \bigr) v \\
 = A(B+2C+D)v + B(B+C)v  \\
=((A+B)(B+C)+A(C+D))v. \hspace{1.8in}
\end{multline}
\end{lem}

%Note that (\ref{O21}) and (\ref{O22}) imply that for all $v\in V$, 
%\begin{equation}\label{square2}(\alpha(-3)+2\alpha(-4)+\alpha(-5))^2v\end{equation} is an element of $O_2(V)$ for all $v\in V$.
%
%Also, (\ref{O21}) and (\ref{O22}) together imply
%
%\begin{equation}(\alpha(-3)+2\alpha(-4)+\alpha(-5))(\alpha(-3)-\alpha(-5))v\in O_2(V)\text{ for all }v\in V.\end{equation} This combined with (\ref{O23}) implies 
%
%\begin{equation}\alpha(-3)(\alpha(-3)-\alpha(-5))v\in O_2(V)\text{ for all }v\in V
%\end{equation}
\begin{proof}
Using Lemma \ref{Y+-lemma} and Eqn.\ (\ref{start-of-O2-lemma}), and continuing to use the convention that negative powers of $(1+x)$ are understood as expansions in positive powers of $x$, we calculate
\begin{eqnarray*}
\lefteqn{\alpha(-1)\alpha(-3){\mathbf{1}}\circ_2v }\\
%&=& \res_x\frac{(1+x)^6}{x^6}  {}_{\circ}^{\circ} \left(Y^-(\alpha(-1) \mathbf{1}, x) + Y^+(\alpha(-1) \mathbf{1},x) \right)\left( Y^-(\alpha(-3)\mathbf{1},x) + Y^+(\alpha(-3) \mathbf{1}, x) \right)  {}_{\circ}^{\circ} v\\
&=& 
\res_x\frac{(1+x)^{2}}{x^6} {}_{\circ}^{\circ} \left(  (1+ x)Y^-(\alpha(-1){\mathbf{1}}, x)+ (1+x)Y^+(\alpha(-1){\mathbf{1}}, x) \right) \\
& & \quad \left( ( 1 + x)^3 Y^-(\alpha(-3){\mathbf{1}}, x)  + \ (1 + x)^3 Y^+(\alpha(-3){\mathbf{1}}, x) \right) {}_{\circ}^{\circ}v\\
&\sim_2 &
\res_x\frac{(1+x)^2}{x^6}   {}_{\circ}^{\circ} \Big( (1+ x)Y^-(\alpha(-1){\mathbf{1}}, x) + \alpha(-1)(1 + x) +\alpha(-2)x(1 + x) +\alpha(-3) x^2 \\
& & \quad + \ (\alpha(-3) +  \alpha(-4))\frac{x^3}{1+x}   + (\alpha(-3)+2\alpha(-4)+\alpha(-5))\frac{x^4}{(1+x)^2}\Big) \\
& & \quad \Big(  ( 1 + x)^3 Y^-(\alpha(-3){\mathbf{1}}, x)  +  \alpha(-3) + (\alpha(-3)+\alpha(-4))\frac{3x}{1+x} \\
& & \quad + \  (\alpha(-3)+2\alpha(-4) + \alpha(-5))\frac{6x^2}{(1+x)^2}\Big)  {}_{\circ}^{\circ} v    \\
%&=&
%\res_x\frac{1}{x^6}   {}_{\circ}^{\circ} \Big( (1+ x)^2Y^-(\alpha(-1){\mathbf{1}}, x) + \alpha(-1)(1 + x)^2 +\alpha(-2)x(1 + x)^2 +\alpha(-3) x^2(1+x) \\
%& & \quad + \ (\alpha(-3) +  \alpha(-4)) x^3 + (\alpha(-3)+2\alpha(-4)+\alpha(-5))\frac{x^4}{1+x}\Big) \\
%& & \quad \Big(  ( 1 + x)^4 Y^-(\alpha(-3){\mathbf{1}}, x)  +  \alpha(-3)(1+ x) + (\alpha(-3)+\alpha(-4)) 3x \\
%& & \quad + \  (\alpha(-3)+2\alpha(-4) + \alpha(-5))\frac{6x^2}{1+x}\Big)  {}_{\circ}^{\circ} v    \\
&=&
\res_x\frac{1}{x^6}\Big(\alpha(-1)(1+x)^2+\alpha(-2)x(1+x)^2+\alpha(-3)x^2(1+x)\\
& & \quad + \ (\alpha(-3)+\alpha(-4))x^3+(\alpha(-3)+2\alpha(-4)+\alpha(-5))\frac{x^4}{1+x}\Big)\Big(\alpha(-3)(1+x)\\
& & \quad + \ (\alpha(-3)+\alpha(-4))3x+(\alpha(-3)+2\alpha(-4)+\alpha(-5))\frac{6x^2}{(1+x)}\Big)v \\
& & \quad +\  \res_x \frac{(1+ x)^6}{x^6} Y^-(\alpha(-1){\mathbf{1}}, x) Y^-(\alpha(-3){\mathbf{1}}, x)v + \res_x \frac{(1 + x)}{x^6} p_1(x) Y^-(\alpha(-1){\mathbf{1}}, x) v \\
& & \quad + \  \res_x  \frac{(1+ x)^3}{x^6} p_3(x) Y^-(\alpha(-3){\mathbf{1}}, x) v,
\end{eqnarray*}
where $p_1(x)$ and $p_3(x)$ are polynomials in $x$ of degree 2 and 4, respectively, with coefficients of the form $\alpha(-m) $.   By Eqn. (\ref{Y-}) and (\ref{Y-2}), the last three terms are zero, and we have
\begin{eqnarray*}
\lefteqn{\alpha(-1)\alpha(-3){\mathbf{1}}\circ_2v }\\
&\sim_2&
\res_x\frac{1}{x^6}(\alpha(-3)+\alpha(-4))(\alpha(-3)+2\alpha(-4)+\alpha(-5))\left( \frac{6x^5}{1 + x} + \frac{3x^5}{1 + x} \right)v \\
&=&
9(\alpha(-3)+\alpha(-4))(\alpha(-3)+2\alpha(-4)+\alpha(-5))v,
\end{eqnarray*}
which  implies \eqref{O21}.

%\begin{eqnarray*}\alpha(-1)\alpha(-2){\mathbf{1}}\circ_2v\sim &&\res_x\frac{(1+x)^5}{x^6}\Big(\alpha(-1)+\alpha(-2)x+\alpha(-3)\frac{x^2}{1+x}+(\alpha(-3)+\alpha(-4))\frac{x^3}{(1+x)^2}\\&+&(\alpha(-3)+2\alpha(-4)+\alpha(-5))\frac{x^4}{(1+x)^3}\Big)\Big(\alpha(-2)+\alpha(-3)\frac{2x+x^2}{(1+x)^2}\\&+&(\alpha(-3)+\alpha(-4))\frac{3x^2+x^3}{(1+x)^3}+(\alpha(-3)+2\alpha(-4)+\alpha(-5))\frac{4x^3+x^4}{(1+x)^4}\Big)v\\&=&
%\res_x\frac{1}{x^6}\Big(\alpha(-1)(1+x)^2+\alpha(-2)x(1+x)^2+\alpha(-3)x^2(1+x)\\&+&(\alpha(-3)+\alpha(-4))x^3+(\alpha(-3)+2\alpha(-4)+\alpha(-5))\frac{x^4}{1+x}\Big)\Big(\alpha(-2)(1+x)^3\\&+&\alpha(-3)(2x+x^2)(1+x)+(\alpha(-3)+\alpha(-4))(3x^2+x^3)\\&+&(\alpha(-3)+2\alpha(-4)+\alpha(-5))\frac{4x^3+x^4}{1+x}\Big)v \\&=&[\alpha(-1)(\alpha(-2)+3\alpha(-3)+3\alpha(-4)+\alpha(-5))\\&+&5\alpha(-2)(\alpha(-2)+3\alpha(-2)+3\alpha(-3)+3\alpha(-4)+\alpha(-5))\\&+&4\alpha(-3)(\alpha(-2)+3\alpha(-3)+3\alpha(-4)+\alpha(-5))\\&+&3(\alpha(-3)+\alpha(-4))(\alpha(-2)+2\alpha(-3)+\alpha(-4))\\&+&2(\alpha(-3)+2\alpha(-4)+\alpha(-5))(\alpha(-2)+\alpha(-3))]v\in O_2(V)\end{eqnarray*}
%ACTUALLY, I DON'T SEE WHERE WE ARE USING THIS. MAYBE IT SHOULD BE COMMENTED OUT.

Similarly, to prove \eqref{O22}, using Lemma \ref{Y+-lemma} and Eqn.\ (\ref{start-of-O2-lemma}), we have
\begin{eqnarray*}
\lefteqn{\alpha(-3)^2{\mathbf{1}}\circ_2v \ = \   \res_x\frac{(1+x)^{8}}{x^6} {}_{\circ}^{\circ} \left(  Y^-(\alpha(-3){\mathbf{1}}, x)+ Y^+(\alpha(-3){\mathbf{1}}, x) \right)^2  {}_{\circ}^{\circ}v} \\
&\sim_2& \res_x\frac{(1+x)^8}{x^6} {}_{\circ}^{\circ}  \Big(Y^-(\alpha(-3) \mathbf{1}, v) + \alpha(-3)\frac{1}{(1+x)^3}+(\alpha(-3)+\alpha(-4))\frac{3x}{(1+x)^4}\\
& &\quad + \ (\alpha(-3)+2\alpha(-4)+\alpha(-5))\frac{6x^2}{(1+x)^5}\Big)^2 {}_{\circ}^{\circ} v\\
%&=&
%\res_x\frac{1}{x^6}   {}_{\circ}^{\circ} \Big((1 + x)^4 Y^-(\alpha(-3) \mathbf{1}, v) + 
%\alpha(-3)(1+x)+(\alpha(-3)+\alpha(-4))3x\\
%& & \quad + \ (\alpha(-3)+ 2\alpha(-4)+\alpha(-5))\frac{6x^2}{1+x}\Big)^2   {}_{\circ}^{\circ}v \\
&=& \res_x\frac{1}{x^6} \Big( (\alpha(-3)+\alpha(-4))3x + (\alpha(-3)+ 2\alpha(-4)+\alpha(-5))\frac{6x^2}{1+x}\Big)^2 v \\
& & + \  \res_x \frac{(1+ x)^8}{x^6} (Y^-(\alpha(-3){\mathbf{1}}, x))^2v  + \res_x \frac{(1 + x)^3}{x^6} p(x) Y^-(\alpha(-3){\mathbf{1}}, x)v  ,
\end{eqnarray*}
for $p(x)$ a polynomial of in $x$ of degree 2 with coefficients of the form $\alpha(-m)$.  By Eqn. (\ref{Y-}) and (\ref{Y-2}), the last two terms are zero, and we have
\begin{eqnarray*}
\alpha(-3)^2{\mathbf{1}}\circ_2v & \sim_2 &  \res_ x\frac{1}{x^6} \Big( (\alpha(-3)+\alpha(-4))(\alpha(-3)+ 2\alpha(-4)+\alpha(-5)) \frac{18x^3}{1+ x}  \\
& & + \   (\alpha(-3)+ 2\alpha(-4)+\alpha(-5))^2   \frac{36x^4}{(1+x)^2}\Big)^2 v \\
&=& 36(\alpha(-3)+\alpha(-4))(\alpha(-3)+2\alpha(-4)+\alpha(-5))v\\
& & - \, 72(\alpha(-3)+2\alpha(-4)+\alpha(-5))^2 v\\
\end{eqnarray*}

This combined with \eqref{O21} implies $(\alpha(-3)+2\alpha(-4)+\alpha(-5))^2 v \in O_2(M_a(1))$, giving Eqn.\ (\ref{O??}).  

Thus $(C+D)^2v - C(C+D)v = (CD + D^2)v = D(C+D)v \in O_2(M_a(1))$ 
which is Eqn.\ \eqref{O22}.

For Eqn.\ (\ref{O23}), using Lemma \ref{Y+-lemma} and Eqn.\ (\ref{start-of-O2-lemma}), we calculate
\begin{eqnarray*}
\lefteqn{\alpha(-2)^2{\mathbf{1}}\circ_2 v} \\
&=& \res_x\frac{(1+x)^6}{x^6} \left( Y^-(\alpha(-2) \mathbf{1}, x) + Y^+(\alpha(-2)\mathbf{1}, x \right)^2 \\
&\sim_2& \res_x\frac{(1+x)^6}{x^6}\Big(Y^-(\alpha(-2) \mathbf{1}, x) + \alpha(-2)+\alpha(-3)\frac{2x+x^2}{(1+x)^2}+(\alpha(-3) \\
&& + \, \alpha(-4))\frac{3x^2+x^3}{(1+x)^3}+ (\alpha(-3)+2\alpha(-4)+\alpha(-5))\frac{4x^3+x^4}{(1+x)^4}\Big)^2v\\
%&=& \res_x\frac{(1+x)^6}{x^6}\Big(\alpha(-2)+\alpha(-3)\frac{2x+x^2}{(1+x)^2}+(\alpha(-3)+\alpha(-4))\frac{3x^2+x^3}{(1+x)^3}\\
%&&+\,  (\alpha(-3)+2\alpha(-4)+\alpha(-5))\frac{4x^3+x^4}{(1+x)^4}\Big)^2v\\
&=&
\res_x\frac{1}{x^6}\Big(\alpha(-2)(1+x)^3+\alpha(-3)(2x+x^2)(1+x)+(\alpha(-3)+\alpha(-4))(3x^2+x^3)\\
& & + \,  (\alpha(-3)+2\alpha(-4)+\alpha(-5))\frac{4x^3+x^4}{(1+x)}\Big)^2v \\
&=& 6\Bigl( (B+C)^2 + 2(C+D)B + 4(C+D)C \Big)v \\
&\sim_2& 6\Bigl( (B+C)^2 + 2(C+D)B  \Big)v, 
\end{eqnarray*}
where the last line follows from Eqn.\ \eqref{O21}, and this gives  Eqn.\ (\ref{O23}).

Finally, for Eqn.\ (\ref{O24}), using Lemma \ref{Y+-lemma} and Eqn.\ (\ref{start-of-O2-lemma}), we calculate
\begin{eqnarray*}
\lefteqn{\alpha(-1)^2{\mathbf{1}}\circ_2v }\\
&=& \res_x\frac{(1+x)^4}{x^6}  \left( Y^-(\alpha(-1) \mathbf{1}, x) + Y^+(\alpha(-1)\mathbf{1}, x \right)^2 v \\
&\sim_2 &\res_x\frac{(1+x)^4}{x^6}\Big(Y^-(\alpha(-1) \mathbf{1}, x)  + \alpha(-1)+\alpha(-2)x+\alpha(-3)\frac{x^2}{1+x}+(\alpha(-3)\\
& & + \, \alpha(-4))\frac{x^3}{(1+x)^2}+(\alpha(-3)+2\alpha(-4)+\alpha(-5))\frac{x^4}{(1+x)^3}\Big)^2v\\
%&=& \res_x\frac{(1+x)^4}{x^6}\Big( \alpha(-1)+\alpha(-2)x+\alpha(-3)\frac{x^2}{1+x}+(\alpha(-3)\\
%& & + \, \alpha(-4))\frac{x^3}{(1+x)^2}+(\alpha(-3)+2\alpha(-4)+\alpha(-5))\frac{x^4}{(1+x)^3}\Big)^2v\\
&=& \res_x\frac{1}{x^6}\Big(\alpha(-1)(1+x)^2+\alpha(-2)x(1+x)^2\\
& & + \, \alpha(-3)x^2(1+x)+(\alpha(-3)+\alpha(-4))x^3+(\alpha(-3)+2\alpha(-4)+\alpha(-5))\frac{x^4}{1+x}\Big)^2v \\
%&=& \Big( 2\alpha(-1) \alpha(-2) + 2\alpha(-1) \alpha(-3)  + 2 \alpha(-1) ( \alpha(-3) + \alpha(-4))   \\
%& & + \, 2 \alpha(-1) ( \alpha(-3) + 2 \alpha(-4) + \alpha(-5))  + \, 4 \alpha(-2)^2  \\
%& & + \,  6 \alpha(-2) \alpha(-3) + 4 \alpha(-2) ( \alpha(-3) + \alpha(-4))  + 2 \alpha(-2) ( \alpha(-3) + 2 \alpha(-4) + \alpha(-5))   \\
%& & + \,  2\alpha(-3)^2  + 2 \alpha(-3) ( \alpha(-3) + \alpha(-4)) \Big)v \\
&=& 2 \Big( \big( \alpha(-1)+\alpha(-2) \big) \big(\alpha(-2)+3\alpha(-3)+3\alpha(-4)+\alpha(-5) \big)\\
&& + \, \big(\alpha(-2)+\alpha(-3) \big) \big(\alpha(-2)+2\alpha(-3)+\alpha(-4) \big) \Big)v  \in O_2(V) 
\end{eqnarray*}
giving Eqn.\ \eqref{O24}. 
\end{proof}

\begin{lem}\label{helprel} 
For $l  \in \mathbb{N}$, we have 
\begin{eqnarray}
\alpha(-4)^l \alpha(-5) {\mathbf{1}} \! &\equiv_2 & \!  -\alpha(-4)^{l+1}{\mathbf{1}}, \label{45=44}\\
\alpha(-3)\alpha(-4)^l{\mathbf{1}}  \! & \equiv_2 &\!    -\alpha(-4)^{l+1}{\mathbf{1}}, \label{34=44} \\
\alpha(-3)^2\alpha(-4)^l{\mathbf{1}} \! &  \equiv_2 & \!   \alpha(-4)^{l+2}{\mathbf{1}},  \label{3^24} \\
\ \ \ \ \ \ \ \ \  \ \alpha(-3)\alpha(-4)^l\alpha(-5){\mathbf{1}} \! &  \equiv_2  & \! 
\frac{1}{4(l+1)} \big(-3\alpha(-4)^{l+2}{\mathbf{1}}  \label{345} \\
& & \quad - \, (3+4(l+1))\alpha(-3)\alpha(-4)^{l+1}{\mathbf{1}} \big)  \nonumber \\
&  \equiv_2  &  \!  \alpha(-4)^{l+2}{\bf 1}, \nonumber\\
\label{455}
\alpha(-4)^l\alpha(-5)^2{\mathbf{1}}
\! & \equiv_2 & \!  \frac{1}{4(l+1)} \big(5\alpha(-3)\alpha(-4)^{l+1}{\mathbf{1}}\\
& & \quad + \, (5+4(l+1))\alpha(-4)^{l+2}{\bf{1}} \big),\nonumber \\
\label{145}
\alpha(-1)\alpha(-4)^{l}\alpha(-5){\bf{1}}  \! & 
 \equiv_2   \! &   -\frac{1}{4(l+1)}\alpha(-2)\alpha(-4)^{l+1}{\bf{1}}  \\
& &   \quad - \,  \frac{(1+4(l+1))}{4(l+1)}\alpha(-1)\alpha(-4)^{l+1}{\bf{1}}, \nonumber\\
\label{245}
\alpha(-2)\alpha(-4)^l\alpha(-5){\mathbf{1}}
\! &  \equiv_2 & \! 
   \frac{1}{2(l+1)}\alpha(-4)^{l+2} \mathbf{1} - \frac{(1+2(l+1))}{2(l+1)}\alpha(-2)\alpha(-4)^{l+1}\mathbf{1} , \\
\label{234}
\alpha(-2)\alpha(-3)\alpha(-4)^l{\mathbf{1}} \!  &  \equiv_2  &  \!  \frac{1}{10}\alpha(-1){\mathbf{1}}*_2\alpha(-2)\alpha(-4)^l{\mathbf{1}}
- \frac{3(3l+2)}{10(l+1)}\alpha(-2)\alpha(-4)^{l+1}{\mathbf{1}}\\
& & \quad - \, \frac{3}{10(l+1)}\alpha(-4)^{l+2}{\mathbf{1}},  \nonumber\\
\label{134} 
\alpha(-1)\alpha(-3)\alpha(-4)^l{\mathbf{1}} \! &  \equiv_2   & \!  \frac{1}{10}\alpha(-1){\mathbf{1}}*_2\alpha(-1)\alpha(-4)^l{\mathbf{1}}
- \frac{3(6l+5)}{20(l+1)}\alpha(-1)\alpha(-4)^{l+1}{\mathbf{1}}\\
& & \quad + \, \frac{3}{20(l+1)}\alpha(-2)\alpha(-4)^{l+1}{\mathbf{1}}, \nonumber \\
\label{224}
\alpha(-2)^2\alpha(-4)^l{\mathbf{1}} \! &  \equiv_2 &  \!  -\frac{3}{5}\alpha(-1){\mathbf{1}}*_2\alpha(-2)\alpha(-4)^l{\mathbf{1}}+\frac{7l+3}{5(l+1)}\alpha(-2)\alpha(-4)^{l+1}{\mathbf{1}} \\
& & \quad - \,  \frac{(5l+1)}{5(l+1)}\alpha(-4)^{l+2}{\mathbf{1}}. \nonumber
\end{eqnarray} 
\end{lem}

\begin{proof} Eqn.\  \eqref{45=44}, the first equivalences in Eqn.\ \eqref{345}, and Eqn.\ \eqref{145} follow from Eqn.\ \eqref{reduce-equation-Heisenberg-level2}.  (The second equivalence in \eqref{345}  requires \eqref{34=44}, which we prove below.) Eqn.\  
\eqref{455} follows by applying Eqn.\ \eqref{reduce-equation-Heisenberg-level2} twice.

When $l=0$, Eqn.\  \eqref{34=44}  follows from the fact that $(L(-1)+L(0))\alpha(-3){\bf{1}}\in O_2(V)$.  
For $l>0$, by Eqn.\ \eqref{O22}, applied to $v = \alpha(-4)^l \mathbf{1}$, we have
\begin{multline*}
\alpha(-3)\alpha(-4)^{l+1}{\mathbf{1}} \equiv_2 -\alpha(-3)\alpha(-4)^l\alpha(-5){\mathbf{1}}\\
-2\alpha(-4)^{l+2}{\mathbf{1}}- 3 \alpha(-4)^{l+1}\alpha(-5){\mathbf{1}}
-\alpha(-4)^l\alpha(-5)^2{\mathbf{1}}.
\end{multline*} 
Substituting Eqns.\ (\ref{45=44}), the first equivalence in Eqn.\  \eqref{345}, and Eqn.\ \eqref{455} into this equation and simplifying gives 
\[
\alpha(-3)\alpha(-4)^{l+1}{\mathbf{1}} \equiv_2 \frac{-2}{4(l+1)} \alpha(-4)^{l+2} \mathbf{1} + \left( \frac{-2+4(l+1)}{4(l+1)} \right) \alpha(-3) \alpha(-4)^{l+1} \mathbf{1} \]
which simplifies to  Eqn.\ \eqref{34=44}. Substituting Eqn.\ \eqref{34=44} into the first equivalence in (\ref{345}) gives the second equivalence in Eqn.\ (\ref{345}). 

Eqn.\ \eqref{245} follows from applying Eqn.\ \eqref{reduce-equation-Heisenberg-level2}  and then applying \eqref{34=44}.

Eqn.\ \eqref{3^24} follows from applying Eqn.\  \eqref{O21} and then applying Eqns.\  \eqref{34=44}, \eqref{345}, and \eqref{45=44}.

Eqns.\ \eqref{134} and \eqref{234} follow from applying Eqn.\ \eqref{multiplication-one-one}, and then Eqn.\  \eqref{reduce-equation-Heisenberg-level2}, and in the case of \eqref{234}, then applying \eqref{34=44}.

To prove \eqref{224}, we apply \eqref{O23} to obtain
\begin{eqnarray*}
\alpha(-2)^2\alpha(-4)^l{\mathbf{1}}&  \equiv_2 &-6 \alpha(-2) \alpha(-3)\alpha(-4)^l{\mathbf{1}} - 6 \alpha(-3)^2\alpha(-4)^l{\mathbf{1}} - 6 \alpha(-2)\alpha(-4)^{l+1}{\mathbf{1}}\\
& & \quad - \, 8 \alpha(-3)\alpha(-4)^{l+1}{\mathbf{1}} - \alpha(-4)^{l+2}{\mathbf{1}} - 2 \alpha(-2)\alpha(-4)^l\alpha(-5){\mathbf{1}} \\
& & \quad - \, 2\alpha(-3)\alpha(-4)^l\alpha(-5){\mathbf{1}}
\end{eqnarray*}
and then apply \eqref{34=44}, \eqref{3^24},  \eqref{345}, \eqref{245}, and \eqref{234}.
\end{proof}

\begin{lem}\label{powersalpha1-lemma}
For $l \in \mathbb{N}$, we have the following relation in $A_2(M_a(1))$
\begin{equation}\label{powersalpha1}
(\alpha(-1){\mathbf{1}})^l \equiv_2 (-1)^l\alpha(-4)^l{\mathbf{1}}.
\end{equation}
That is, in $A_2(M_a(1))$ we have  $(\alpha(-1) {\bf 1}  + O_2(M_a(1))^l  = (-1)^l\alpha(-4)^l{\mathbf{1}}+ O_2(M_a(1)$.
\end{lem}

\begin{proof}
We proceed by induction on $l$.  Eqn.\ (\ref{powersalpha1}) is clearly true for $l=0$. Assuming Eqn.\ (\ref{powersalpha1}) holds for some $l$, applying \eqref{multiplication-one-one}, the induction hypothesis, and Eqns.\ (\ref{45=44}) and (\ref{34=44}), we have 
\begin{eqnarray*}
 (\alpha(-1){\mathbf{1}})^{l+1} &=& \alpha(-1) \mathbf{1} *_2 (\alpha(-1) \mathbf{1})^l   \ \equiv_2  \ \alpha(-1) \mathbf{1} *_2 ( (-1)^l\alpha(-4)^l{\mathbf{1}})  \\
& = & (-1)^l  \left( 10 \alpha(-3)  + 15 \alpha(-4) + 6 \alpha(-5) \right) \alpha(-4)^l \mathbf{1} \\
%&=& (-1)^l10\alpha(-3)\alpha(-4)^l{\bf{1}} + (-1)^l15\alpha(-4)^{l+1}{\mathbf{1}}+  (-1)^l6\alpha(-5) \alpha(-4)^l{\mathbf{1}} \\
& \ \ \equiv_2& (-1)^{l+1} 10 \alpha(-4)^{l+1} {\bf{1}} + (-1)^l15\alpha(-4)^{l+1}{\mathbf{1}}+  (-1)^{l+1} 6 \alpha(-4)^{l+1}{\mathbf{1}} \\
&=& (-1)^{l+1} \alpha(-4)^{l+1}{\mathbf{1}}.
\end{eqnarray*}
\end{proof}

\begin{lem}\label{powersofalpha345}
For all $k,l,m\in \mathbb{N},$ 
\begin{equation*}\alpha(-3)^k\alpha(-4)^l\alpha(-5)^m{\bf{1}}\equiv_2(-1)^{k+m}\alpha(-4)^{k+l+m}{\bf{1}}\equiv_2(-1)^l(\alpha(-1){\bf1})^{k+l+m}.
\end{equation*}
\end{lem} 
  
\begin{proof}
The second equivalence follows from Lemma \ref{powersalpha1-lemma}, so we need only prove the first equivalence. To prove the first equivalence, we  first  fix $m=0$ and $l\in \mathbb{N}$ and induct on $k$, thereby proving the equivalence for $m=0$ and for all $k,l\in \mathbb{N}$. We then fix $k, l\in \mathbb{N}$ and induct on $m$, proving the equivalence for all $k, l, m\in \mathbb{N}$.

Let $m = 0$ and fix $l \in \mathbb{N}$.  Then the first equivalence is trivially true in the 
case that $k=0$, and follows from Eqn.\ \eqref{34=44} in the case that $k = 1$. Now fix 
$k>1$ and assume the first equivalence holds for $m=0$, $l \in \mathbb{N}$, and  $k' \in \mathbb{N}$ 
with $k'<k$. By Eqn.\  \eqref{O21} applied to $v = \alpha(-3)^{k-2} \alpha(-4)^l {\bf 1}$,   using Eqn.\ \eqref{reduce-equation-Heisenberg-level2}, and then applying the inductive assumption,  we have  
\begin{eqnarray*}
\lefteqn{\alpha(-3)^k\alpha(-4)^l{\bf 1} }\\
&\sim_2& - 3\alpha(-3)^{k-1}\alpha(-4)^{l+1}{\bf 1} -\alpha(-3)^{k-1}\alpha(-4)^l\alpha(-5){\bf 1}\\
& & \quad - \, 2\alpha(-3)^{k-2}\alpha(-4)^{l+2}{\bf 1} -\alpha(-3)^{k-2}\alpha(-4)^{l+1}\alpha(-5){\bf 1}\\
&\equiv_2& - 3\alpha(-3)^{k-1}\alpha(-4)^{l+1}{\bf 1} + \frac{1}{4(l+1)}\Big( \big(3(k-1)+4(l+1) \big)\alpha(-3)^{k-1}\alpha(-4)^{l+1}{\bf 1} \\
& & \quad + \,  3(k-1)\alpha(-3)^{k-2}\alpha(-4)^{l+2}{\bf 1} \Big)  -2\alpha(-3)^{k-2}\alpha(-4)^{l+2}{\bf 1}\\
& & \quad + \, \frac{1}{4(l+2)} \Big( \big(3(k-2)+4(l+2) \big)\alpha(-3)^{k-2}\alpha(-4)^{l+2}{\bf1} \\
& & \quad + \,  3(k-2)\alpha(-3)^{k-3}\alpha(-4)^{l+3}{\bf 1}\Big)\\
%&\equiv_2 & - 3(-1)^{k-1} \alpha(-4)^{k+l}{\bf 1} + \frac{1}{4(l+1)}\Big( \big(3(k-1)+4(l+1) \big)(-1)^{k-1}\alpha(-4)^{k+ l}{\bf 1} \\
%& & \quad  + \,  3(k-1)(-1)^k \alpha(-4)^{k+ l}{\bf 1} \Big)  -2(-1)^k \alpha(-4)^{k+ l}{\bf 1}\\
%& & \quad + \, \frac{1}{4(l+2)} \Big( \big(3(k-2)+4(l+2) \big)(-1)^{k}\alpha(-4)^{k+ l}{\bf1} \\
%& & \quad + \,  3(k-2)\alpha(-1)^{k-3}\alpha(-4)^{k+l}{\bf 1}\Big)\\
&\equiv_2& (-1)^k \alpha(-4)^{k+l} {\bf 1},
\end{eqnarray*}
proving the result in the case $m = 0$, and $k,l \in \mathbb{N}$.  

Next fix $m \in  \mathbb{Z}_+$, and assume the equivalence is true for all $m'<m$ and $k,l\in \mathbb{N}$.  Then applying Eqn.\ \eqref{reduce-equation-Heisenberg-level2}  and then applying this most recent inductive assumption, we have  
\begin{eqnarray*}
\lefteqn{\alpha(-3)^k\alpha(-4)^l\alpha(-5)^m{\bf 1} }\\
&\equiv_2& -\frac{1}{4(l+1)} \Big( \big(3k + 4l + 5m -1 \big) \alpha(-3)^k \alpha(-4)^{l+1} \alpha(-5)^{m-1} {\bf 1} \\
& & \quad - \,  5(m-1) \big( \alpha(-3) + 3\alpha(-4) + 3 \alpha(-5) \big) \alpha(-3)^k \alpha(-4)^{l+1} \alpha(-5)^{m-2} {\bf 1}  \\
& & \quad + \, 3k \alpha(-3)^{k-1} \alpha(-4)^{l+2} \alpha(-5)^{m-1} {\bf 1} \Big) \\
&\equiv_2&  -\frac{1}{4(l+1)} \Big( \big(3k + 4l + 5m -1 \big) (-1)^{k+m-1} \alpha(-4)^{k+l  + m } {\bf 1} \\
& & \quad - \,  5(m-1) \big( (-1)^{k + m - 1} + 3(-1)^{k + m - 2}  + 3(-1)^{k + m - 1} \big) \alpha(-4)^{k + l + m} {\bf 1}  \\
& & \quad + \, 3k (-1)^{k + m -2} \alpha(-4)^{k + l+m} {\bf 1} \Big) \\
&=&(-1)^{k+m}\alpha(-4)^{k+l+m}{\bf 1},
\end{eqnarray*} 
proving the result. 
\end{proof}

\subsection{Reducing generators further}

We are now prepared to reduce the generators further than the generating set given in Proposition \ref{generators-prop} in a series of reduction steps.

\begin{prop}\label{reduce-301}
$A_2(V)$ is generated by elements of the form $v + O_2(V)$ for $v$ in the set
\[ \{ \alpha(-1)^i\alpha(-2)^j\alpha(-3)^k\alpha(-4)^l{\mathbf{1}} \, | \, 0\le i\le 2, \, 0\le j\le 1, \, 0\le k\le 1, \, l\in \mathbb{N} \  \mbox{with $(i,j) \neq (0,0)$} \} . \]
\end{prop}

\begin{proof} For the purposes of this proof, we call the set specified in the proposition $S$.  By Proposition \ref{generators-prop} and Lemma \ref{powersofalpha345}, 
we only need to prove that elements in $V$ of the form $
\alpha(-1)^i\alpha(-2)^j\alpha(-3)^{k+2} \alpha(-4)^l\mathbf{1}$ are generated in $A_2(V)$ by $S$ for the following cases 
\begin{enumerate}
\item $i=0$, $j=1$,
\item $i=1$, $j=0$,
\item $i=j=1$,
\item $i=2$, $j=0$,
\item $i=2$, $j=1$,
\end{enumerate} 
and $k,l \in \mathbb{N}$.  

By Eqn.\ \eqref{O21} applied to $v = \alpha(-1)^i\alpha(-2)^j\alpha(-3)^k\alpha(-4)^l\mathbf{1}$ for $i,j,k,l \in \mathbb{N}$, and then applying \eqref{reduce-equation-Heisenberg-level2}, we have
\begin{eqnarray}\label{redpowersofalpha3}
\lefteqn{\ \ \alpha(-1)^i\alpha(-2)^j\alpha(-3)^{k+2}\alpha(-4)^l \mathbf{1}}\\
\ \ \ &\equiv_2& \! \!  -3\alpha(-1)^i\alpha(-2)^j\alpha(-3)^{k+1}\alpha(-4)^{l+1}\mathbf{1}   - \alpha(-1)^i\alpha(-2)^j\alpha(-3)^{k+1}\alpha(-4)^l\alpha(-5) \mathbf{1} \nonumber \\
& & \  - \, 2\alpha(-1)^i\alpha(-2)^j\alpha(-3)^k\alpha(-4)^{l+2} \mathbf{1}   -  \alpha(-1)^i\alpha(-2)^j\alpha(-3)^k\alpha(-4)^{l+1}\alpha(-5) \mathbf{1} \nonumber\\
&\equiv_2& \! \!  -3\alpha(-1)^{i}\alpha(-2)^{j}\alpha(-3)^{k+1}\alpha(-4)^{l+1}{\bf 1} \nonumber \\
& & \  + \, \frac{1}{4(l+1)}\Big((i+2j+3(k+1)+4(l+1))\alpha(-1)^i\alpha(-2)^j\alpha(-3)^{k+1}\alpha(-4)^{l+1}{\bf 1}\nonumber \\
& & \  + \, 3(k+1)\alpha(-1)^i\alpha(-2)^j\alpha(-3)^k\alpha(-4)^{l+2}{\bf 1} \nonumber \\
& & \   + \, 2j\alpha(-1)^i\alpha(-2)^{j-1}\alpha(-3)^{k+2}\alpha(-4)^{l+1}{\bf 1}\nonumber\\
& & \  + \, i\alpha(-1)^{i-1}\alpha(-2)^{j+1}\alpha(-3)^{k+1}\alpha(-4)^{l+1}{\bf 1} \Big)  \nonumber \\
& & \  - \, 2\alpha(-1)^i\alpha(-2)^j\alpha(-3)^k\alpha(-4)^{l+2}{\bf 1} \nonumber \\
& & \  + \, \frac{1}{4(l+2)}\Big( (i+2j+3k+4(l+2))\alpha(-1)^i\alpha(-2)^j\alpha(-3)^{k}\alpha(-4)^{l+2}{\bf 1} \nonumber \\
& &\  + \,  3k\alpha(-1)^i\alpha(-2)^j\alpha(-3)^{k-1}\alpha(-4)^{l+3}{\bf 1}  \nonumber \\
& & \  + \,  2j\alpha(-1)^i\alpha(-2)^{j-1}\alpha(-3)^{k+1}\alpha(-4)^{l+2}{\bf 1} \nonumber\\
& & \  + \,  i\alpha(-1)^{i-1}\alpha(-2)^{j+1}\alpha(-3)^{k}\alpha(-4)^{l+2}{\bf 1} \Big).\nonumber
\end{eqnarray}
 
Case (1):  Eqn.\ \eqref{redpowersofalpha3} together with Lemma \ref{powersofalpha345} implies  $\alpha(-2)\alpha(-3)^{k+2}\alpha(-4)^l \mathbf{1}$ is equivalent to a linear combination of elements in $S$ or elements of the form $\alpha(-2)\alpha(-3)^{k'}\alpha(-4)^l \mathbf{1}$ with $k'<k+2$.  The result for this case follows by induction on $k$. 

Case (2):  Eqn.\ \eqref{redpowersofalpha3} implies  $\alpha(-1)\alpha(-3)^{k+2}\alpha(-4)^l \mathbf{1}$ is equivalent to a linear combination of elements in $S$, elements of the form $\alpha(-1)\alpha(-3)^{k'}\alpha(-4)^l \mathbf{1}$ with $k'<k+2$, or elements of the form of Case (1).  The result for this case follows by induction on $k$. 

Case (3): Eqn.\ \eqref{redpowersofalpha3} implies  $\alpha(-1)\alpha(-2)\alpha(-3)^{k+2}\alpha(-4)^l \mathbf{1}$ is equivalent to a linear combination of elements in $S$,  elements of the form of Case (2), elements of the form $\alpha(-1)\alpha(-2)\alpha(-3)^{k'}\alpha(-4)^l \mathbf{1}$ for $k'<k+2$, or elements of the form $\alpha(-2)^2 \alpha(-3)^{k'} \alpha(-4)^{l'} {\bf 1}$ for $k'<k+2$ and $l'\in \mathbb{N}$.  Thus we are done by induction on $k$ and by showing these last terms are generated by $S$. 

To show that  $\alpha(-2)^2 \alpha(-3)^{k'} \alpha(-4)^{l'} {\bf 1}$ is generated by $S$, we apply Eqn.\ \eqref{O23} to $v =  \alpha(-3)^{k'} \alpha(-4)^{l'} {\bf 1}$  and then use  \eqref{reduce-equation-Heisenberg-level2}, to obtain 
\begin{eqnarray*}
\lefteqn{\alpha(-2)^2\alpha(-3)^{k'}\alpha(-4)^{l'}{\bf 1} } \\
%&\equiv_2&  - 6\alpha(-2)\alpha(-3)^{k'+1}\alpha(-4)^{l'}{\bf 1}-6\alpha(-2)\alpha(-3)^{k'}\alpha(-4)^{l'+1}{\bf 1}\\
%& & \quad - \,  2\alpha(-2)\alpha(-3)^{k'}\alpha(-4)^{l'}\alpha(-5){\bf 1}-6\alpha(-3)^{k'+2}\alpha(-4)^{l'}{\bf 1}\\
%& & \quad  - \, 8\alpha(-3)^{k'+1}\alpha(-4)^{l'+1}{\bf 1}-2\alpha(-3)^{k'+1}\alpha(-4)^{l'}\alpha(-5){\bf 1}\\
%& & \quad - \,  \alpha(-3)^{k'}\alpha(-4)^{l'+2}{\bf 1}\\
&\equiv_2& -6\alpha(-2)\alpha(-3)^{k'+1}\alpha(-4)^{l'}{\bf 1}-6\alpha(-2)\alpha(-3)^{k'}\alpha(-4)^{l'+1}{\bf 1}\\
& &\quad + \,  \frac{1}{2(l'+1)} \Big((2+3k'+4(l'+1))  \alpha(-2)\alpha(-3)^{k'}\alpha(-4)^{l'+1}{\bf 1}  \\
& & \quad + \,  3k'\alpha(-2)\alpha(-3)^{k'-1}\alpha(-4)^{l'+2}{\bf 1} + 2\alpha(-3)^{k'+1}\alpha(-4)^{l'+1}{\bf 1}    \Big) \\
& & \quad - \, 6\alpha(-3)^{k'+2}\alpha(-4)^{l'}{\bf 1}  - 8\alpha(-3)^{k'+1}\alpha(-4)^{l'+1}{\bf 1} \\
& & \quad - \,  2\alpha(-3)^{k'+1}\alpha(-4)^{l'}\alpha(-5){\bf 1} - \alpha(-3)^{k'}\alpha(-4)^{l'+2}{\bf 1}.
\end{eqnarray*} 
Applying Lemma \ref{powersofalpha345} to this expression for $\alpha(-2)^2 \alpha(-3)^{k'} \alpha(-4)^{l'} {\bf 1}$, we see that all terms in this case are  on the righthand side above are in $S$  or reduce to terms in Case (1).

Case (4):  Eqn.\ \eqref{redpowersofalpha3} implies  $\alpha(-1)^2 \alpha(-3)^{k+2}\alpha(-4)^l \mathbf{1}$ is equivalent to a linear combination of elements in $S$, elements of the form $\alpha(-1)^2\alpha(-3)^{k'}\alpha(-4)^l \mathbf{1}$ with $k'<k+2$, or elements of the form of Case (3).  The result for this case follows by induction on $k$.  

Case (5): Eqn.\ \eqref{redpowersofalpha3} implies  $\alpha(-1)^2\alpha(-2)\alpha(-3)^{k+2}\alpha(-4)^l \mathbf{1}$ is equivalent to a linear combination of elements in $S$,  elements of the form of Case (4), elements of the form $\alpha(-1)^2\alpha(-2)\alpha(-3)^{k'}\alpha(-4)^l \mathbf{1}$ for $k'<k+2$, or of the form $\alpha(-1) \alpha(-2)^2 \alpha(-3)^{k'} \alpha(-4)^{l'} {\bf 1}$ for $k'<k+2$ and $l'\in \mathbb{N}$.  Thus we are done by induction on $k$ and by showing these last terms are generated by $S$. 

To show that  $\alpha(-1) \alpha(-2)^2 \alpha(-3)^{k'} \alpha(-4)^{l'} {\bf 1}$ is generated by $S$, we apply Eqn.\ \eqref{O23} to $v = \alpha(-1) \alpha(-3)^{k'} \alpha(-4)^{l'} {\bf 1}$  and then use  \eqref{reduce-equation-Heisenberg-level2}, to obtain
\begin{eqnarray}
\lefteqn{\alpha(-1)\alpha(-2)^2\alpha(-3)^{k'}\alpha(-4)^{l'}{\bf 1} }\\
%&\equiv_2&  - 6\alpha(-1)\alpha(-2)\alpha(-3)^{k'+1}\alpha(-4)^{l'}{\bf 1} - 6\alpha(-1)\alpha(-2)\alpha(-3)^{k'}\alpha(-4)^{l'+1}{\bf 1} \nonumber \\
%& &\quad - \, 2\alpha(-1)\alpha(-2)\alpha(-3)^{k'}\alpha(-4)^{l'}\alpha(-5){\bf 1} - 6\alpha(-1)\alpha(-3)^{k'+2}\alpha(-4)^{l'}{\bf 1} \nonumber \\
%& & \quad - \, 8\alpha(-1)\alpha(-3)^{k'+1}\alpha(-4)^{l'+1}{\bf 1} - 2\alpha(-1)\alpha(-3)^{k'+1}\alpha(-4)^{l'}\alpha(-5){\bf 1} \nonumber \\
%& & \quad - \, \alpha(-1)\alpha(-3)^{k'}\alpha(-4)^{l'+2}{\bf 1} \nonumber \\
&\equiv_2& - 6\alpha(-1)\alpha(-2)\alpha(-3)^{k'+1}\alpha(-4)^{l'}{\bf 1}-6\alpha(-1)\alpha(-2)\alpha(-3)^{k'}\alpha(-4)^{l'+1}{\bf 1} \nonumber \\
& & \quad + \, \frac{1}{2(l'+1)} \Big((3+3k'+4(l'+1))\alpha(-1)\alpha(-2)\alpha(-3)^{k'}\alpha(-4)^{l'+1}{\bf 1} \nonumber \\
& & \quad + \,  3k' \alpha(-1) \alpha(-2) \alpha(-3)^{k'-1} \alpha(-4)^{l'+2} {\bf 1}  + 2\alpha(-1)\alpha(-3)^{k'+1}\alpha(-4)^{l'+1}{\bf 1}  \nonumber \\
& & \quad  + \,  \alpha(-2)^2\alpha(-3)^{k'}\alpha(-4)^{l'+1}{\bf 1}   \Big) - 6\alpha(-1)\alpha(-3)^{k'+2}\alpha(-4)^{l'}{\bf 1} \nonumber \\
& & \quad - \, 8\alpha(-1)\alpha(-3)^{k'+1}\alpha(-4)^{l'+1}{\bf 1} \nonumber \\
& &\quad   + \, \frac{1}{2(l'+1)} \Big((1+3(k'+1)+4(l'+1))\alpha(-1)\alpha(-3)^{k'+1}\alpha(-4)^{l'+1}{\bf 1}  \nonumber \\
& & \quad + \,  3(k'+1)\alpha(-1)\alpha(-3)^{k'}\alpha(-4)^{l'+2}{\bf 1}  + \alpha(-2)\alpha(-3)^{k'+1}\alpha(-4)^{l'+1} {\bf 1} \Big)  \nonumber \\
& & \quad -  \, \alpha(-1)\alpha(-3)^{k'}\alpha(-4)^{l'+2}{\bf 1}, \nonumber
\end{eqnarray}
and all terms reduce to terms in Cases (1), (2) or (3), or terms appearing in subcases of (3) already proved to be generated by $S$.
\end{proof}

We are now ready to prove the following refinement of the previous proposition.

\begin{prop}\label{reduce-300}
$A_2(V)$ is generated by elements of the form $v + O_2(V)$ for $v$ in the set
\[ \{\alpha(-1)^i\alpha(-2)^j\alpha(-4)^l{\mathbf{1}} \, | \, 0\le i\le 2, \ 0\le j\le 1,\ (i,j) \neq (0,0),  \ l\in \mathbb{N}\} . \]
\end{prop}

\begin{proof}  For the purposes of this proof, we call the set in the proposition $S$.  
By Proposition \ref{reduce-301}, we need only show that elements of the form $\alpha(-1)^i\alpha(-2)^j\alpha(-3)\alpha(-4)^l{\bf 1}$, for  $0\le i\le 2$,  $0\le j\le 1$, with $(i,j) \neq (0,0)$, and $l\in \mathbb{N}$, are generated in $A_2(V)$ by elements of the form $v+O_2(V)$ for $v\in S$.  Applying  Eqn.\  \eqref{multiplication-one-one}  for $v =  \alpha(-1)^i\alpha(-2)^j\alpha(-4)^l{\bf 1}$, and then Eqn.\ 
\eqref{reduce-equation-Heisenberg-level2},  we have
\begin{eqnarray}\label{red-301b}
\lefteqn{ \alpha(-1)^i\alpha(-2)^j\alpha(-3)\alpha(-4)^l{\bf 1}} \\
%&\equiv_2&  \frac{1}{10} \big(\alpha(-1)   \mathbf{1} *_2 \alpha(-1)^i\alpha(-2)^j\alpha(-4)^l{\bf 1}- 15\alpha(-1)^i\alpha(-2)^j\alpha(-4)^{l+1}{\bf 1} \nonumber \\
%&& \quad  - \, 6\alpha(-1)^i\alpha(-2)^j\alpha(-4)^l\alpha(-5){\bf 1}\big) \nonumber \\
&\equiv_2& \frac{1}{10} \Big(\alpha(-1)   \mathbf{1} *_2 \alpha(-1)^i\alpha(-2)^j\alpha(-4)^l{\bf 1}-15\alpha(-1)^i\alpha(-2)^j\alpha(-4)^{l+1}{\bf 1} \nonumber \\
&& \quad   +  \, 6\frac{1}{4(l+1)} \big(i\alpha(-1)^{i-1}\alpha(-2)^{j+1}\alpha(-4)^{l+1}{\bf 1} \nonumber \\
& & \quad  +  \, 2j\alpha(-1)^i\alpha(-2)^{j-1}\alpha(-3)\alpha(-4)^{l+1}{\bf 1}   \nonumber \\
& & \quad  +  \,  (i+2j+4(l+1))\alpha(-1)^i\alpha(-2)^j\alpha(-4)^{l+1}{\bf 1} \big)  \Big) . \nonumber
%= \left( \binom{3}{ 2}  - 3  \binom{3 }{3} + 6  \binom{3 }{4}   \right) \alpha(-1)  v    +   \left( \binom{3}{ 1}  - 3  \binom{3 }{2} + 6  \binom{3 }{3}   \right) \alpha(-2)  v  +     \left( \binom{3}{ 0}  - 3  \binom{3 }{1} + 6  \binom{3 }{2}   \right) \alpha(-3)  v  +    \left( \binom{3}{ -1}  - 3  \binom{3 }{0} + 6  \binom{3 }{1}   \right) \alpha(-4)  v + \left( \binom{3}{ -2}  - 3  \binom{3 }{-1} + 6  \binom{3 }{0}   \right) \alpha(-5)  v + f_{0}(v) , 
%=\\
%=  10 \alpha(-3)  v  + 15 \alpha(-4)  v + 6 \alpha(-5)  v,
\end{eqnarray}

We have the following cases 
\begin{enumerate}
\item $0< i\le 2$, $j=0$,
\item $i=0$, $j = 1$,
%\item $0<i\le 2$, $j=1$,
\item $i=j=1$,
\item $i=2$, $j=1$.
\end{enumerate}

Eqn.\ \eqref{red-301b} directly implies  that elements in Case (1) are generated by elements in $S$.  Case (2)  follows from Eqn.\  \eqref{red-301b} by observing that   the only term in the righthand side of \eqref{red-301b} that is not  already in $S$ is the term $2j \alpha(-1)^i\alpha(-2)^{j-1}\alpha(-3)\alpha(-4)^{l+1}{\bf 1}$, when $j = 1$, and in this case we are assuming $i = 0$, so that this term is $2 \alpha(-3) \alpha(-4)^{l+1} {\bf 1}$.  However, by Eqn.\ (\ref{34=44}) in Lemma \ref{helprel}, this is  equivalent to $-2\alpha(-4)^{l+2}{\bf 1}$ which is in $S$.

For elements  in  Case (3), all elements on the righthand side of  Eqn.\  \eqref{red-301b}, are in $S$ or are generated in $A_2(V)$ by $S$ via Case (1), except for the $\alpha(-2)^{2}\alpha(-4)^{l+1}{\bf 1}$ term, which is shown to be generated by $S$ by Eqn.\ \eqref{224} in Lemma \ref{helprel}.

It remains to handle Case (4). In this case, all elements on the righthand side of  Eqn.\ \eqref{red-301b}, are in $S$ or are generated in $A_2(V)$ by $S$ via Case (1), except for the following term: $\alpha(-1) \alpha(-2)^{2}\alpha(-4)^{l+1}{\bf 1}$.  To prove that this term is generated by $S$, we 
apply Eqn.\ \eqref{O23} to $v = \alpha(-1) \alpha(-4)^{l+1} {\bf 1}$ and then apply Eqn.\ \eqref{reduce-equation-Heisenberg-level2} giving 
\begin{eqnarray}\label{two-squared-four-reduce}
\lefteqn{ \alpha(-1)\alpha(-2)^2\alpha(-4)^{l+1}{\bf 1}}\\
%&\equiv_2& -  \, 6\alpha(-1)\alpha(-2)\alpha(-3)\alpha(-4)^{l+1}{\bf 1}-6\alpha(-1)\alpha(-2)\alpha(-4)^{l+2}{\bf 1} \nonumber \\
%& & \quad  -  \,  2\alpha(-1)\alpha(-2)\alpha(-4)^{l+1}\alpha(-5){\bf 1}-6\alpha(-1)\alpha(-3)^{2}\alpha(-4)^{l+1}{\bf 1} \nonumber \\
%& & \quad  - \, 8\alpha(-1)\alpha(-3)\alpha(-4)^{l+2}{\bf 1}-2\alpha(-1)\alpha(-3)\alpha(-4)^{l+1}\alpha(-5){\bf 1} \nonumber \\
%& & \quad - \, \alpha(-1)\alpha(-4)^{l+3}{\bf 1} \nonumber \\
&\equiv_2& - 6\alpha(-1)\alpha(-2)\alpha(-3)\alpha(-4)^{l+1}{\bf 1}-6\alpha(-1)\alpha(-2)\alpha(-4)^{l+2}{\bf 1} \nonumber \\
& & \quad   +  \, \frac{1}{2(l+2)} \big((3+4(l+2))\alpha(-1)\alpha(-2)\alpha(-4)^{l+2}{\bf 1}  \nonumber \\
& & \quad    + \,  2\alpha(-1)\alpha(-3)\alpha(-4)^{l+2}{\bf 1} + \alpha(-2)^2\alpha(-4)^{l+2}{\bf 1} \big)  \nonumber \\
& & \quad - \, 6\alpha(-1)\alpha(-3)^{2}\alpha(-4)^{l+1}{\bf 1}-8\alpha(-1)\alpha(-3)\alpha(-4)^{l+2}{\bf 1} \nonumber \\
& & \quad + \,  \frac{1}{2(l+2)} \big((4+4(l+2))\alpha(-1)\alpha(-3)\alpha(-4)^{l+2}{\bf 1} \nonumber \\
& & \quad +  \,  3\alpha(-1)\alpha(-4)^{l+3}{\bf 1}  + \alpha(-2)\alpha(-3)\alpha(-4)^{l+2}{\bf1}  \big)  \nonumber \\
& &  \quad  - \, \alpha(-1)\alpha(-4)^{l+3}{\bf 1} \nonumber .
\end{eqnarray}

The only term in the righthand side of   Eqn.\  \eqref{two-squared-four-reduce} that is not in $S$, or does not reduce to previous cases is the term  $\alpha(-1)\alpha(-3)^2\alpha(-4)^{l+1}{\bf 1}$.  We show this term is generated by $S$ by 
applying Eqn.\  \eqref{O21}  to $v = \alpha(-1) \alpha(-4)^{l+1} {\bf 1}$ and then applying Eqn.\ \eqref{reduce-equation-Heisenberg-level2}, to  give
\begin{eqnarray*}
\lefteqn{ \alpha(-1)\alpha(-3)^2\alpha(-4)^{l+1}{\bf 1}}\\
%&\equiv_2&- 3\alpha(-1)\alpha(-3)\alpha(-4)^{l+2}{\bf 1} -  2 \alpha(-1)\alpha(-4)^{l+3}{\bf 1} \\
%& & \quad  -  \, \alpha(-1)\alpha(-3)\alpha(-4)^{l+1}\alpha(-5){\bf 1}  - \alpha(-1)\alpha(-4)^{l+2}\alpha(-5){\bf 1}\\
&\equiv_2&- 3\alpha(-1)\alpha(-3)\alpha(-4)^{l+2}{\bf 1} - 2 \alpha(-1)\alpha(-4)^{l+3}{\bf 1} \\
& & \quad +  \, \frac{1}{4(l+2)} \big( (4+4(l+2))\alpha(-1)\alpha(-3)\alpha(-4)^{l+2}{\bf 1} \\
& & \quad + \, 3\alpha(-1)\alpha(-4)^{l+3}{\bf 1} + \alpha(-2)\alpha(-3)\alpha(-4)^{l+2}{\bf 1} \big) \\
& & \quad  + \,  \frac{1}{4(l+3)}\bigl( (1+4(l+3))\alpha(-1)\alpha(-4)^{l+3}{\bf1} + \alpha(-2)\alpha(-4)^{l+3}{\bf 1} \big).
\end{eqnarray*} 
All terms on the righthand side of this expression correspond to Case (1) or (2), which we have already considered, or are elements in $S$.  The result follows.
\end{proof}

To further reduce the set of generators, we first prove a lemma, which will also be helpful later when we prove what relations are satisfied by the generators.

\begin{lem}\label{eqnB}
For all $l\in \mathbb{N}$ 
\begin{eqnarray}\label{red124to1424}
\ \ \ 10\alpha(-1)\alpha(-2)\alpha(-4)^l\mathbf{1} 
&\equiv_2 & \frac{7l+5}{l+1}\alpha(-1)\alpha(-4)^{l+1}{\mathbf{1}}-3\alpha(-1){\mathbf{1}}*_2\alpha(-1)\alpha(-4)^l\mathbf{1}   \\
& & \quad - \, \frac{4l+3}{l+1}\alpha(-2)\alpha(-4)^{l+1}{\mathbf{1}}+\frac{10l+7}{l+1}\alpha(-4)^{l+2}\mathbf{1} \nonumber \\
& & \quad + \, 6 \alpha(-1) \mathbf{1} *_2\alpha(-2)\alpha(-4)^{l}\mathbf{1}. \nonumber 
\end{eqnarray}
\end{lem}

\begin{proof}
Applying Eqn.\ \eqref{O24} to $v = \alpha(-4)^l$, we have 
\begin{eqnarray*}
\lefteqn{\alpha(-1)\alpha(-2)\alpha(-4)^l\mathbf{1}}\\
&\equiv_2&-3\alpha(-1)\alpha(-3)\alpha(-4)^l\mathbf{1} -3 \alpha(-1)\alpha(-4)^{l+1}\mathbf{1} - \alpha(-1)\alpha(-4)^l\alpha(-5)\mathbf{1}\\
& & \quad - \, 2 \alpha(-2)^2\alpha(-4)^l\mathbf{1} - 6 \alpha(-2)\alpha(-3)\alpha(-4)^l\mathbf{1 } - 4 \alpha(-2)\alpha(-4)^{l+1} \\
& & \quad - \, \alpha(-2)\alpha(-4)^l\alpha(-5)\mathbf{1}-2\alpha(-3)^2\alpha(-4)^l\mathbf{1} - \alpha(-3)\alpha(-4)^{l+1}\mathbf{1}.
 \end{eqnarray*}
Applying Eqns.\ \eqref{134}, \eqref{145}, \eqref{224}, \eqref{234}, \eqref{245}, \eqref{3^24}, and \eqref{34=44}, we have  
\begin{eqnarray*}
\lefteqn{10\alpha(-1)\alpha(-2)\alpha(-4)^l\mathbf{1}}\\
&\equiv_2& -3\Big(\alpha(-1)\mathbf{1}*_2\alpha(-1)\alpha(-4)^l\mathbf{1}-\frac{3(6l+5)}{2(l+1)} \alpha(-1)\alpha(-4)^{l+1}\mathbf{1}\\
& & \quad + \,  \frac{3}{2(l+1)}\alpha(-2)\alpha(-4)^{l+1}\mathbf{1}\Big) -  30\alpha(-1)\alpha(-4)^{l+1}\mathbf{1}\\
& & \quad + \, 10\Big(\frac{1}{4(l+1)} \big(\alpha(-2)\alpha(-4)^{l+1}\mathbf{1} + (1+4(l+1))\alpha(-1)\alpha(-4)^{l+1}\mathbf{1} \big)\Big)\\
& & \quad - \,  20\Big(-\frac{3}{5}\alpha(-1){\mathbf{1}}*_2\alpha(-2)\alpha(-4)^l \mathbf{1} +\frac{7l+3}{5(l+1)}\alpha(-2)\alpha(-4)^{l+1}\mathbf{1}\\
& & \quad - \, \frac{(5l+1)}{5(l+1)}\alpha(-4)^{l+2}\mathbf{1}\Big)  - 6\Big(\alpha(-1)\mathbf{1}*_2\alpha(-2)\alpha(-4)^l{\mathbf{1}} \\
& & \quad - \,  \frac{3(3l+2)}{l+1} \alpha(-2)\alpha(-4)^{l+1}\mathbf{1} - \frac{3}{l+1}\alpha(-4)^{l+2}\mathbf{1}\Big) - 40\alpha(-2)\alpha(-4)^{l+1}\mathbf{1}\\
& & \quad - \,  10\Big(\frac{2}{4(l+1)}\alpha(-4)^{l+2}\mathbf{1}-\frac{(1+2(l+1))}{2(l+1)}\alpha(-2)\alpha(-4)^{l+1}\mathbf{1}\Big)\\
& & \quad - \,  20\alpha(-4)^{l+2}{\mathbf{1}}+10\alpha(-4)^{l+2}\mathbf{1}.
\end{eqnarray*}

Simplifying this expression by combining like terms proves the lemma.
\end{proof}

Lemma \ref{eqnB}  along with Proposition \ref{reduce-300} immediately implies the following proposition: 

\begin{prop}\label{reduce-no11}
$A_2(V)$ is generated by elements of the form $v + O_2(V)$ for $v$ in the set
\[ \{\alpha(-1)^i\alpha(-2)^j\alpha(-4)^l{\mathbf{1}} \, | \, 0\le i\le 2, \, 0\le j\le 1,\, (i,j)\ne  (0,0), (1,1),\, l\in \mathbb{N}\} . \]
\end{prop}

To further reduce the set of generators, we prove the following lemma.

\begin{lem}\label{eqnC} 
For  $l \in \mathbb{N}$
\begin{eqnarray}\label{eqnLemC}
\ \ \ \ \ \ \ \frac{3(4l+1)}{10(l+1)}\alpha(-2)\alpha(-4)^{l+1} \mathbf{1}
\! \! & \equiv_2&  \! \! \frac{(2l+1)}{2(l+1)}\alpha(-1)\alpha(-4)^{l+1}\mathbf{1}  - \frac{5l+12}{5(l+1)}\alpha(-4)^{l+2}\mathbf{1} \\
& &\    + \, \alpha(-1){\mathbf{1}}*_2\alpha(-1)\alpha(-4)^l\mathbf{1}  + \alpha(-1)^2 \mathbf{1} *_2\alpha(-4)^{l}\mathbf{1}  \nonumber \\
& & \    - \, \frac{6}{5}\alpha(-1) \mathbf{1}*_2\alpha(-2)\alpha(-4)^l\mathbf{1}. \nonumber
\end{eqnarray}
\end{lem}

\begin{proof}
By the definition of multiplication in $A_2(V)$, Eqn.\  (\ref{*_n-definition}), and the recursion Eqn.\ (\ref{recursion-level-2}), we have 
\begin{eqnarray*}
\lefteqn{\alpha(-1)^2{\mathbf{1}}*_2\alpha(-4)^l\mathbf{1}}\\
&=&  \res_x  \left( \frac{1}{x^3}  -  \frac{3}{x^4} +  \frac{6}{x^5} \right) (1+x)^4 Y(\alpha(-1)^2{\mathbf{1}}, x) \alpha(-4)^l\mathbf{1} \\
%&=& \Big( 6\big(2 \alpha(-1) \alpha(-5) + 2 \alpha(-2) \alpha(-4) + \alpha(-3)^2 + 2 \alpha(-10) \alpha(4) \big) + \, 21 \big( 2 \alpha(-1) \alpha(-4) \\
%& & \quad +\,  2 \alpha(-2) \alpha(-3) + 2 \alpha(-9)\alpha(4) \big) + 25 \big( 2 \alpha(-1) \alpha(-3) + \alpha(-2)^2 + 2 \alpha(-8) \alpha(4) \big) \\
%& & \quad + \, 10 \big( 2 \alpha(-1) \alpha(-2) + 2 \alpha(-7) \alpha(4) \big) +  2 \alpha(-5) \alpha(4) + 2 \alpha(-4) \alpha(4) \Big) \alpha(-4)^l {\bf 1} \\
%&=& 12 \alpha(-1)\alpha(-4)^l  \alpha(-5){\bf 1}  + 12 \alpha(-2) \alpha(-4)^{l+1} {\bf 1}  +  6 \alpha(-3)^2\alpha(-4)^l {\bf 1}  \\
%& & \quad + 12 \alpha(-10) \alpha(4)\alpha(-4)^l {\bf 1}  + 42  \alpha(-1) \alpha(-4)^{l+1} {\bf 1}  + 42  \alpha(-2) \alpha(-3)\alpha(-4)^l {\bf 1} \\
%& & \quad  + \,  42 \alpha(-9)\alpha(4)\alpha(-4)^l {\bf 1}  +  50  \alpha(-1) \alpha(-3)\alpha(-4)^l {\bf 1}  + 25 \alpha(-2)^2\alpha(-4)^l {\bf 1}  \\
%& & \quad + \,  50 \alpha(-8) \alpha(4)\alpha(-4)^l {\bf 1}  + 20 \alpha(-1) \alpha(-2)\alpha(-4)^l {\bf 1}  + 20 \alpha(-7) \alpha(4)\alpha(-4)^l {\bf 1}  \\
%& & \quad + \, 2 \alpha(-5) \alpha(4) \alpha(-4)^l {\bf 1} + 2 \alpha(-4) \alpha(4) \alpha(-4)^l {\bf 1} \\
&=& 12 \alpha(-1)\alpha(-4)^l  \alpha(-5){\bf 1}  + 12 \alpha(-2) \alpha(-4)^{l+1} {\bf 1}  + 6 \alpha(-3)^2\alpha(-4)^l {\bf 1}  \\
& & \quad  +  \, 48l \alpha(-10)\alpha(-4)^{l-1}  {\bf 1}  + 42  \alpha(-1) \alpha(-4)^{l+1} {\bf 1}  + 42  \alpha(-2) \alpha(-3)\alpha(-4)^l {\bf 1}  \\
& & \quad +  \, 168l  \alpha(-9)\alpha(-4)^{l-1} {\bf 1}  + 50  \alpha(-1) \alpha(-3)\alpha(-4)^l {\bf 1}  + 25 \alpha(-2)^2\alpha(-4)^l {\bf 1}  \\
& & \quad +  \, 200l \alpha(-8) \alpha(-4)^{l-1} {\bf 1}  + 20 \alpha(-1) \alpha(-2)\alpha(-4)^l {\bf 1}  + 80l \alpha(-7) \alpha(-4)^{l-1} {\bf 1}  \\
& & \quad + \, 8l  \alpha(-5) \alpha(-4)^{l-1} {\bf 1} + 8l \alpha(-4)^{l} {\bf 1} \\
&\sim_2& 12 \alpha(-1)\alpha(-4)^l  \alpha(-5){\bf 1}  + 12 \alpha(-2) \alpha(-4)^{l+1} {\bf 1}  + 6 \alpha(-3)^2\alpha(-4)^l {\bf 1}  - 48l \big(15\alpha(-3) \\
& & \quad + \,  35\alpha(-4) + 21\alpha(-5)  \big) \alpha(-4)^{l-1}  {\bf 1}  + 42  \alpha(-1) \alpha(-4)^{l+1} {\bf 1}  + 42  \alpha(-2) \alpha(-3)\alpha(-4)^l {\bf 1} \\
& & \quad  +  \, 168l  \big(10 \alpha(-3) + 24\alpha(-4) + 15 \alpha(-5)    \big)\alpha(-4)^{l-1} {\bf 1}  + 50  \alpha(-1) \alpha(-3)\alpha(-4)^l {\bf 1}  \\
& & \quad + \,  25 \alpha(-2)^2\alpha(-4)^l {\bf 1}  - 200l \big( 6 \alpha(-3) + 15 \alpha(-4) + 10 \alpha(-5)   \big) \alpha(-4)^{l-1} {\bf 1}  \\
& & \quad + \, 20 \alpha(-1) \alpha(-2)\alpha(-4)^l {\bf 1}  + 80l \big(3 \alpha(-3) + 8 \alpha(-4) + 6 \alpha(-5)   \big) \alpha(-4)^{l-1} {\bf 1}  \\
& & \quad + \, 8l  \alpha(-5) \alpha(-4)^{l-1} {\bf 1} + 8l \alpha(-4)^{l} {\bf 1} \\
%&=& 12 \alpha(-1)\alpha(-4)^l  \alpha(-5){\bf 1}  + 12 \alpha(-2) \alpha(-4)^{l+1} {\bf 1}  + 6 \alpha(-3)^2\alpha(-4)^l {\bf 1} \\
%& & \quad + \,42  \alpha(-1) \alpha(-4)^{l+1} {\bf 1}  + 42  \alpha(-2) \alpha(-3)\alpha(-4)^l {\bf 1}  +  50  \alpha(-1) \alpha(-3)\alpha(-4)^l {\bf 1}  \\
%& & \quad + \,  25 \alpha(-2)^2\alpha(-4)^l {\bf 1}  +  20 \alpha(-1) \alpha(-2)\alpha(-4)^l {\bf 1}   - 48l \big(15\alpha(-3) + 35\alpha(-4) \\
%& & \quad + \,  21\alpha(-5)  \big) \alpha(-4)^{l-1}  {\bf 1}  + 168l  \big(10 \alpha(-3) + 24\alpha(-4) + 15 \alpha(-5)    \big)\alpha(-4)^{l-1} {\bf 1}  \\
%& & \quad - \,  200l \big( 6 \alpha(-3) + 15 \alpha(-4) + 10 \alpha(-5)   \big) \alpha(-4)^{l-1} {\bf 1}  + 80l \big(3 \alpha(-3) + 8 \alpha(-4) \\
%& & \quad + \,  6 \alpha(-5)   \big) \alpha(-4)^{l-1} {\bf 1}  + 8l  \alpha(-5) \alpha(-4)^{l-1} {\bf 1} + 8l \alpha(-4)^{l} {\bf 1} \\
&=& 12 \alpha(-1)\alpha(-4)^l  \alpha(-5){\bf 1}  + 12 \alpha(-2) \alpha(-4)^{l+1} {\bf 1}  + 6 \alpha(-3)^2\alpha(-4)^l {\bf 1} \\
& & \quad + \,42  \alpha(-1) \alpha(-4)^{l+1} {\bf 1}  + 42  \alpha(-2) \alpha(-3)\alpha(-4)^l {\bf 1}  +  50  \alpha(-1) \alpha(-3)\alpha(-4)^l {\bf 1}  \\
& & \quad + \,  25 \alpha(-2)^2\alpha(-4)^l {\bf 1}  +  20 \alpha(-1) \alpha(-2)\alpha(-4)^l {\bf 1}  .
\end{eqnarray*}
Applying  \eqref{145},  \eqref{3^24}, \eqref{234}-\eqref{224}, and then simplifying by combining like terms we have
\begin{eqnarray*}
\lefteqn{\alpha(-1)^2{\mathbf{1}}*_2\alpha(-4)^l\mathbf{1}}\\
&\equiv_2& - \, \frac{3}{l + 1}  \left(\alpha(-2) \alpha(-4)^{l+1} {\bf 1} + (1 + 4(l+1)) \alpha(-1) \alpha(-4)^{l+1} {\bf 1}    \right) + 12 \alpha(-2) \alpha(-4)^{l+1} {\bf 1}  \\
& & \quad + \, 6 \alpha(-4)^{l+2}  {\bf 1}  +  42  \alpha(-1) \alpha(-4)^{l+1} {\bf 1}   +  42\Big(   \frac{1}{10}\alpha(-1) \mathbf{1} *_2\alpha(-2)\alpha(-4)^l \mathbf{1} \\
& & \quad - \,  \frac{3(3l+2)}{10(l+1)}\alpha(-2)\alpha(-4)^{l+1}{\mathbf{1}} -  \frac{3}{10(l+1)}\alpha(-4)^{l+2}{\mathbf{1}}\Big)   \\
& & \quad + \,  50 \Big(  \frac{1}{10}\alpha(-1) \mathbf{1} *_2\alpha(-1)\alpha(-4)^l \mathbf{1} -   \frac{3(6l+5)}{20(l+1)}\alpha(-1)\alpha(-4)^{l+1} \mathbf{1} \\
& & \quad + \,  \frac{3}{20(l+1)}\alpha(-2)\alpha(-4)^{l+1} \mathbf{1} \Big)    +   25\Big(   -\frac{3}{5}\alpha(-1){\mathbf{1}}*_2\alpha(-2)\alpha(-4)^l \mathbf{1} \\
& & \quad + \, \frac{7l+3}{5(l+1)}\alpha(-2)\alpha(-4)^{l+1} \mathbf{1}  - \ \frac{(5l+1)}{5(l+1)}\alpha(-4)^{l+2} \mathbf{1} \Big)  +  20 \alpha(-1) \alpha(-2)\alpha(-4)^l {\bf 1}   \\
&=& \frac{92l+63}{10(l+1)}\alpha(-2)\alpha(-4)^{l+1} \mathbf{1} -\frac{3(10l+7)}{2(l+1)}\alpha(-1)\alpha(-4)^{l+1}\mathbf{1} - \frac{(95l+58)}{5(l+1)}\alpha(-4)^{l+2} \mathbf{1}\\
& & \quad - \, \frac{54}{5}\alpha(-1) \mathbf{1} *_2\alpha(-2)\alpha(-4)^l \mathbf{1} + 5\alpha(-1) \mathbf{1} *_2 \alpha(-1)\alpha(-4)^l \mathbf{1}  \\
& & \quad + \, 20\alpha(-1)\alpha(-2)\alpha(-4)^l{\mathbf{1}}.
\end{eqnarray*}

Using this expression to rewrite $10 \alpha(-1)\alpha(-2)\alpha(-4)^{l}{\bf 1}$ in \eqref{red124to1424} and simplifying proves Lemma \ref{eqnC}.
\end{proof}

Applying Lemma \ref{powersofalpha345}, we see that Lemma \ref{eqnC} gives a recursive definition for $\alpha(-2)\alpha(-4)^{l}\mathbf{1}$, $l>0$ in terms of powers of $\alpha(-1){\bf 1}$, elements of the form $\alpha(-1)\alpha(-4)^{l'} \mathbf{1}$, with  $l' \leq l$, and $\alpha(-1)^2 {\bf 1}$.  In addition, when $l=0$, we have $\alpha(-2)\alpha(-4)^{l}{\mathbf{1}}=\alpha(-2){\bf{1}} \equiv_2 - \alpha(-1) {\bf 1}$ by the fact that $(L(-1)+L(0))\alpha(-1){\bf 1}\in O_2(V)$. These facts along with Proposition \ref{reduce-no11}, give the following proposition.

\begin{prop}\label{reduce-mislabeled}
$A_2(V)$ is generated by elements of the form $v + O_2(V)$ for $v$ in the set
\begin{align*} 
 \{\alpha(-1)^i&\alpha(-2)^j\alpha(-4)^l{\mathbf{1}} \, | \, 0\le i\le 2, \, 0\le j\le 1,\, (i,j)\ne (0,0),   (1,1), (0,1),\, l\in \mathbb{N}\} \\
&= \{\alpha(-1)^i\alpha(-2)^j\alpha(-4)^l{\mathbf{1}} \, | \, (i,j) =  (1,0),   (2,0), (2,1),\, l\in \mathbb{N}\}. 
\end{align*}
\end{prop}

To continue to reduce the set of generators we will need the following lemma.

\begin{lem}\label{eqnD-lemma} 
 For $l \in \mathbb{N}$
\begin{eqnarray}\label{l4times4l}
\lefteqn{ \alpha(-1)\alpha(-4){\bf 1} *_2 \alpha(-4)^l{\bf 1} }\\
&\equiv_2&  2\alpha(-1){\bf 1}*_2\alpha(-1)\alpha(-4)^l{\bf 1}+\frac{2l+1}{l+1}\alpha(-1)\alpha(-4)^{l+1}{\bf 1} \nonumber \\
& & \quad - \, \frac{3l}{l+1}\alpha(-2)\alpha(-4)^{l+1}{\bf 1}-3\alpha(-1){\bf 1}*_2\alpha(-2)\alpha(-4)^l{\bf 1} \nonumber \\
& & \quad - \, \frac{(l+5)}{l+1}\alpha(-4)^{l+2}{\bf 1} \nonumber \\
\label{eqnD}
&\equiv_2& - \frac{2(l-1)}{4l+1}  \alpha(-1){\bf 1}*_2\alpha(-1)\alpha(-4)^l{\bf 1}-\frac{(2l+1)(l-1)}{(l+1)(4l+1)}\alpha(-1)\alpha(-4)^{l+1}{\bf 1} \\
& & \quad + \, \frac{6l^2+3l-5}{(l+1)(4l+1)}\alpha(-4)^{l+2}{\bf 1}-\frac{10l}{4l+1}\alpha(-1)^2{\bf 1}*_2\alpha(-4)^l{\bf 1}\nonumber \\
& & \quad - \, \frac{3}{4l+1}\alpha(-1){\bf 1}*_2\alpha(-2)\alpha(-4)^l{\bf 1}. \nonumber
\end{eqnarray}
\end{lem}

\begin{proof} By the formula for multiplication in $A_2(V)$,  Eqn.\  (\ref{*_n-definition}), the recursion \eqref{recursion-level-2}, simplifying, and then applying Lemma  \ref{powersofalpha345},  we have 
\begin{eqnarray*}
\lefteqn{\alpha(-1)\alpha(-4){\bf 1}*_2\alpha(-4)^l{\bf 1} } \\
&=&   \res_x  \left( \frac{1}{x^3}  -  \frac{3}{x^4} +  \frac{6}{x^5} \right) (1+x)^7 \nord \Biggl( \sum_{j \in \mathbb{Z}} \alpha(j) x^{-j-1}\Biggr)\Biggl( \sum_{k \in \mathbb{Z}}  \frac{1}{3!}\left( \frac{\partial}{\partial x}\right)^3 \alpha(k) x^{-k-1}\Biggr)  \nord \\
& & \quad \cdot \alpha(-4)^l\mathbf{1}   \\
&=& - \frac{1}{6} \res_x \sum_{j,k \in \mathbb{Z}} \big( 6x^{-5} + 39 x^{-4} + 106 x^{-3} + 154 x^{-2} + 126 x^{-1} + 56 + 14x + 6x^2 + 4x^3 \\
& & \quad +  \, x^4 \big) x^{-j-k-5}  (k+1)(k+2) (k+3) \nord \alpha(j) \alpha(k) \nord \alpha(-4)^l {\bf 1} \\
%&=& \Big(210\alpha(-1)\alpha(-8) +120\alpha(-2)\alpha(-7)+60\alpha(-3)\alpha(-6)+30\alpha(-4)\alpha(-5)\\
%& & \quad + \,  1110\alpha(-13)\alpha(4)  + 780\alpha(-1)\alpha(-7) + 390\alpha(-2)\alpha(-6) + 156\alpha(-3)\alpha(-5) \\
%& & \quad + \,  39\alpha(-4)^2 + 5070\alpha(-12)\alpha(4) + 1060\alpha(-1)\alpha(-6) + 424\alpha(-2)\alpha(-5) \\
%& & \quad + \,  106\alpha(-3)\alpha(-4) +  9010\alpha(-11)\alpha(4)   + 616\alpha(-1)\alpha(-5)
% + 154\alpha(-2)\alpha(-4) \\
% & & \quad + \,  7546\alpha(-10)\alpha(4)   +126\alpha(-1)\alpha(-4) + 2646\alpha(-9)\alpha(4) -210\alpha(-7) \alpha(4) \\
 %& & \quad - \, 150\alpha(-6)\alpha(4) -  124\alpha(-5) \alpha(4) -34\alpha(-4)\alpha(4) \Big)\alpha(-4)^l{\bf 1}\\
&=& \Big(210\alpha(-1)\alpha(-8) + 120\alpha(-2)\alpha(-7) + 60\alpha(-3)\alpha(-6) + 30\alpha(-4)\alpha(-5)\\
&&\quad + \, 780\alpha(-1)\alpha(-7) + 390\alpha(-2)\alpha(-6) + 156\alpha(-3)\alpha(-5) +  39\alpha(-4)^2\\
&& \quad  + \, 1060\alpha(-1)\alpha(-6) + 424\alpha(-2)\alpha(-5) + 106\alpha(-3)\alpha(-4) + 616\alpha(-1)\alpha(-5)\\
&& \quad + \, 154\alpha(-2)\alpha(-4) + 126\alpha(-1)\alpha(-4) \Big)\alpha(-4)^l{\bf 1}\\
&& \quad  + \,  \Big(1110\alpha(-13) + 5070\alpha(-12) + 9010\alpha(-11) + 7546\alpha(-10) + 2646\alpha(-9)\\
&&\quad - \, 210\alpha(-7) - 150\alpha(-6)- 1 24\alpha(-5) - 34\alpha(-4)\Big)\alpha(4)\alpha(-4)^l{\bf 1}\\
%&\sim_2& \Big( - 210\alpha(-1) \big( 6 \alpha(-3) + 15 \alpha(-4) + 10 \alpha(-5) \big) + 120\alpha(-2)\big( 3 \alpha(-3) + 8 \alpha(-4)  \\
%& & \quad + \,  6 \alpha(-5) \big) -  60\alpha(-3)\big( \alpha(-3) + 3 \alpha(-4) + 3 \alpha(-5) \big) + 30\alpha(-4)\alpha(-5)\\
%& &\quad + \, 780\alpha(-1)\big(3 \alpha(-3) + 8 \alpha(-4) + 6\alpha(-5)  \big) - 390 \alpha(-2)\big( \alpha(-3) + 3 \alpha(-4) \\
%& & \quad + \,  3 \alpha(-5) \big) +  156\alpha(-3)\alpha(-5) +  39\alpha(-4)^2 - 1060\alpha(-1)\big( \alpha(-3) + 3 \alpha(-4) \\
%& & \quad + \,  3 \alpha(-5) \big)  +   424\alpha(-2)\alpha(-5) + 106\alpha(-3)\alpha(-4) + 616\alpha(-1)\alpha(-5) \\
%& & \quad + \,  154\alpha(-2)\alpha(-4)  +  126\alpha(-1)\alpha(-4) \Big)\alpha(-4)^l{\bf 1} + \Big(1110\big( 36\alpha(-3) + 80 \alpha(-4) \\
%& & \quad +\,  45 \alpha(-5) \big)  -  5070\big( 28 \alpha(-3) + 63 \alpha(-4) + 36 \alpha(-5) \big) + 9010 \big(21 \alpha(-3) + 48\alpha(-4) \\
%& & \quad + \,  28 \alpha(-5) \big)  -  7546\big( 15 \alpha(-3) + 35\alpha(-4) + 21 \alpha(-5) \big) + 2646 \big(10 \alpha(-3) + 24 \alpha(-4) \\
%& & \quad +  \, 15 \alpha(-5) \big) -  210\big(3 \alpha(-3) + 8 \alpha(-4) + 6 \alpha(-5)  \big) + 150\big( \alpha(-3) + 3 \alpha(-4) + 3 \alpha(-5) \big)\\
%& & \quad  -  \, 124\alpha(-5)  - 34\alpha(-4)\Big)4l\alpha(-4)^{l-1} {\bf 1}\\
%&=&  
& \sim_2 & \Big( 20 \alpha(-1) \alpha(-3) + 36 \alpha(-1) \alpha(-4) + 16 \alpha(-1)  \alpha(-5)  - 30 \alpha(-2) \alpha(-3)\\
& & \quad  - \, 56 \alpha(-2) \alpha(-4)  -  26 \alpha(-2) \alpha(-5)  -60 \alpha(-3)^2  - 74 \alpha(-3) \alpha(-4) \\
& & \quad - \,   24\alpha(-3) \alpha(-5)   +  30\alpha(-4)\alpha(-5)  +  39\alpha(-4)^2    \Big)\alpha(-4)^l{\bf 1} \\
&\equiv_2 &  20 \alpha(-1) \alpha(-3) \alpha(-4)^l{\bf 1}+ 36 \alpha(-1)  \alpha(-4)^{l+1} {\bf 1}+ 16 \alpha(-1)   \alpha(-4)^l \alpha(-5) {\bf 1} \\
& & \quad - \,  30 \alpha(-2) \alpha(-3) \alpha(-4)^l{\bf 1}  -  56 \alpha(-2) \alpha(-4)^{l+1}{\bf 1}  -  26 \alpha(-2) \alpha(-4)^l  \alpha(-5) {\bf 1}\\
& & \quad  - \,  \alpha(-4)^{l+2} {\bf 1} .
\end{eqnarray*} 

%
%MAYBE WE SHOULD MAKE A REMARK ABOUT THE $\alpha(4) \alpha(-4)$ TERMS VANISHING IN THE CALCULATION ABOVE?  OR NOTE THAT MAYBE THERE IS A SIMPLIFICATION WE ARE MISSING?! 

Applying Eqns.\ \eqref{145}--\eqref{134} and simplifying proves Eqn.\  (\ref{l4times4l}).

Solving for $\alpha(-2)\alpha(-4)^{l+1}{\bf 1}$ in  \eqref{eqnLemC} and plugging this expression into \eqref{l4times4l}, gives Eqn.\ (\ref{eqnD}). 
\end{proof}

Using Lemma \ref{eqnD-lemma} in conjunction with Lemma \ref{eqnC} and Proposition \ref{reduce-mislabeled}, we have the following further reduction of generators:

\begin{prop}\label{reduce-no01-first}
$A_2(V)$ is generated by elements of the form $v + O_2(V)$ for $v$ in the set
\begin{multline*} 
\{\alpha(-1){\bf 1}, \ \alpha(-1)^2{\bf 1}, \ \alpha(-1)\alpha(-4){\bf 1}, \ \alpha(-1) \alpha(-4)^2 {\bf 1}, \ \alpha(-1)^2\alpha(-4)^l{\bf 1}, \\ 
\alpha(-1)^2\alpha(-2)\alpha(-4)^{l}{\bf 1} \, | \, l\in \mathbb{N}\} . 
\end{multline*}
\end{prop}

\begin{proof} 
By Proposition \ref{reduce-mislabeled}, we need only show that elements of the form $\alpha(-1) \alpha(-4)^{l+1} {\bf 1}$ for  $l >1$ are generated modulo $O_2(V)$ by elements in the set specified, which for the purposes of this proof, we call $S$.    We do this by first solving for $\alpha(-1)\alpha(-4)^{l+1}{\bf 1}$ in \eqref{eqnD} (which has a nonzero coefficient if $l>1$) and then using Lemma \ref{powersalpha1-lemma}  to show $\alpha(-1) \alpha(-4)^{l+1} {\bf 1}$ is generated by terms in $S$, terms of the form $\alpha(-1) \alpha(-4)^{l'} {\bf 1}$ with $l'<l+1$, and terms of the form $\alpha(-2) \alpha(-4)^l {\bf 1}$.   Applying   Lemma \ref{eqnC} repeatedly, we  obtain an expression for $\alpha(-1)\alpha(-4)^{l+1}{\bf 1}$ where all terms involve only $\alpha(-1)\alpha(-4)^{l'}{\bf 1}$ with $l'<l+1$, $\alpha(-4)^k{\bf 1}=(-1)^k(\alpha(-1){\bf 1})^k$, $\alpha(-1)^2{\bf 1}$, and $\alpha(-2){\bf 1}\equiv_2-\alpha(-1){\bf 1}$.
The result follows by induction on $l$.
\end{proof} 

To decrease the amount of calculations needed to reduce the generators further, we note that in the calculations that lead to the set of generators given in Proposition \ref{reduce-no01-first}, it is helpful to deduce the following corollary based on our previous calculations: 

\begin{cor}\label{termsR1}
Any element  in $A_2(M_a(1))$ of the form
\[ \alpha(-1)^{i_1}\alpha(-2)^{i_2}\cdots \alpha(-t)^{i_t}{\bf 1}+O_2(M_a(1)), \] 
where $i_1\le 1$,  $t \in \mathbb{Z}_+$, and $i_2,\dots, i_t\in \mathbb{N}$, is generated in $A_2(M_a(1))$ by the set
\[ R =\{v + O_2(M_a(1)) \in A_2(M_a(1)) \; | \; v =  \alpha(-1) {\bf 1}, \, \alpha(-1)^2 {\bf 1}, \,  \alpha(-1) \alpha(-4) {\bf 1}, \,  \alpha(-1)\alpha(-4)^2{\bf 1} \}.\]
\end{cor}

\begin{proof}  We trace the steps up to the proof of Proposition \ref{reduce-no01-first} to verify that reducing elements of the form in question to those in the set of generators given in Proposition \ref{reduce-no01-first} never result in terms of the form $\alpha(-1)^2 \alpha(-4)^l {\bf 1}$ for $l>0$ or $\alpha(-1)^2 \alpha(-2) \alpha(-4)^l {\bf 1}$ for $l \in \mathbb{N}$.  That is, never result in an increase in the exponent of the $\alpha(-1)$ term in the expression modulo $O_2(V)$, unless it involves just $\alpha(-1)^2 {\bf 1}$, and thus are generated by elements in the set $R$.

By the recursion, Eqn.\ (\ref{recursion-level-2}), any element of the form given in the Corollary can be 
written as $\alpha(-1)^{i_1}\alpha(-2)^{i_2}\cdots \alpha(-5)^{i_5}{\bf 1}+O_2(M_a(1))$,
where $i_1\le 1$ and $ i_2,\cdots, i_5\in \mathbb{N}$, since the recursion does not introduce any $\alpha(-1)$ terms.  

Then in using Eqn.\ (\ref{reduce-equation-Heisenberg-level2}) to reduce to elements of the form $\alpha(-1)^{j_1}\alpha(-2)^{j_2}\alpha(-3)^{j_3} \alpha(-4)^{j_4}{\bf 1}$ modulo $O_2(M_a(1))$, we note that this again does not introduce any $\alpha(-1)$ terms, and thus the elements in question can be generated by elements of this form with $j_1\le 1$ and $j_2,j_3, j_4\in \mathbb{N}$.

Then, in proving Proposition \ref{generators-prop}, which reduces the $j_2$ exponent of $\alpha(-2)$ in these terms of the form $\alpha(-1)^{j_1}\alpha(-2)^{j_2}\alpha(-3)^{j_3} \alpha(-4)^{j_4}{\bf 1}+O_2(M_a(1))$, we note that the exponent $j_1$ of $\alpha(-1)$ in all expressions is only equal to or less than that in previous expressions.  Thus the terms in question are all generated by elements of the form $\alpha(-1)^{j_1}\alpha(-2)^{j_2}\alpha(-3)^{j_3} \alpha(-4)^{j_4}{\bf 1}+O_2(M_a(1))$ where $j_1, j_2 \le 1$, and $j_3, j_4 \in \mathbb{N}$.

Following the proof of Proposition \ref{reduce-301}, we have that cases (1) and (3), which handles terms that have a nonzero exponent for the $\alpha(-2)$ term, reduce  via Eqn.\ (\ref{redpowersofalpha3})  to elements in the generating set given in Proposition \ref{reduce-301} without introducing any $\alpha(-1)$ terms.  

Following the proof of Proposition \ref{reduce-301}, for case (2) which handles terms that have an $\alpha(-1)$ but no $\alpha(-2)$, gives a reduction to terms of the form $\alpha(-1) \alpha(-3)^{j_3} \alpha(-4)^{j_4}{\bf 1} + O_2(M_a(1))$ for $j_3 = 0$ or $1$ and $j_4 \in \mathbb{N}$.   
And then terms with no $\alpha(-1)$ or $\alpha(-2)$, i.e., of the form $\alpha(-3)^{j_3} \alpha(-4)^{j_4} {\bf 1} + O_2(M_a(1))$ are generated by elements in $R$ by Lemma \ref{powersofalpha345}.

We are left to show that elements of the form $\alpha(-1)^{j_1} \alpha(-2)^{j_2} \alpha(-3)^{j_3} 
\alpha(-4)^{j_4} {\bf 1}$ modulo $O_2(M_a(1))$, with $j_1, j_2, j_3 \leq 1$, $(j_1, j_2) \neq (0,0)$ and $j_4 \in \mathbb{N}$ are generated by $R$.  

Moving to Proposition \ref{reduce-300} and the proof of this, we see that via Eqn.\ (\ref{red-301b}) and Cases (1) and (2),  terms of these forms with an $\alpha(-3)$ term can be written as $\alpha(-1)^{j_1} \alpha(-2)^{j_2}\alpha(-4)^{j_4} {\bf 1}$ with $j_1, j_2 \leq 1$, $(j_1, j_2) \neq (0,0)$, $j_4 \in \mathbb{N}$ and $\alpha(-2)^2 \alpha(-4)^{j_4} {\bf 1}$.  But this last term can be written in terms of elements generated by $R$ or of the form $\alpha(-2) \alpha(-4)^i {\bf 1}$ via Lemmas \ref{helprel} and \ref{powersalpha1-lemma}.

We are left to show that elements of the form $\alpha(-1)^{j_1} \alpha(-2)^{j_2} 
\alpha(-4)^{j_4} {\bf 1} + O_2(M_a(1))$ with $j_1, j_2 \leq 1$, $(j_1, j_2) \neq (0,0)$ and $j_4 \in \mathbb{N}$ are generated by $R$.  That $\alpha(-1) \alpha(-2) \alpha(-4)^l {\bf 1}$ terms are generated by elements in $R$ as well as elements of the form $\alpha(-1) \alpha(-4)^j {\bf 1}$ and 
$\alpha(-2) \alpha(-4)^j {\bf 1}$ for $j>0$ follows from Lemmas \ref{eqnB} and \ref{powersalpha1-lemma}. 

That terms involving $\alpha(-2) \alpha(-4)^j {\bf 1}$ for $j>0$ are generated by elements in $R$ and elements of the form $\alpha(-1) \alpha(-4)^j {\bf 1}, j>0$ follows from Lemmas \ref{eqnC} and \ref{powersalpha1-lemma}. 

Finally, that elements of the form $\alpha(-1)\alpha(-4)^l{\bf 1}$,  for $ l>0$ are generated by the set $R$ follows from the proof of Proposition \ref{reduce-no01-first}.
\end{proof}

The usefulness of the corollary above is that it will allow us to avoid writing down explicit formulas for many of the terms in our later calculations, and instead only write the terms up to elements generated by the set $R$.

For the remainder of this section, we denote by $\mathcal{R}$  the set of elements that 
are linear combinations of terms of the form $\alpha(-1)^{i_1}\alpha(-2)^{i_2}\cdots \alpha(-t)^{i_t}   {\bf 1} + O_2(M_a(1))$, where $t \in \mathbb{Z}_+$,  $i_2, \dots, i_t \in \mathbb{N}$, and $i_1\le 1$. By Corollary  
\ref{termsR1}, such terms are generated in $A_2(V)$ by the set  $R$.  Avoiding writing down
 explicit formulas for the  $\mathcal{R}$ terms greatly simplifies the calculations required
 in further reducing our set of generators.

We use this to prove the following two lemmas.  The first lemma will allow us to express the term $\alpha(-1)^2 \alpha(-4)^2 {\bf 1}$ in terms of other generators and elements in $\mathcal{R}$, and thus help in the further reduction of generators from those given in Proposition \ref{reduce-no01-first} to those given in Proposition \ref{reduce-no01}, below.

\begin{lem}\label{reduce-one^2-four^2}  
\[
\alpha(-1)^2{\bf 1}*_2\alpha(-1)\alpha(-4){\bf 1} \equiv_2  -\frac{1}{2}\alpha(-1)^2\alpha(-4)^2{\bf 1} -\alpha(-1){\bf 1}*_2\alpha(-1)^2\alpha(-4){\bf 1} + r
\]
for some $r \in \mathcal{R}$.
\end{lem}

\begin{proof}
By the definition of multiplication in $A_2(V)$, Eqn.\ (\ref{*_n-definition}), Corollary \ref{termsR1}, the recursion \eqref{recursion-level-2}, and simplifying, we have that for some $r \in \mathcal{R}$,  
\begin{eqnarray*}
\lefteqn{ \alpha(-1)^2{\bf 1}*_2\alpha(-1)\alpha(-4){\bf 1}} \\
&=&  12 \alpha(-1)^2 \alpha(-4) \alpha(-5) {\bf 1} +  42 \alpha(-1)^2 \alpha(-4)^2 {\bf 1} + 50 \alpha(-1)^2  \alpha(-3) \alpha(-4) {\bf 1} \\
& & \quad + \,  20 \alpha (-1)^2 \alpha(-2) \alpha(-4) {\bf 1} + r .
\end{eqnarray*}

By Eqn.\ (\ref{O24}) applied to $v = \alpha(-1) \alpha(-4){\bf 1}$, we have, 
\begin{eqnarray*}
\lefteqn{20 \alpha (-1)^2 \alpha(-2) \alpha(-4) {\bf 1}  }\\
&=&  20\alpha(-1)\alpha(-2)\big(\alpha(-1) \alpha(-4){\bf 1} \big) \\
&\sim_2 & -20\big(3 \alpha(-1) \alpha(-3) + 3 \alpha(-1) \alpha(-4) + \alpha(-1) \alpha(-5) + 2 \alpha(-2)^2 + 4 \alpha(-2) \alpha(-3) \\
& & \quad + \, 4 \alpha(-2) \alpha(-4) + \alpha(-2) \alpha(-5) + 2 \alpha(-2) \alpha(-3)  + 2 \alpha(-3)^2 \\
& & \quad + \, \alpha(-3) \alpha(-4) \big) \alpha(-1) \alpha(-4) {\bf 1} ,
\end{eqnarray*}
and thus for some $r' \in \mathcal{R}$, we have
\begin{eqnarray*}
\lefteqn{ \alpha(-1)^2{\bf 1}*_2\alpha(-1)\alpha(-4){\bf 1}} \\
&\sim_2&  12 \alpha(-1)^2 \alpha(-4) \alpha(-5) {\bf 1} +  42 \alpha(-1)^2 \alpha(-4)^2 {\bf 1} + 50 \alpha(-1)^2  \alpha(-3) \alpha(-4) {\bf 1} \\
& & \quad  - \, 60 \alpha(-1)^2 \alpha(-3)\alpha(-4) {\bf 1}  -60  \alpha(-1)^2 \alpha(-4)^2 {\bf 1}  -20 \alpha(-1)^2 \alpha(-4)  \alpha(-5) {\bf 1} + r' \\
&=&  -8  \alpha(-1)^2 \alpha(-4) \alpha(-5) {\bf 1} -18  \alpha(-1)^2 \alpha(-4)^2 {\bf 1} -10  \alpha(-1)^2  \alpha(-3) \alpha(-4) {\bf 1}  + r' .
\end{eqnarray*}

Applying \eqref{multiplication-one-one} with $v =\alpha(-1)^2 \alpha(-4) {\bf 1}$  to the term $-10\alpha(-1)^2\alpha(-3)\alpha(-4){\bf 1}$, simplifying,  then applying \eqref{reduce-equation-Heisenberg-level2} to $-2\alpha(-1)^2\alpha(-4)\alpha(-5){\bf 1}$, and simplifying gives
\begin{eqnarray*}
\lefteqn{ \alpha(-1)^2{\bf 1}*_2\alpha(-1)\alpha(-4){\bf 1}} \\
&\sim_2& -8  \alpha(-1)^2 \alpha(-4) \alpha(-5) {\bf 1} -18  \alpha(-1)^2 \alpha(-4)^2 {\bf 1} - \alpha(-1){\bf 1} *_2 \alpha(-1)^2 \alpha(-4) {\bf 1} \\
& & \quad + \,  15 \alpha(-4) \alpha(-1)^2 \alpha(-4) {\bf 1} +  6 \alpha(-5) \alpha(-1)^2 \alpha(-4) {\bf 1} + r'\\
&=&- 2\alpha(-1)^2\alpha(-4)\alpha(-5){\bf 1}-3\alpha(-1)^2\alpha(-4)^2{\bf 1}-\alpha(-1){\bf 1}*_2\alpha(-1)^2\alpha(-4){\bf 1} + r''\\
&\equiv_2&  \frac{1}{4} \Big( 10 \alpha(-1)^2 \alpha(-4)^2 {\bf 1}  + 2\alpha(-1) \alpha(-2) \alpha(-4)^2 {\bf 1} \Big) -3\alpha(-1)^2\alpha(-4)^2{\bf 1} \\
& & \quad - \, \alpha(-1){\bf 1}*_2\alpha(-1)^2\alpha(-4){\bf 1} + r''\\
&=&    -\frac{1}{2} \alpha(-1)^2 \alpha(-4)^2 {\bf 1}  - \alpha(-1){\bf 1}*_2\alpha(-1)^2\alpha(-4){\bf 1} + r''',
\end{eqnarray*}
for some $r'',  r''' \in \mathcal{R}$, proving the result. 
\end{proof}

The next lemma will allow us to express the terms involving $\alpha(-1)^2 \alpha(-4)^j {\bf 1}$, for $j >2$, in terms of other generators, $\alpha(-1)^2 \alpha(-4)^2 {\bf 1}$  and  elements in $\mathcal{R}$.

\begin{lem}\label{reduce-one^2_four^l} 
For $l \in \mathbb{Z}_+$,  we have 
\begin{equation*}
\frac{l-1}{l+1}\alpha(-1)^2\alpha(-4)^{l+1}{\bf 1}\equiv_2-\alpha(-1){\bf 1}*_2\alpha(-1)^2\alpha(-4)^l{\bf 1}-\alpha(-1)^2\alpha(-4){\bf 1}*_2\alpha(-4)^l{\bf 1}+  r
\end{equation*}
where $r \in \mathcal{R}$.
\end{lem}

\begin{proof}
By the definition of multiplication in $A_2(V)$, i.e., Eqn.\ (\ref{*_n-definition}), we have 
\begin{eqnarray*}
\lefteqn{ \alpha(-1)^2 \alpha(-4) {\bf 1}*_2 \alpha(-4)^l{\bf 1}} \\
%&=&  \res_x  \left( \frac{1}{x^3}  -  \frac{3}{x^4} +  \frac{6}{x^5} \right) (1+x)^8 Y(\alpha(-1)^2 \alpha(-4) \mathbf{1}, x) \alpha(-4)^l\mathbf{1} \\
%&=&  \res_x  \left( \frac{1}{x^3}  -  \frac{3}{x^4} +  \frac{6}{x^5} \right) (1+x)^8 \nord \Big( Y(\alpha(-1){\bf 1}, x) )^2 \frac{1}{3!} \Big(\frac{\partial}{\partial x} \Big)^3 Y(\alpha(-1) {\bf 1}, x)  \Big) \nord \alpha(-4)^l\mathbf{1} \\
%&=&  \res_x  \left( \frac{1}{x^3}  -  \frac{3}{x^4} +  \frac{6}{x^5} \right) (1+x)^8 \nord \Big( \sum_{j\in \mathbb{Z}} \alpha(j) x^{-j-1} \Big)^2\\
%& & \quad   \frac{1}{3!} \Big(\sum_{k \in \mathbb{Z}} -(k+1)(k+2)(k+3)  \alpha(k) x^{-k-4} \Big) \nord \alpha(-4)^l\mathbf{1} \\
%&=& \res_x  \left( \frac{1}{x^3}  -  \frac{3}{x^4} +  \frac{6}{x^5} \right) (1+x)^8 \nord \Big( \sum_{j\in \mathbb{Z}} \alpha(j) x^{-j-1} \Big)^2\\
%& & \quad   \frac{1}{3!} \Big(\sum_{k \in \mathbb{Z}} -(k+1)(k+2)(k+3)  \alpha(k) x^{-k-4} \Big) \nord \alpha(-4)^l\mathbf{1} \\
&=&  -\res_x  \left( \frac{1}{x^3}  -  \frac{3}{x^4} +  \frac{6}{x^5} \right) (1+x)^8 \sum_{i,j,k\in \mathbb{Z}}   \frac{1}{3!} (k+1)(k+2)(k+3) \nord  \alpha(i) \alpha(j) \alpha(k) \nord  \\
& & \quad x^{-i-j-k-6}  \alpha(-4)^l\mathbf{1}.
\end{eqnarray*} 
Observe that the only nonnegative mode $\alpha(n)$, for $n \geq 0$, acting on $\alpha(-4)^l{\bf 1}$ that does not annihilate $\alpha(-4)^l {\bf 1}$ is $\alpha(4)$, and we will see below that such terms do not appear once we take the residue or they are in $\mathcal{R}$.  In addition, observe that $k = -1, -2, -3$ gives zero,  and so the highest power of $\alpha(-1)$ that appears on the righthand side is 2.  Thus by these observations and Corollary \ref{termsR1}, we have
\begin{eqnarray*}
\lefteqn{ \alpha(-1)^2 \alpha(-4) {\bf 1}*_2 \alpha(-4)^l{\bf 1}} \\
&=&   -\res_x  \left( \frac{1}{x^3}  -  \frac{3}{x^4} +  \frac{6}{x^5} \right) (1+x)^8 \Bigl( \alpha(-1)^2 \sum_{k\leq -4}   \frac{1}{3!} (k+1)(k+2)(k+3) \alpha(k)  x^{-k-4} \\
& & \quad + \, 35 \alpha(-1)^2   \alpha(4)  x^{-8}  \Bigr)  \alpha(-4)^l\mathbf{1} + r
\end{eqnarray*}
for some $r \in \mathcal{R}$. Taking the residue, using Eqn.\ \eqref{recursion-level-2}, Eqn.\ \eqref{multiplication-one-one} with $v = \alpha(-1)^2 \alpha(-4)^l {\bf 1}$, Eqn.\ \eqref{reduce-equation-Heisenberg-level2}, and then Corollary \ref{termsR1} again, we have
\begin{eqnarray*}
\lefteqn{ \alpha(-1)^2 \alpha(-4) {\bf 1}*_2 \alpha(-4)^l{\bf 1}} \\
&=&   -\res_x  \left( \frac{1}{x^3}  -  \frac{3}{x^4} +  \frac{6}{x^5} \right) (1+x)^8 \Bigl( \alpha(-1)^2 \sum_{k\leq -4}  \binom{k+3}{3} \alpha(k)  x^{-k-4} \\
& & \quad + \,  35 \alpha(-1)^2   \alpha(4)  x^{-8}  \Bigr)  \alpha(-4)^l\mathbf{1} + r\\
%&=&  - \alpha(-1)^2 \Bigg(6\binom{-8+3}{3}\alpha(-8)+45\binom{-7+3}{3} \alpha(-7)+145\binom{-6+3}{3} \alpha(-6)\\
%& & \quad +\, 260\binom{-5+3}{3} \alpha(-5) + 280\binom{-4+3}{3} \alpha(-4)\Bigg)\alpha(-4)^l {\bf 1}+r\\
&=&   \alpha(-1)^2 \Big(210 \alpha(-8)+900 \alpha(-7)+1450 \alpha(-6)+1040 \alpha(-5) +  280 \alpha(-4)\Big)\alpha(-4)^l{\bf 1}+r\\
&\sim_2&  \alpha(-1)^2  \Big(-210(6\alpha(-3)+15\alpha(-4)+10\alpha(-5)) + 900 (3\alpha(-3)+8\alpha(-4)+6\alpha(-5))\\
& & \quad - \, 1450 (\alpha(-3)+3\alpha(-4)+3\alpha(-5))+1040 \alpha(-5) + \, 280 \alpha(-4)\Big) \alpha(-4)^l{\bf 1} +r\\
%&=&   \big(-210(6) +900(3)-1450\big)\alpha(-1)^2\alpha(-3)\alpha(-4)^l{\bf 1}\\& & +\big(-210(15)+900(8)-1450(3)+280\big)\alpha(-1)^2\alpha(-4)^{l+1}{\bf 1}\\& & +\big(-210(10)+900(6)-1450(3)+1040\big)\alpha(-1)^2\alpha(-4)^{l}\alpha(-5){\bf 1}+r\\
&=&-10\alpha(-1)^2\alpha(-3)\alpha(-4)^l{\bf 1}-20\alpha(-1)^2\alpha(-4)^{l+1}{\bf 1}-10\alpha(-1)^2\alpha(-4)^l\alpha(-5){\bf 1}+r\\
&=&-\alpha(-1){\bf 1}*_2\alpha(-1)^2\alpha(-4)^l{\bf 1}+15\alpha(-1)^2\alpha(-4)^{l+1}{\bf 1}+6\alpha(-1)^2\alpha(-4)^l\alpha(-5){\bf 1}\\&&-20\alpha(-1)^2\alpha(-4)^{l+1}{\bf 1}-10\alpha(-1)^2\alpha(-4)^l\alpha(-5){\bf 1}+r\\
%&=&-\alpha(-1){\bf 1}*_2\alpha(-1)^2\alpha(-4)^l{\bf 1}-5\alpha(-1)^2\alpha(-4)^{l+1}{\bf 1}-4\alpha(-1)^2\alpha(-4)^l\alpha(-5){\bf 1}+r\\
&\equiv_2&-\alpha(-1){\bf 1}*_2\alpha(-1)^2\alpha(-4)^l{\bf 1}-5\alpha(-1)^2\alpha(-4)^{l+1}{\bf 1}\\&&
-4\left(-\frac{1}{4(l+1)}\left( (4l+6)\alpha(-1)^2\alpha(-4)^{l+1}{\bf 1} +    2\alpha(-1)\alpha(-2)\alpha(-4)^{l+1}{\bf 1} \right)\right)+r\\
& = &-\alpha(-1){\bf 1} *_2 \alpha(-1)^2\alpha(-4)^l{\bf 1}-\frac{l-1}{l+1}\alpha(-1)^2\alpha(-4)^{l+1}{\bf 1}+r',\\
\end{eqnarray*} 
for some $r, r'\in \mathcal{R}$. Rearranging, the result follows. 
\end{proof}

We are now ready to prove the following:
\begin{prop}\label{reduce-no01}
Terms of the form $\alpha(-1)^2\alpha(-4)^l{\bf 1}+O_2(V)$ in $A_2(M_a(1))$ are generated by elements of the form $v+O_2(V)$ where $v$ is an element in the set 
\begin{equation}\label{set}
\{\alpha(-1){\bf 1}, \alpha(-1)^2{\bf 1}, \alpha(-1)\alpha(-4){\bf 1}, \alpha(-1)^2\alpha(-4){\bf 1}, \alpha(-1)\alpha(-4)^2{\bf 1}\}.
\end{equation} 
\end{prop}

\begin{proof}
Lemma \ref{reduce-one^2-four^2} together with Corollary \ref{termsR1}  implies that $\alpha(-1)^2\alpha(-4)^2{\bf 1}$ can be expressed in terms of elements in the set given in the present proposition, which we denote by $S$ for the purposes of this proof. Lemma \ref{reduce-one^2_four^l}  together with  Corollary  \ref{termsR1} implies that if $\alpha(-1)^2\alpha(-4)^{l}{\bf 1}$, with $l>1$ is generated by elements in the set $S$, then so is $\alpha(-1)^2\alpha(-4)^{l+1}{\bf 1}$. The result follows by induction on $l$. 
\end{proof}

\begin{lem}\label{one^1two^2four}
For $l \in \mathbb{N}$, 
\[
\alpha(-1)^2\alpha(-2)\alpha(-4)^l{\bf 1}
\equiv_2-\frac{3}{10}\alpha(-1){\bf 1}*_2\alpha(-1)^2\alpha(-4)^l{\bf 1}+\frac{7l+3}{10(l+1)}\alpha(-1)^2\alpha(-4)^{l+1}{\bf 1}+  r
\]
for some $r \in \mathcal{R}$.
\end{lem}

\begin{proof}
Applying Eqn.\  \eqref{O24} to $v = \alpha(-1) \alpha(-4)^l {\bf 1}$, we have that for some $r \in \mathcal{R}$
\begin{eqnarray*}
\lefteqn{\alpha(-1)^2\alpha(-2)\alpha(-4)^l{\bf 1} \ =  \alpha(-1) \alpha(-2) \big(\alpha(-1) \alpha(-4)^l {\bf 1} \big) }\\
&\sim _2&  - \big(   3\alpha(-1)\alpha(-3) + 3\alpha(-1)\alpha(-4) + \alpha(-1) \alpha(-5) \big) \alpha(-1) \alpha(-4)^l {\bf 1} + r\\
&=& -3 \alpha(-1)^2 \alpha(-3) \alpha(-4)^l {\bf 1} - 3 \alpha(-1)^2 \alpha(-4)^{l+1} {\bf 1} - \alpha(-1)^2 \alpha(-4)^l \alpha(-5) {\bf 1} + r.
\end{eqnarray*}
Applying \eqref{multiplication-one-one}  with $v = \alpha(-1)^2\alpha(-4)^l {\bf 1}$, simplifying, and applying \eqref{reduce-equation-Heisenberg-level2},  we have
\begin{eqnarray*}
\lefteqn{\alpha(-1)^2\alpha(-2)\alpha(-4)^l{\bf 1} }\\
&\sim_2& -\frac{3}{10} \alpha(-1) {\bf 1} *_2   \alpha(-1)^2  \alpha(-4)^l {\bf 1}  + \frac{3}{10} \big(  15 \alpha(-4) + 6 \alpha(-5) \big) \alpha(-1)^2 \alpha(-4)^l {\bf 1}\\
& & \quad  - \,  3 \alpha(-1)^2 \alpha(-4)^{l+1} {\bf 1} - \alpha(-1)^2 \alpha(-4)^l \alpha(-5) {\bf 1} + r\\
&=&  -\frac{3}{10} \alpha(-1) {\bf 1} *_2   \alpha(-1)^2  \alpha(-4)^l {\bf 1}  + \frac{3}{2} \alpha(-1)^2 \alpha(-4)^{l+1} {\bf 1} \\
& & \quad + \,  \frac{4}{5}  \alpha(-1)^2 \alpha(-4)^l \alpha(-5) {\bf 1} + r\\
&\equiv_2 &  -\frac{3}{10} \alpha(-1) {\bf 1} *_2   \alpha(-1)^2  \alpha(-4)^l {\bf 1}  + \frac{3}{2} \alpha(-1)^2 \alpha(-4)^{l+1} {\bf 1} \\
& & \quad - \,  \frac{1}{5(l+1)}  \Big( (2 + 4(l+1))   \alpha(-1)^2 \alpha(-4)^{l+1}  {\bf 1} + 2 \alpha(-1) \alpha(-2) \alpha(-4)^{l+1} {\bf 1}  \Bigr)  + r\\
&=&   -\frac{3}{10} \alpha(-1) {\bf 1} *_2   \alpha(-1)^2  \alpha(-4)^l {\bf 1}  + \frac{7l+3}{10(l+1)} \alpha(-1)^2 \alpha(-4)^{l+1} {\bf 1}  + r',
\end{eqnarray*}
for some $r' \in \mathcal{R}$, 
giving the result.
\end{proof}

Lemma \ref{one^1two^2four} together with Propositions \ref{reduce-no01-first} and  \ref{reduce-no01} and  Corollary \ref{termsR1} immediately imply the following theorem.

\begin{thm}\label{main-generators-thm}
The algebra $A_2(M_a(1))$ is generated by $v + O_2(M_a(1))$ for $v$ in  
\[\{ \alpha(-1) {\bf 1}, \, \alpha(-1)^2 {\bf 1}, \,  \alpha(-1) \alpha(-4) {\bf 1}, \, \alpha(-1)^2\alpha(-4){\bf 1}, \,  \alpha(-1)\alpha(-4)^2{\bf 1} \}.\]
\end{thm}

\section{Relations for $A_2(M_a(1))$}\label{relations-section} 

In the previous section, we produced a generating set for $A_2(M_a(1))$. In this section we prove the relations in $A_2(M_a(1))$ satisfied by these generators  and then in the next section we prove that these are all the relations and give the structure of $A_2(M_a(1))$.  

To determine the relations in $A_2(M_a(1))$, we use two main tools: The level one Zhu algebra $A_1(M_a(1))$ and the zero mode action of the generators of $A_2(M_a(1))$ on the irreducible modules for $M_a(1)$, which are completely determined by $A_0(M_a(1))$.

{\bf Notation:} For convenience, we will use the following notation for the generators for $A_2(M_a(1))$ given in Theorem \ref{main-generators-thm} as elements in $A_n(M_a(1))$ for $n \in \mathbb{N}$
\begin{eqnarray}
x_n &=& \alpha(-1) \mathbf{1} + O_n(M_a(1)) \label{variable-notation-first} \\
y_n &=& \alpha(-1)^2 \mathbf{1} + O_n(M_a(1)) \\
\tilde y_n &=& \alpha(-1)\alpha(-4) \mathbf{1} + O_n(M_a(1)) \\
z_n  &=& \alpha(-1)^2 \alpha(-4) \mathbf{1} + O_n(M_a(1)) \\
\tilde z_n &=& \alpha(-1)\alpha(-4)^2  \mathbf{1} + O_n(M_a(1)) . \label{variable-notation-last}
\end{eqnarray}

We have the following corollary to  Lemma \ref{DLM-commutator-lemma} and Theorem \ref{main-generators-thm}:

\begin{cor}\label{first-commutator-cor} For $V = M_a(1)$, we have 

(i) $x_n = \alpha(-1){\bf 1}+O_n(V)$  and $y_n = \alpha(-1)^2{\bf 1}+O_n(V)$ are central in $A_n(V)$, for any $n \in \mathbb{N}$.

(ii) Let $R = \mathbb{C}[x,y]\langle \tilde y, z, \tilde z\rangle$ be the ring of finite words in the commuting variables $x, y$ and the non commuting variables $\tilde y, z, \tilde z$.    Then  $A_2(V) \cong R/I_2$ for some ideal $I_2$.

(iii) There exists a surjective algebra homomorphism $A_2(V) \cong R/I_2 \twoheadrightarrow A_1(V) \cong R/I_1$ and thus $I_2 \subset I_1$. 
\end{cor}

\begin{proof} By Eqn.\ (\ref{reduce-alpha(-1)}) we have that  $\omega_a = \frac{1}{2} \alpha(-1)^2\mathbf{1}  + a \alpha(-2) \mathbf{1} \approx \frac{1}{2} \alpha(-1)^2 \mathbf{1} - a \alpha(-1) \mathbf{1}$, which is central in $A_n(V)$ by Lemma \ref{DLM-commutator-lemma} (ii).   

For $\alpha(-1){\bf 1}+O_n(V)$, observe that Eqn.\ (\ref{dlmcommutation}) for any $n\in\mathbb{N}$ implies that $\alpha(-1)\mathbf{1} *_n v - v *_n \alpha(-1) \mathbf{1}$ is equivalent mod $O_n(V)$ to $\res_x Y(u,x)v = \alpha(0)v = 0$.  Thus statement (i) holds.  

For statement (ii) we note that  from (i) and Theorem \ref{main-generators-thm}, 
\[ A_2(V) \cong \mathbb{C}[x_2, y_2] \langle \tilde y_2, z_2, \tilde z_2 \rangle/I_2\] for some two-sided ideal $I_2$.  Thus under the identification $x_2 \mapsto x$, $y_2 \mapsto y$, etc., $A_2(V) \cong R/I_2$.  

Statement (iii) follows from Lemma \ref{DLM-commutator-lemma} (i), and the surjection (\ref{surjection}).
\end{proof}

\begin{rem}\label{abuse-notation-remark}
{\em We will sometimes abuse the notation given in (\ref{variable-notation-first})--(\ref{variable-notation-last}) and use, for instance, the notation $x_n$ for both $\alpha(-1) \mathbf{1} + O_n(V)$ and just $\alpha(-1) \mathbf{1}$ when the context is clear.   }
\end{rem}

\subsection{The relations in $A_1(M_a(1))$  for the generating set of $A_2(M_a(1))$}

To make it easier to determine the relations satisfied by the generators of $A_2(M_a(1))$ given in Theorem \ref{main-generators-thm}, we first find the relations these generators satisfy in $A_1(M_a(1))$.  That is, we determine the relations that generate the ideal $I_1\subset R = \mathbb{C}[x,y]\langle \tilde y, z, \tilde z\rangle$ of Corollary \ref{first-commutator-cor}. 

\begin{prop}\label{use-level-one-prop} 
For $V = M_a(1)$, and $R = \mathbb{C}[x,y]\langle \tilde y, z, \tilde z\rangle$, we have
\[A_1(V) \cong R/I_1 \]
where $I_1$ is the two-sided ideal 
\begin{equation}\label{I_1}
I_1 = \big( (x^2  -y)(x^2 - y  + 2) , (x^2 -2y -  \tilde{y}),  (4x^3 -5xy - z), \\
 (3x^3 - 4xy + \tilde{z}) \big).
\end{equation} 
Thus $A_2(V) \cong R/I_2$ with $I_2 \subset I_1$, where $I_1$ is given by (\ref{I_1}). 
%In particular, the following are in $I_1$:
%\begin{eqnarray} 
%(x^2 + \tilde{y}) (x^2 + \tilde y + 4),  \quad  (x^3 + z)(x^3 + 10x + z), \quad 
%(x^3 - \tilde{z})(x^3 + 8x - \tilde z), \label{I_1-1} \\ 
%\tilde y z - z \tilde y, \quad  (\tilde y \tilde z - \tilde z \tilde y), \quad (z \tilde z - \tilde z z), \label{I_1-2} \\
%-12(x^3 - xy - z - \tilde z), \quad 
%-12(x^3 +  2x \tilde y + \tilde z),  \quad ????? .\label{I_1-3} .
%\end{eqnarray}
\end{prop}

\begin{proof}
From  Corollary \ref{first-commutator-cor} we have a surjective algebra homomorphism $\varphi: R \twoheadrightarrow  A_1(V)$.  We also have from \cite{BVY-Heisenberg}  that  $A_1(V) \cong \mathbb{C}[x, y]/((x^2- y)(x^2 - y + 2)) \cong \mathbb{C}[x] \oplus \mathbb{C}[x]$, where the isomorphism is given by $x \mapsto \alpha(-1) {\bf 1}$ and $y \mapsto \alpha(-1)^2 {\bf 1}$.  It follows that 
$((x^2- y)(x^2 - y + 2))$ is in $I_1$.  

By Eqn.\ (\ref{recursion-equation}) with $n = 1$ and $m = 4$, we have 
\begin{equation}\label{recursion-level-1}
\alpha(-4)v \sim_1 - \alpha(-2)v - 2 \alpha(-3)v,
\end{equation} 
which gives, using the abuse of notation as noted in Remark \ref{abuse-notation-remark},
\begin{eqnarray}\label{reduce-4-first}
\ \ \ \ \  \ \ \ \ \tilde y_1 = & \alpha(-4) \alpha(-1) {\bf 1} &\sim_1 - \alpha(-2) \alpha(-1) {\bf 1}  - 2 \alpha(-3) \alpha(-1) {\bf 1}\\
z_1 =& \alpha(-4) \alpha(-1)^2 {\bf 1}  &\sim_1 -\alpha(-2) \alpha(-1)^2 {\bf 1} - 2 \alpha(-3) \alpha(-1)^2 {\bf 1} \\
\tilde z_1 =& \alpha(-4)^2 \alpha(-1) {\bf 1} &\sim_1  \alpha(-4) \left( - \alpha(-2) \alpha(-1) {\bf 1} - 2 \alpha(-3) \alpha(-1) {\bf 1} \right) \label{reduce-4-last} \\
& &\sim_1 \alpha(-2)^2 \alpha(-1) {\bf 1} + 4 \alpha(-3) \alpha(-2) \alpha(-1) {\bf 1} + 4 \alpha(-3)^2 \alpha(-1)  {\bf 1} . \nonumber
\end{eqnarray} 

By Eqn.\  (\ref{reduce-equation-n=1}), or equivalently Eqn.\  (27) in \cite{BVY-Heisenberg},  applied with the appropriate multiplicities of $\alpha(-1), \alpha(-2)$, and $\alpha(-3)$, we have 
\begin{eqnarray}\label{reduce-3-first}
\alpha(-3) \alpha(-1) {\bf 1} &\equiv_1 & - \frac{1}{2}  \left( \alpha(-2)^2 {\bf 1} + 3 \alpha(-1) \alpha(-2) {\bf 1} \right)\\
\alpha(-3) \alpha(-1)^{2}  {\bf 1} &\equiv_1 & - \alpha(-1) \alpha(-2)^2  {\bf 1} - 2  \alpha(-1)^2  \alpha(-2) {\bf 1}  \\
\alpha(-3)\alpha(-2) \alpha(-1) {\bf 1}& \equiv_1 & - \frac{1}{4}  \left( \alpha(-2)^3 {\bf 1} + 5 \alpha(-1) \alpha(-2)^2 {\bf 1} \right) \\
\alpha(-3)^2 \alpha(-1)  {\bf 1}& \equiv_1 &  \frac{1}{2} \left( - \alpha(-2)^2  \alpha(-3)  {\bf 1} + 3 \alpha(-1) \alpha(-2)^2  {\bf 1} \right) \label{reduce-3-last} \\
&\equiv_1&  \frac{1}{2} \left(  \alpha(-2)^3 {\bf 1} + 3 \alpha(-1) \alpha(-2)^2 {\bf 1} \right), \nonumber
\end{eqnarray}
where Eqn.\ (\ref{reduce-equation-n=1}) has been used twice in the last equivalence.  Substituting Eqns.\ (\ref{reduce-3-first})--(\ref{reduce-3-last}) into Eqns.\  (\ref{reduce-4-first})--(\ref{reduce-4-last}), we have
\begin{eqnarray}\label{reduce-4-first-again}
 \tilde y_1 & \equiv_1&  \alpha(-2)^2 {\bf 1} + 2 \alpha(-1) \alpha(-2) {\bf 1} \\
z_1 & \equiv_1& 2\alpha(-1) \alpha(-2)^2  {\bf 1} + 3 \alpha(-1)^2  \alpha(-2) {\bf 1}  \\
\tilde z_1  &\equiv_1&  \alpha(-2)^3 {\bf 1} + 2 \alpha(-1) \alpha(-2)^2 {\bf 1} . \label{reduce-4-last-again}
\end{eqnarray}

By the fact that $L(-1) \alpha(-1)^i {\bf 1} = i \alpha(-1)^{i-1} \alpha(-2) {\bf 1}$, we have that 
\[\alpha(-1)^i {\bf 1} \approx - \alpha(-1)^{i-1} \alpha(-2) {\bf 1} , \]
so that
\begin{equation}\label{first-2-one}
\alpha(-1) \alpha(-2) {\bf 1}  \equiv_1  - \alpha(-1)^2 {\bf 1} . 
\end{equation}

Eqn.\ (37) in \cite{BVY-Heisenberg} (or equivalently using the definition of $\circ_1$, and Eqn.\ (\ref{Y+})  applied to $\alpha(-1)^2 {\bf 1} \circ_1 \alpha(-2)^2 {\bf 1}$ to obtain the relation in $O_1(V)$ given by (34) of \cite{BVY-Heisenberg},  then substituting using  Eqn.\  (\ref{reduce-equation-n=1}) into the relation (34) to obtain Eqn.\ (35) of \cite{BVY-Heisenberg}, then using  (\ref{multiplication-cor-formula})  to determine $\alpha(-1)^2 {\bf 1} *_1 \alpha(-2)^2 {\bf 1}$, and adding this to a multiple of Eqn.\ (35) of \cite{BVY-Heisenberg}), we have
\begin{equation}\label{almost-last-2-one}
 \alpha(-1) \alpha (-2)^2 {\bf 1}  \equiv_1  -2 \alpha(-1)^2 {\bf 1} *_2 \alpha(-2) {\bf 1} + \alpha(-2)^3 {\bf 1}  ,
 \end{equation}
 which is Eqn.\ (37) of \cite{BVY-Heisenberg} with $j = 2$.  
 By Eqn.\ (36) of \cite{BVY-Heisenberg}, or the equivalent derivation of this equation as described above, we also have 
\[\alpha(-1)^2 {\bf 1} *_1 \alpha(-2) {\bf 1} \equiv_1 -3 \alpha(-1)^2 \alpha(-2) {\bf 1} - \alpha(-2)^3 {\bf 1} - 5 \alpha(-1) \alpha(-2)^2 {\bf 1}\]
 which, when combined with Eqn.\ (\ref{almost-last-2-one}), gives
 \begin{eqnarray}
\lefteqn{ \alpha(-1)^2 \alpha(-2) {\bf 1} \   \equiv_1 \  -  \frac{1}{3} \alpha(-1)^2 {\bf 1} *_1 \alpha(-2) {\bf 1} - \frac{1}{3} \alpha(-2)^3 {\bf 1} - \frac{5}{3} \alpha(-1) \alpha(-2)^2 {\bf 1}} \label{closer-to-last}  \\
 &\equiv_1 & -  \frac{1}{3} \alpha(-1)^2 {\bf 1} *_1 \alpha(-2) {\bf 1} - \frac{1}{3} \alpha(-2)^3 {\bf 1} + \frac{10}{3} \alpha(-1)^2 {\bf 1} *_2 \alpha(-2) {\bf 1} \nonumber\\
 & & \quad  - \, \frac{5}{3} \alpha(-2)^3 {\bf 1}  \nonumber\\
 &\equiv_1 & 3 \alpha(-1)^2 {\bf 1} *_1 \alpha(-2) {\bf 1} - 2 \alpha(-2)^3 {\bf 1}  . \nonumber
 \end{eqnarray} 
 
By Eqn.\ (33) in \cite{BVY-Heisenberg}, or equivalently Eqn.\ (\ref{multiplication-cor-formula}) above with $i = t = 1$ and $v = \alpha(-1)^j{\bf 1}$ and using induction on $j$, we have for $j \in \mathbb{N}$
\begin{equation}\label{last-2-one}
\alpha(-2)^j {\bf 1} \equiv_1 (-1)^j ( \alpha(-1) {\bf 1})^j .
\end{equation} 
Substituting Eqn.\ (\ref{last-2-one}) into Eqns.\ (\ref{almost-last-2-one}) and (\ref{closer-to-last}), and using the notation $x_1 = \alpha(-1) {\bf 1} + O_1(V)$ and $y_1 = \alpha(-1)^2 {\bf 1} + O_1(V)$, we have that Eqns.\ (\ref{first-2-one})--(\ref{closer-to-last}) become
\begin{eqnarray*}
\alpha(-1) \alpha(-2) {\bf 1} &  \equiv_1&  - y_1,  \\
\alpha(-1) \alpha(-2)^2 {\bf 1} &  \equiv_1& 2 y_1  *_1 x_1 - x_1^3,\\
 \alpha(-1)^2 \alpha(-2) {\bf 1} & \equiv_1 & -3 y_1 *_1 x_1 + 2 x_1^3.
 \end{eqnarray*}
Substituting these along with Eqn.\ (\ref{last-2-one}) into  Eqns.\  (\ref{reduce-4-first-again})--(\ref{reduce-4-last-again}), we have
\begin{eqnarray}\label{relations-z1}
 \tilde y_1 \  \equiv_1& \!   x_1^2 -2y_1 \\
z_1  \ \equiv_1&  \!  4x_1y_1  - 2x_1^3 - 9 x_1y_1 + 6x_1^3 \! &=  \ 4x_1^3 -5 x_1y_1 \\
\tilde z_1 \  \equiv_1& \!  -x_1^3  + 4 x_1y_1 - 2x_1^3 \! &= \ 4x_1y_1 -3x_1^3. \label{relations-z3}
\end{eqnarray} 
This proves that $I_1 \subseteq \mathrm{ker} \, \varphi$.  To prove that $I_1 = \mathrm{ker} \, \varphi$, we note that
\[
R/I_1 
\cong \mathbb{C}[x,y]/((x^2 - y)(x^2 - y + 2)) \cong A_1(V) .\]

The last statement of the Proposition follows directly from Corollary \ref{first-commutator-cor} (iii). 
\end{proof}

\subsection{Zero mode action on irreducible modules of $M_a(1)$}\label{zero-modes-section}

Recall from (\ref{irreducibles}), that the irreducible modules for $V= M_a(1)$ are given by $M_a(1, \lambda) = M_a(1) \otimes_\mathbb{C} \mathbb{C}_\lambda$ where $\mathbb{C}_{\lambda}$ is the 
one-dimensional $\mathfrak{h}$-module such that $\alpha(0)$  acts as multiplication by $\lambda \in \mathbb{C}$.  That is, letting $v_\lambda \in \mathbb{C}_\lambda$ such that $\mathbb{C}_\lambda = \mathbb{C} v_\lambda$, then $\alpha(0) v_\lambda = \lambda v_\lambda$ and $\alpha(n) v_\lambda = 0$ if $n\geq 1$. Then recalling the functor $\Omega_n$ from (\ref{Omega}), we have that $\Omega_n(M_a(1,\lambda))$ is an $A_n(M_a(1))$-module under the zero mode action.  Since $\alpha(-1)^2 v_\lambda, \alpha(-2) v_\lambda  \in \Omega_2(M_a(1,\lambda))$, we can consider the $A_2(M_a(1))$ module generated by these vectors.  

We have the following lemma giving the zero modes acting on $\Omega_2(M_a(1,\lambda))$:

\begin{lem}\label{o-mode-lemma}
The zero modes corresponding to the action of the elements $x_n, y_n, \tilde y_n, z_n, \tilde z_n \in  A_n(M_a(1))$ for $n \in \mathbb{N}$ that are effective on $\Omega_2(M_a(1,\lambda))$ are given by
\begin{eqnarray}
o(x_n)  &=& o(\alpha(-1){\bf{1}})  \label{zero-mode-first}\\
& =&  \alpha(0)  \nonumber \\
o(y_n)  &=& o(\alpha(-1)^2{\bf{1}})  \\
& =&  \alpha(0)^2 + 2\alpha(-1) \alpha(1) + 2 \alpha(-2)\alpha(2)  \nonumber \\
o(\tilde y_n ) &=&  o(\alpha(-1)\alpha(-4){\bf{1}})  \\
 & =&   -\alpha(0)^2-4\alpha(-1)\alpha(1)-10\alpha(-2)\alpha(2)  \nonumber \\
o(z_n) &= & o(\alpha(-1)^2\alpha(-4){\bf{1}}) \\
&=&   -\alpha(0)^3-8\alpha(-2)\alpha(1)^2-10\alpha(-1)\alpha(0)\alpha(1)  - 22\alpha(-2)\alpha(0)\alpha(2) \nonumber\\
& & \quad - \, 10\alpha(-1)^2\alpha(2)  \nonumber\\
o(\tilde z_n) &=&  o(\alpha(-1)\alpha(-4)^2{\bf{1}})  \label{zero-mode-last}\\
& =&  \alpha(0)^3+16\alpha(-2)\alpha(1)^2+8\alpha(-1)\alpha(0)\alpha(1) +
20\alpha(-2)\alpha(0)\alpha(2) \nonumber . 
\end{eqnarray}

Thus acting on $v_\lambda, u= \alpha(-1) v_\lambda$, $v = \alpha(-1)^2 v_\lambda$ and $w = \alpha(-2) v_\lambda$ in $\Omega_2(M_a(1, \lambda))$, and dropping the subscript for this action, we have
\begin{align*}
x. v_\lambda & =  \lambda v_\lambda, & x. u &= \lambda u,  & x. v &= \lambda v,  & x.w &= \lambda w, \\
y. v_\lambda & = \lambda^2 v_\lambda,  & y. u &= (\lambda^2+2) u, & y. v &= (\lambda^2+4) v,  & y.w &= (\lambda^2+4) w, \\
\tilde y . v_\lambda &= -  \lambda^2 v_\lambda, & \tilde y . u  &= (-\lambda^2-4) u, & \tilde y . v  &=  (-\lambda^2-8) v, & \tilde y . w &= (-\lambda^2-20)w, \\
z. v_\lambda  &=   -\lambda^3 v_\lambda,  & z.u  &=  (-\lambda^3-10 \lambda)u, &  z.v  &=  (-\lambda^3-20\lambda)v -16w,  &  z. w &= -20v \\
& & & & & &  & \ \ \   + \, (-\lambda^3 -44\lambda)w\\
\tilde z . v_\lambda &=  \lambda^3 v_\lambda,  &  \tilde z . u &=  (\lambda^3 + 8 \lambda)u, &  \tilde z . v &=  (\lambda^3+16\lambda)v +32w, & \tilde z . w &= (\lambda^3+40\lambda)w .
\end{align*}
On the polynomials in the ideal $I_1\subset R =  \mathbb{C}[x,y]\langle \tilde y, z, \tilde z\rangle$ giving the structure of $A_1(M_a(1)) \cong R/I_1$, we have
\begin{align*}
(x^2 - y)(x^2 - y + 2).v &= 8v  & (x^2- y)(x^2 - y + 2).w &= 8w \\
(x^2 - 2y - \tilde y). v & = 0 & (x^2 - 2y - \tilde y).w &= 12w \\
(4x^3 - 5xy - z).v &= 16w & (4x^3 - 5xy - z).w &= 20v + 24\lambda w \\
(3x^3 - 4xy + \tilde z).v &= 32w & (3x^3 - 4xy + \tilde z).w &= 24\lambda w
\end{align*}
\end{lem}

\begin{proof} 
These are straightforward calculations as follows: $o(v) = v_{\mathrm{wt} \ v - 1} = \res_x x^{\mathrm{wt} \, v - 1} Y(v,x)$ and so
\begin{eqnarray*}
o(\alpha(-1){\bf{1}})&=& \res_x(\sum_{i\in \mathbb{Z}}\alpha(i)x^{-i-1}) \ = \  \alpha(0),\\
o(\alpha(-1)^2{\bf{1}}) &=& \res_x x \Big(\sum_{i_1,i_2\in \mathbb{Z}}{}_{\circ}^{\circ}\alpha(i_1)\alpha(i_2){}_{\circ}^{\circ}x^{-i_1-i_2-2} \Big) \\
& = & \alpha(0)^2+2\alpha(-1)\alpha(1)+2\alpha(-2)\alpha(2)+ o(h_1), \\
o(\alpha(-1)\alpha(-4){\bf{1}})&=& \res_x x^4 \Big(\sum_{i_1,i_2\in \mathbb{Z}} \frac{(-i_2 - 1)(-i_2-2)(-i_2-3)}{3!} {}_{\circ}^{\circ}\alpha(i_1)\alpha(i_2){}_{\circ}^{\circ}x^{-i_1-i_2-5} \Big) \\
&=& -\alpha(0)^2-4\alpha(-1)\alpha(1)-10\alpha(-2)\alpha(2) + o(h_2) , \\
o(\alpha(-1)^2\alpha(-4){\bf{1}}) &=&  \res_x x^5 \Big(\sum_{i_1,i_2\in \mathbb{Z}} \frac{(-i_3- 1)(-i_3-2)(-i_3-3)}{3!} {}_{\circ}^{\circ}\alpha(i_1)\alpha(i_2) \alpha(i_3) {}_{\circ}^{\circ}x^{-i_1-i_2-6} \Big)\\
&=& -\alpha(0)^3-8\alpha(-2)\alpha(1)^2-10\alpha(-1)\alpha(0)\alpha(1)-22\alpha(-2)\alpha(0)\alpha(2)\\&&-10\alpha(-1)^2\alpha(2) + o(h_3),\\
o(\alpha(-1)\alpha(-4)^2{\bf{1}}) &=& \res_x x^8 \Big(\sum_{i_1,i_2, i_3 \in \mathbb{Z}} \frac{(i_2 + 1)(i_2+2)(i_2+3)(i_3+ 1)(i_3+ 2)(i_3+ 3)}{3!^2} \\
& & {}_{\circ}^{\circ}\alpha(i_1)\alpha(i_2)\alpha(i_3)  {}_{\circ}^{\circ}x^{-i_1-i_2- i_3 -9} \Big) \\
&=& \alpha(0)^3+16\alpha(-2)\alpha(1)^2+8\alpha(-1)\alpha(0)\alpha(1)+20\alpha(-2)\alpha(0)\alpha(2)+ o(h_4) , 
\end{eqnarray*}
where the operators $o(h_j)$, for $j = 1, \dots, 4$, denote modes involving $\alpha(k)$ for $k>2$ and thus act as zero on $\Omega_2(M_a(1, \lambda)$, giving Eqns. (\ref{zero-mode-first})--(\ref{zero-mode-last}). 

The remainder of the results are straightforward calculations following from the definition of $M_a(1, \lambda)$ and $M_a(1)$.  
\end{proof} 

Motivated by the zero mode action, we make the following change of variables 
\begin{eqnarray}\label{change-var1-n}
Y_n &=& \frac{1}{12} (x_n^2 - 2y_n - \tilde y_n),\\
Z_n &=&  \frac{1}{32} ( x_n^3 + 2x_n\tilde y_n + \tilde z_n) \\
&=&  \frac{1}{32} \left( (3x_n^3 - 4x_ny_n + \tilde z_n) - 2x_n(x_n^2 - 2y_n - \tilde y_n) \right), \nonumber  \\
\ \ \  W_n &=& -\frac{1}{40}  \left(2z_n + \tilde z_n + 2x_ny_n - 2x_n \tilde y_n - 3x_n^3 \right) \label{change-var3-n} \\
 &=&  -  \frac{1}{40} \left( 2x_n(x_n^2 - 2y_n - \tilde y_n)-  2(4x_n^3 - 5x_ny_n - z_n) + (3x_n^3 - 4x_ny_n + \tilde z_n)\right) \nonumber,
\end{eqnarray}
so that $R_n = \mathbb{C}[x_n,y_n]\langle \tilde y_n, z_n, \tilde z_n \rangle = \mathbb{C}[x_n,y_n]\langle Y_n, Z_n, W_n \rangle$.  In particular, $Y_2$ is homogeneous of degree 2, and $Z_2$ and $W_2$ are both homogeneous of degree 3 in the $\alpha(-j)$'s for $j \in \mathbb{Z}_+$.

With these change of variables we note the following Corollary to Proposition \ref{use-level-one-prop} and Lemma \ref{o-mode-lemma}.

\begin{cor}\label{level-one-cor}
For $V = M_a(1)$ and $R  = \mathbb{C}[x,y]\langle Y, Z, W \rangle$ we have $A_1(V) \cong R/I_1$ where $I_1$ is the two-sided ideal 
\begin{equation}
I_1 = \big( (x^2 - y)(x^2 - y + 2), Y, Z, W \big),
\end{equation}
$A_2(V) \cong R/I_2$ for $I_2 \subset I_1$, and the zero mode actions of $Y, Z$, and $W$ on $v = \alpha(-1)^2 v_\lambda$ and $w = \alpha(-2)v_\lambda \in \Omega_2(M_a(1, \lambda))$ are given by
\begin{equation}
Y.v = W.v = Z.w = 0, \quad Y.w = Z.v = w, \quad \mbox{and} \quad  
W.w = v.
\end{equation}
\end{cor}

\subsection{Relations satisfied by the elements of the generating set of $A_2(M_a(1))$ up to lower order terms}

We next determine the highest degree terms of the polynomial relations satisfied by the generators of $A_2(M_a(1))$ in $\mathbb{C}[x,y]\langle Y, Z, W \rangle$.  We do this by determining some elements 
 in $O_2(M_a(1))$ up to ``lower degree terms" meaning up to terms in $F_r(\mathbf{1})$ for some $r \in \mathbb{N}$, where $F_r(v)$ for $v \in M_a(1)$ is defined as in the beginning of Section 4.   Then in the next section, we will prove the entire polynomial relations.

\begin{lem}\label{lower-order-lemY}  
For $V = M_a(1)$, and for $x, y, Y, Z, W \in V + O_2(V)$ as defined in Eqns.\ (\ref{variable-notation-first})-(\ref{variable-notation-last}) and (\ref{change-var1-n})-(\ref{change-var3-n}) with $n = 2$ assumed here, we have 
\begin{eqnarray}
(x^2 - y) Y, \ Y^2  & \in & O_2(V) + F_2(\mathbf{1}) ,  \label{YZW1} \\
(x^2 - y)Z, \ (x^2 - y)W, \ ZY, \ YW & \in & O_2(V) + F_3(\mathbf{1}), \label{YZW2} \\
(x^2 - y)^3, \ Z^2, \ W^2,  ZW &\in & O_2(V) + F_4(\mathbf{1}). \label{YZW3} 
\end{eqnarray}
\end{lem}

\begin{proof}
We first prove line \eqref{YZW1}.  To show $(x^2 - y) Y\in O_2(V) +F_2(\bf{1}),$ we use the associativity of $*_2$, the fact that ${\bf 1}$ is a multiplicative identity with respect to $*_2$, apply Lemmas  \ref{Ymult} and  \ref{appendixlem1}, and then use \eqref{O??}, which gives
\begin{eqnarray*}
(x^2-y)Y &=& (x^2-y)  *_2 (Y * _2 {\bf 1} )  \ \sim_2 \   (x^2 - y) *_2 \big( \frac{1}{2} B(C+D){\bf 1} + f_0({\bf 1}) \big) \\
&\sim_2&  (-6A(B+C)-6B(C+D)-2B(A+B))\frac{1}{2}B(C+D){\bf 1}+f_2(\bf {1})\\
&\sim_2& (-3AB(B+C)(C+D)-B^2(A+B)(C+D)){\bf 1}+f_2({\bf 1}),
\end{eqnarray*}
for some $f_2({\bf 1})\in F_2({\bf 1})$ since $Y \in F_2({\bf 1})$.  By  Eqn.\ \eqref{coreq1}, we have $-3AB(B+C)(C+D){\bf 1} \sim_20$ and by Eqn.\ \eqref{coreq6}, we have $B^2(A+B)(C+D){\bf 1}   \sim_2 0$, proving $(x^2-y)Y\in O_2(V) + F_2(\mathbf{1}).$

The fact that $Y^2 \in O_2(V) + F_2({\bf 1})$, follows from  Lemma \ref{Ymult} and Eqn.\ \eqref{O??} which give   
\[Y^2 =  Y *_2 (Y *_2 {\bf 1}) \  \sim_2  \ Y *_2 \big( \frac{1}{2} B(C + D) {\bf 1} + f_0({\bf 1}) \big)  \sim_2 \frac{1}{4} B^2 (C + D)^2 {\bf 1} + f_2({\bf 1}) \sim_2  f_2({\bf 1}), \]
for some $f_2({\bf 1}) \in F_2( {\bf 1})$.

We next prove line \eqref{YZW2}. 
To prove $(x^2-y)Z \in O_2(V) + F_3(\mathbf{1})$, we apply Lemma \ref{Zmult}, then Lemma \ref{appendixlem1}, and then Eqn.\ (\ref{O21}), giving
\begin{eqnarray*}
 (x^2-y)Z &=& (x^2-y) *_2 ( Z *_2 {\bf 1} ) \ \sim_2 \  (x^2-y) *_2 \Big( \frac{1}{2} BC^2 {\bf 1} + f_1({\bf 1}) \Big) \\
&\sim_2 & \big(-6A(B+C)-6B(C+D)-2B(A+B) \big)\frac{1}{2}  BC^2{\bf {1}}+f_3(\mathbf{1})\\
&\sim_2 &\frac{1}{2}\big(-6ABC^2(B+C)-2B^2C^2(A+B) \big){{\bf 1}}+f_3(\mathbf{1}),
\end{eqnarray*}
for some  $f_1({\bf 1}) \in F_1({\bf 1})$ and $f_3(\mathbf{1})\in F_3(\mathbf{1})$. 
By Lemma \ref{C^3}, we have
\begin{equation*}
-6ABC^2(B+C){\bf{1}}\sim_2-6AC^2(B+C)^2{\bf{1}}+  f_3(\mathbf{1}) 
\end{equation*}
for some $f_3(\mathbf{1})\in F_3(\mathbf{1})$, since $AB {\bf 1}, AC {\bf 1} \in F_2({\bf 1})$, and similiarly
\begin{equation*}
-2B^2C^2(A+B){\bf{1}} \sim_2   -2B(B+C)C^2 (A+B){\bf 1} +  g_3(\mathbf{1})
\end{equation*}
for some $g_3(\mathbf{1})\in F_3(\mathbf{1})$.   Therefore, 
\begin{eqnarray*}
(x^2-y)Z \sim_2 \frac{1}{2}  \big(-6AC^2(B+C)^2-2B(B+C)C^2(A+B) \big) {\bf 1} + g_3({\bf 1}),\end{eqnarray*} 
for some $g_3({\bf 1})\in F_3({\bf 1}).$ Then $(x^2-y) Z \sim_2 g_3({\bf 1})$ by  Eqns.\ \eqref{coreq3} and \eqref{coreq5}.

Applying  Lemma \ref{Wmult}, and then  \ref{appendixlem1}, expanding, and then using Eqns.\  \eqref{coreq2}, \eqref{O??}, and \eqref{coreq6}, we have 
\begin{eqnarray*}
(x^2-y)W &=& (x^2-y) *_2 ( W *_2 {\bf 1})  \ \equiv_2 \  (x^2 - y) *_2 \Big( \frac{1}{2} A^2(C+D) {\bf 1} + f_1({\bf 1})  \Big)\\
&\sim_2& \big(-6A(B+C)-6B (C+D)-2B(A+B)\big)\frac{1}{2}A^2(C+D) {\bf 1} +f_3({\bf 1})\\
&=& -3A^3(B+C)(C+D)-3A^2B(C+D)^2-A^2B(A+B)(C+D){\bf 1 }+f_3({\bf 1})\\
& \sim_2  & 0 +  f_3({\bf 1}),
\end{eqnarray*}
 for some $f_1({\bf 1}) \in F_1({\bf 1})$ and $f_3(\mathbf{1})\in F_3(\mathbf{1})$. 

That $ZY \in  O_2(V) + F_3(\mathbf{1})$ follows from Lemmas \ref{Ymult}  and \ref{Zmult},  and Eqn.\ \eqref{O24} by
\begin{equation*}
ZY = Z *_2 (Y *_2 {\bf 1}) \sim_2  2BC^2(\frac{1}{2} B(C+D)) {\bf 1} + f_{3}( {\bf 1}) \sim_2 0 + f_{3}({\bf 1}),
\end{equation*}
for some $f_{3}({\bf 1})  \in F_{3}({\bf 1})$.

Similarly, that $YW \in O_2(V)+ F_3(\mathbf{1})$ follows from Lemmas \ref{Ymult} and \ref{Wmult}, and Eqn.\ \eqref{O??}, by
\begin{equation*}
YW = Y *_2 (W *_2 {\bf 1}) \equiv_2 \frac{1}{2} B(C+D) \frac{1}{2} A^2(C+D){\bf 1} + f_{3}( {\bf 1}) \sim_2 0 + f_{3}({\bf 1}),
\end{equation*}
for some $f_{3}({\bf 1})  \in F_{3}({\bf 1})$.

We next prove \eqref{YZW3}.  That $W^2\in O_2(V) + F_4(\mathbf{1})$ follows from Lemma \ref{Wmult} and \eqref{O??}, by
\begin{equation*}
W^2 = W *_2 ( W *_2 {\bf 1}) \equiv_2 \frac{1}{4} A^4 (C+D)^2 {\bf 1} + f_4 ({\bf 1}) \sim_2 0 + f_4({\bf 1}) 
\end{equation*}
for some $f_4({\bf 1}) \in F_4 ({\bf 1})$. 

We have by Lemma \ref{Zmult},  and then using Lemma \ref{C^3}, that 
\begin{eqnarray*}
Z^2  = Z *_2(Z *_2 {\bf 1})  \sim_2 \big(\frac{1}{2}BC^2\big)^2{\bf 1}+f_4({\bf 1}) \sim_2 g_{4} ({\bf 1}) + f_4({\bf 1})
\end{eqnarray*} 
for some $f_4({\bf 1}), g_{4}({\bf 1}) \in F_4({\bf 1})$, giving $Z^2 \in O_2(V) + F_4({\bf 1})$.

The fact that  $ZW \in  O_2(V) + F_4 (\mathbf{1})$, follows from Lemmas \ref{Wmult} and \ref{Zmult}, and  Eqn.\ \eqref{O21} by 
\begin{eqnarray*}
Z *_2 W *_2 {\bf 1} \sim_2 (\frac{1}{2} BC^2)\frac{1}{2} A^2(C+D) {\bf 1} + f_4({\bf 1}) \sim_2 f_4({\bf 1})
\end{eqnarray*}
for some $f_4({\bf 1}) \in F_4({\bf 1})$.

Finally we have left to show that $(x^2 - y)^3\in O_2(V) + F_4(\mathbf{1})$.  By Lemma \ref{appendixlem1}, we have that 
\begin{eqnarray*}
\lefteqn{(x^2-y)^3  }\\
&=& (x^2-y) *_2 (x^2 - y) *_2 (x^2 - y) {\bf 1 }  \\
& \sim_2& -8  (3A(B+C)+ 3B(C+D)+ B(A+B))^3{\bf 1}+f_4({\bf 1}) \\
&=&  -8 ( 27 A^3(B+C)^3 + 81 A^2(B+C)^2 B(C+D) + 81A(B+C)B^2(C+D)^2 \\
& & \quad + \,  27 B^3(C+D)^3 + 27 A^2(B+C)^2 B(A+ B) +  54A(B+C) B(C+D) B(A+B)    \\
& & \quad + \,   27 B^2(C+D)^2 B(A+B)  + 9 A(B+C) B^2 (A+B)^2 + 9 B(C+D) B^2(A+B)^2 \\
& & \quad + \,  B^3 (A+B)^3) {\bf 1} +  f_4({\bf 1})\\
& \sim_2 & -8B^3(A+B)^3 {\bf 1} +  f_4({\bf 1}) 
\end{eqnarray*}
for some $f_4({\bf 1}) \in F_4 ({\bf 1})$, where the last line follows by showing each of the terms, except the last term,  equivalent to 0 as follows:

$A^3(B+C)^3 {\bf 1} \sim_2 - 2A^3B(B+C)(C+D) {\bf 1}$ by (\ref{O23}) and this is $\sim_2 0$ by (\ref{coreq1}).

$A^2(B+C)^2 B(C+D) {\bf 1} \sim_2 0 $ by (\ref{coreq2}).

$A(B+C)B^2(C+D)^2 {\bf 1} \sim_2 0 $  by (\ref{O??}).

$B^3(C+D)^3 {\bf 1}  \sim_2 0 $  by (\ref{O??}). 

$A^2(B+C)^2 B(A+ B ) {\bf 1} \sim_2 0 $  by (\ref{coreq8}).

$A(B+C) B(C+D) B(A+B)  {\bf 1}\sim_2 0 $   by (\ref{coreq8}).

$B^2(C+D)^2 B(A+B){\bf 1} \sim_2 0  $  by (\ref{O??}).

$A(B+C) B^2 (A+B)^2 {\bf 1} \sim_2 0$  by (\ref{coreq8}). 

$B(C+D) B^2(A+B)^2 {\bf 1}  \sim _2  0$ by (\ref{coreq9}).\\
And  then for the last term, we use Eqn.\ (\ref{O24}), and then  Eqns.\ (\ref{O21}) and (\ref{O??}), and then finally Lemma \ref{C^3} to obtain
\begin{eqnarray*}
B^3 (A+B)^3  {\bf 1} &=& (B(A+B))^3 {\bf 1} \ \sim_2 \   \big(- C (A+B) - A(C+D) \big)^3 {\bf 1}\\
& \sim_2 & -C^3(A+B)^3{\bf 1}  \ \sim_2 \ g_4({\bf 1}),
\end{eqnarray*}
for some $g_4 ({\bf 1}) \in F_4({\bf 1})$, 
giving the result.  
\end{proof}

\subsection{Relations satisfied by the generating set of $A_2(M_a(1))$}

We prove that certain elements are in $I_2$ for $A_2(M_1(1)) \cong \mathbb{C}[x,y]\langle Y, Z, W \rangle/I_2$ in several stages. In general, we do this by using Lemma \ref{lower-order-lemY} in conjunction with the following observation to determine the relations in $A_2(M_a(1))$ satisfied by the generators $x,y, Y, Z, W$:

\begin{cor} Let $V = M_a(1)$ and let $u,v \in V$ be such that $u \in F_r({\bf 1})$ and $v \in F_s({\bf 1})$, with $r,s \in \mathbb{N}$.  By the definition of $*_2$, the product $u *_2 v + O_2(V) \in A_2(V) \cong \mathbb{C}[x,y] \langle Y, Z, W \rangle/I_2$  is in $F_{r+s}({\bf 1}) + O_2(V)$, thus modulo $O_2(V)$ is equivalent to a linear combination of words  in $x, y, Y, Z, W$ that is of total degree less than or equal to $r+s$ where   $x \in F_1({\bf 1}) + O_2(V)$ has degree one, $y, Y \in F_2({\bf 1}) + O_2(V)$ each have degree 2, and $Z, W \in F_3({\bf 1}) + O_2(V)$, each have degree 3.  
\end{cor}

In other words, we can use the degree of polynomials in  $x, y, Y, Z, W$ in conjunction with the filtration $F_r({\bf 1})$ of $M_a(1)$ and the higher order terms of the polynomial relations given in Lemma \ref{lower-order-lemY} to completely determine the lower order terms in the relations, as we now proceed to do.

\begin{prop}\label{first-relations-prop} 
The following are elements in $I_2$ for $A_2(M_a(1))\cong \mathbb{C}[x,y]\langle Y, Z, W \rangle/I_2$
\begin{equation}
(x^2 - y + 4) Y, \quad \mathrm{and} \quad Y^2 - Y.
\end{equation}
\end{prop}

\begin{proof}
From Lemma \ref{lower-order-lemY} we have $(x^2 - y)Y \in O_2(V) + F_2(\mathbf{1})$, and the fact that $I_2 \subset I_1$, but $(x^2 - y)Y \in I_1$ already since $Y \in I_1$,  implies that there exists $f_2(\mathbf{1}) \in I_1 \cap F_2(\mathbf{1})$ of the form $f_2(\mathbf{1}) = c Y$ for some constant $c$ such that 
\begin{equation*}
(x^2 - y) Y  \equiv_2  f_2(\mathbf{1} ) \ = \ cY.
\end{equation*}
Acting on $w = \alpha(-2)v_\lambda$, we have by Lemma \ref{o-mode-lemma} and Corollary \ref{level-one-cor} that for all $\lambda \in \mathbb{C}$
\[-4w \ = \  (\lambda^2 - (\lambda^2 + 4))w  \ =  \ (x^2 - y).w  \ = \  (x^2 - y)Y.w  = cY.w  \  = \  cw ,\]
implying $-4 = c$, and giving $(x^2 - y) Y  \equiv_2  -4Y$,
i.e. $(x^2 - y + 4)Y \equiv_2 0$.

Similarly $Y^2 \in O_2(V) + F_2(\mathbf{1})$, and also $Y^2 \in I_1$,  implies that there exists a constant $c$ such that 
\begin{equation*}
Y^2 \equiv_2  cY .
\end{equation*} 
Acting on $w = \alpha(-2) v_\lambda$, we have 
\begin{equation*}
w = Y.w = Y^2.w  = c Y.w   =  c w,
\end{equation*} 
implying $c = 1$, i.e. $Y^2 - Y \equiv_2 0$.  
\end{proof}

Now we proceed to determine the commutator relations involving $Y$.  First we prove a general statement regarding commutation relations in $A_2(M_a(1))$.  

\begin{lem}\label{reduce-commutators-lemma}
For $i,j \in \mathbb{Z}_+$, let $u + O_2(V), v + O_2(V) \in A_2(V) = V/ O_2(V)$ be of homogeneous degree $i$ and $j$, respectively in the the $\alpha(-k)$ for $k \in \mathbb{Z}_+$. Then modulo $O_2(V)$
\[ u *_2 v - v *_2 u  \in F_{i+ j - 2}(\bf{1}) .\]
\end{lem}

\begin{proof} Since ${\bf 1}$ is a multiplicative identity with respect to $*_2$, we have $u*_2v-v*_2u=u*_2v*_2{\bf 1}-v*_2u*_2{\bf 1}$. The formulas for multiplication by $u$ and $v$ imply that $u *_2 w$ and $v *_2 w$ are linear combinations of $i$ products, respectively $j$ products, of modes acting on any $w \in A_n(V)$, and any regular modes commute with each other. Therefore, $u*_2v-v*_2u=u*_2v*_2{\bf 1}-v*_2u*_2{\bf 1}$ consists only of linear combinations of $i+j$ modes acting on the vacuum, where at least one mode is singular. A singular mode acting on any regular modes either decreases the number of modes by two or is zero.  The result follows.
\end{proof}

\begin{prop}\label{first-commutators-prop}
 We have the following commutation relations in $A_2(V)$:
\begin{equation}
YZ-ZY \equiv_2 Z , \qquad YW - WY \equiv_2 -  W .
\end{equation}
\end{prop}

\begin{proof}
By Lemma \ref{reduce-commutators-lemma}, the fact that $Y$ and $Z$ have degree 2 and 3 respectively and commute in $A_1(M_a(1))$, the fact that $I_2 \subset I_1$ as given in Corollary \ref{level-one-cor},  and the fact that $YZ - ZY \in I_1$, we have
\[ YZ-ZY \equiv_2   (c_1x+c_2)Y + c_3 Z   +c_4W\] 
for some $c_1, \cdots, c_4\in \mathbb{C}$.

Acting on $v = \alpha(-1)^2v_\lambda \in M(1, \lambda)$ by the zero modes given by Lemma \ref{o-mode-lemma} and Corollary \ref{level-one-cor},  we have that for all $\lambda \in \mathbb{C}$ 
\begin{eqnarray*}
w \ = \ Y.w  - Z.0 \ = \ (YZ - ZY).w \ =  \  c_3w,
\end{eqnarray*}
which implies $c_3 = 1$. Then acting on $w = \alpha(-2) v_\lambda$, we have
\[0 \ = \ Y.  0 - Z.  w \   = \  (YZ- ZY).w \  = \ (c_1 x +  c_2).w + c_4 \ = \ (c_1\lambda+c_2)w+c_4v ,\]
which implies that  $c_1=c_2 = c_4 = 0$, or equivalently $YZ-ZY \equiv_2  Z $.

For $YW-WY$, similarly to the proof above for the commutator $YZ - ZY$,  by Lemma \ref{reduce-commutators-lemma}, the fact that $Y$ and $W$ have degree 2 and 3 respectively, and the fact that $I_2 \subset I_1$ as given in Corollary \ref{level-one-cor}, there are constants $c_1, \cdots, c_4$ such that 
\[YW - WY \equiv_2  (c_1x+c_2)Y + c_3 Z   +c_4W.\]

Acting on $v = \alpha(-1)^2v_\lambda \in M(1, \lambda)$ by the zero modes given by Lemma \ref{o-mode-lemma} and Corollary \ref{level-one-cor},  we have that for all $\lambda \in \mathbb{C}$ 
\[ 0 \ =  \ Y.0   - W.0  \ = \ (YW -WY). v \ = \ c_3w \]
which implies $c_3 = 0$. Then acting on $w = \alpha(-2) v_\lambda$, we have
\[ -v \ = \  Y.v - W.w  \ = \  (YW - WY).w  \ =  \  (c_1 x + c_2).w + c_4v\ =\ (c_1\lambda +c_2) w+  c_4 v,\]
which implies that  $c_1 = c_2 = 0$ and $c_4= -1$, or equivalently $YW-WY \equiv_2  - W$.
\end{proof}

Next we prove the remaining relations in the next three propositions:

\begin{prop}\label{second-relations-prop}
The following are elements in $I_2$ for $A_2(M_a(1))\cong \mathbb{C}[x,y]\langle Y, Z, W \rangle/I_2$
\begin{equation}
(x^2 - y + 4) Z,  \  (x^2 - y + 4) W, \ ZY, \ \mbox{and} \  YW .
\end{equation}
\end{prop}

\begin{proof}  
Since $I_2 \subset I_1$, and by Proposition \ref{first-relations-prop}, we have $(x^2 - y+ 4)Y, Y^2 - Y  \in I_2$, it follows that any element in $f_3(\mathbf{1}) \in F_3(\mathbf{1})$ that is also in $I_1$ is equivalent modulo $O_2(V)$ to 
\begin{equation}\label{general-f_4}
 f_3(\mathbf{1}) \equiv_2  (c_1x + c_2  )Y + c_3Z + c_4W ,
 \end{equation}
 for some constants $c_1, \dots, c_4\in \mathbb{C}$.  

We note that acting on $v = \alpha(-1)^2v_\lambda \in M(1, \lambda)$ by the zero modes given by Lemma \ref{o-mode-lemma} and Corollary \ref{level-one-cor},  we have that for all $\lambda \in \mathbb{C}$ 
\begin{eqnarray}\label{f_4.v}
f_3(\mathbf{1}).v  &= & (c_1x + c_2  )Y.v + c_3Z.v + c_4W.v  \  = \ c_3 w,
\end{eqnarray}
and acting on $w = \alpha(-2) v_\lambda$, we have
\begin{eqnarray}\label{f_4.w}
f_3(\mathbf{1}).w  &=&  (c_1x + c_2  )Y.w + c_3Z.w + c_4W.w \ = \  (c_1x + c_2  )w + c_4 v  \\
&=&   (c_1 \lambda + c_2  )w + c_4 v  \nonumber
\end{eqnarray}

Since by Lemma \ref{lower-order-lemY} we have $(x^2 - y)Z \in O_2(V) + F_3(\mathbf{1})$, and that $(x^2 - y)Z \in I_1$, we have that there exists $f_3(\mathbf{1}) \in F_3(\mathbf{1})$ of the form (\ref{general-f_4}), such that $(x^2 - y)Z \equiv_2 f_3(\mathbf{1})$.  Thus 
acting on $v = \alpha(-1)^2v_\lambda \in M(1, \lambda)$ by the zero modes given by Lemma \ref{o-mode-lemma} and Corollary \ref{level-one-cor},  and using Eqns.\ (\ref{f_4.v}) and (\ref{f_4.w}), we have that for all $\lambda \in \mathbb{C}$ 
\[
-4w \ = \  (\lambda^2 - (\lambda^2  + 4))w  \   =  \ (x^2 - y).w    \ = \    (x^2 - y)Z.v \ = \  f_3(\mathbf{1}).v \ = \  c_3w ,\]
implying $c_3 = -4$.   Then acting on $w = \alpha(-2) v_\lambda$, we have
\[
0 \ = \ (x^2 - y).0  \ =   \    (  x^2 - y)Z.w  \ = \  f_3(\mathbf{1}).w \ = \ (c_1 \lambda + c_2  )w + c_4 v,\]
implying $c_1 = c_2  = c_4 = 0$, i.e. $(x^2 - y)Z \equiv_2 -4Z$.

Similarly, $(x^2 - y)W \in O_2(V) + F_3(\mathbf{1})$ and $(x^2 - y)W \in I_1$ implies there exists $f_3(\mathbf{1}) \in F_3(\mathbf{1})$ of the form (\ref{general-f_4}), such that $(x^2 - y)W \equiv_2 f_3(\mathbf{1})$.  Thus acting on $v = \alpha(-1)^2v_\lambda \in M(1, \lambda)$,  we have that for all $\lambda \in \mathbb{C}$ 
\[0 \ = \ (x^2 - y).0    \ = \    (x^2 - y)W.v \ = \ f_3(\mathbf{1}).v \  = \  c_3w ,\]
implying $c_3 = 0$.   Then acting on $w = \alpha(-2) v_\lambda$, we have
\[ -4v  \ =\   (\lambda^2 - (\lambda^2 + 4))v  \  =   \ (x^2 - y).v  \ =   \    (  x^2 - y)W.w   \ = \  f_3(\mathbf{1}).w \ = \  (c_1 \lambda + c_2  )w + c_4 v ,\]
implying $c_1 = c_2 = 0$ and $c_4 = -4$ i.e. $(x^2 - y)W \equiv_2 -4W$.

To prove $ZY \in I_2$, we have that $ZY \in O_2(V) + F_3(\mathbf{1})$ and $ZY \in I_1$ implies there exists $f_3(\mathbf{1}) \in F_3(\mathbf{1})$ of the form (\ref{general-f_4}), such that $ZY \equiv_2 f_3(\mathbf{1})$.   Acting on $v = \alpha(-1)^2v_\lambda \in M(1, \lambda)$,  we have that for all $\lambda \in \mathbb{C}$ 
\[ 0  \ = \   Z.0    \ = \    ZY.v \ = \ f_3(\mathbf{1}).v \ = \  c_3 w ,\]
implying $c_3 = 0$.   Then acting on $w = \alpha(-2) v_\lambda$, we have
\begin{eqnarray*}
0 &=& Z.w  \  =  ZY.w  \ =   \ f_3(\mathbf{1}).w  \ = \ (c_1 \lambda + c_2  )w + c_4 v ,
\end{eqnarray*}
implying $c_1 = c_2 =   c_4 =  0$,  i.e. $ZY \equiv_2 0$. 

And finally, the fact that $YW.v = Y.0 = 0$ and $YW.w = Y.v = 0$ similarly gives $YW \equiv_2 0$.  
\end{proof}

\begin{prop}\label{last-commutators-prop}
We have the following commutation relation in $A_2(V)$:
\begin{equation} 
ZW-WZ \equiv_2  - \frac{1}{8}(x^2 - y)(x^2 - y + 2)  + 2 Y  . 
 \end{equation}
\end{prop} 

\begin{proof}
By Lemma \ref{reduce-commutators-lemma}, the fact that $Z$ and $W$ each have degree 3,  $I_2 \subset I_1$ as given in Proposition \ref{use-level-one-prop}, and the relations in $I_2$ already determined  by Propositions \ref{first-relations-prop}, \ref{first-commutators-prop}, and \ref{second-relations-prop},  we have that there exist constants $c_1, \dots, c_8$ such that 
\begin{eqnarray*}
ZW-WZ &\equiv_2& c_1 (x^2 - y)(x^2 - y + 2) + (c_2 y  + c_3 x  + c_4) Y  +  ( c_5 x + c_6) Z  +( c_{7} x + c_{8})  W .
 \end{eqnarray*}

Acting on $v = \alpha(-1)^2v_\lambda \in M(1, \lambda)$ by the zero modes given by Lemma \ref{o-mode-lemma} and Corollary \ref{level-one-cor},  we have that for all $\lambda \in \mathbb{C}$ 
\[
-v \ =\  Z.0 - W.w = (ZW-WZ).v \ =  \ 8c_1v   +  ( c_5 x  + c_6).w \ 
= \ 8c_1v   +  ( c_5\lambda + c_6)w ,\]
which implies $c_5 = c_6 = 0$ and $8c_1 = -1$.  Thus
\[
ZW-WZ  \ \equiv_2 \ - \frac{1}{8} (x^2 - y)(x^2 - y + 2) + (c_2 y  + c_3 x  + c_4) Y 
+  ( c_{7} x  + c_{8} x )  W . \]
Then acting on $w = \alpha(-2) v_\lambda$, we have
\begin{eqnarray*}
w &=& Z.v - W.0 = (ZW-WZ).w \ = \   -w + (c_2 y  + c_3 x  + c_4).w  +  ( c_{7} x + c_{8}).v\\
%&=& -w  + c_2(\lambda^2 +4)w  + c_3 \lambda w  + c_4 w +  ( c_{7}\lambda   + c_{8})v\\
&=& (c_2 \lambda^2 + c_3\lambda + (-1+4c_2 + c_4) ) w  +  ( c_{7} \lambda + c_{8} ) v
\end{eqnarray*}
which implies $c_2 = c_3 = c_7 = c_8 = 0$ and $-1+c_4 = 1$, giving the result. 
 \end{proof}

Finally, we prove what will be the last set of relations in $R = \mathbb{C}[x,y]\langle Y, Z , W \rangle$ giving the structure of $A_2(M_a(1)) \cong R/I_2$ as we will show in Section \ref{structure-section}. 

\begin{prop}\label{last-relations-prop}
The following are elements in $I_2$ for $A_2(M_a(1))\cong \mathbb{C}[x,y]\langle Y, Z, W \rangle/I_2$
\[(x^2 - y)(x^2 - y + 2)(x^2 - y + 4), \ Z^2, \ W^2, \ \mbox{and} \ ZW - Y.\]
\end{prop}

\begin{proof} $I_2 \subset I_1$ as given in Proposition \ref{use-level-one-prop}, and by the relations in $I_2$ already determined  by Propositions \ref{first-relations-prop}, \ref{first-commutators-prop},  \ref{second-relations-prop}, and \ref{last-commutators-prop},  any element in $f_4(\mathbf{1}) \in F_4(\mathbf{1})$ that is also in $I_1$ is equivalent modulo $O_2(V)$ to 
\begin{eqnarray}\label{general-f_5}
 f_4(\mathbf{1}) & \equiv_2&  c_1 (x^2 - y)(x^2 - y + 2)  + (c_2y + c_3 x + c_4  )Y +    (c_5x + c_6)Z  \\
 & & \quad + \,   ( c_7x + c_8)W ,  \nonumber
 \end{eqnarray}
 for some constants $c_1,\dots, c_8 \in \mathbb{C}$.

We note that acting on $v = \alpha(-1)^2v_\lambda \in M(1, \lambda)$ by the zero modes given by Lemma \ref{o-mode-lemma} and Corollary \ref{level-one-cor},  we have that for all $\lambda \in \mathbb{C}$ 
\begin{equation}\label{f_5.v}
f_4(\mathbf{1}).v  \ = \ 8c_1 v +  (c_5x + c_6).w \ = \ 8 c_1 v + (c_5 \lambda + c_6) w , 
\end{equation}
and acting on $w = \alpha(-2) v_\lambda$, we have
\begin{eqnarray}\label{f_5.w}
f_5(\mathbf{1}).w  &=& 8c_1 w + (c_2y + c_3 x + c_4  ).w  + ( c_7x + c_8).v  \\
 &=&  8c_1w + (c_2 (\lambda^2 + 4)  + c_3 \lambda + c_4  )w +  (c_7 \lambda + c_8)v  \nonumber.
\end{eqnarray}

Since by Lemma \ref{lower-order-lemY} we have $Z^2 \in O_2(V) + F_4(\mathbf{1})$, but also $Z^2 \in I_1$, there exists $f_4(\mathbf{1}) \in F_4({\bf 1})$ of the form (\ref{general-f_5}), such that $Z^2 \equiv_2 f_4(\mathbf{1})$.  Thus acting on $v = \alpha(-1)^2v_\lambda \in M(1, \lambda)$ by the zero modes given by Lemma \ref{o-mode-lemma} and Corollary \ref{level-one-cor},  and using Eqns.\ (\ref{f_5.v}) and (\ref{f_5.w}), we have that for all $\lambda \in \mathbb{C}$ 
\[0  \ = \ Z.w  \  =  \  Z^2.v \ = \ f_4(\mathbf{1}).v \ =  \ 8 c_1 v + (c_5 \lambda + c_6) w ,\]
implying $c_1 = c_ 5   = c_6= 0$.  Then acting on  $w = \alpha(-2) v_\lambda$, we have
\[0 = Z.0  \  =  \ Z^2.w \ = \ f_4(\mathbf{1}).w \ = \   (c_2 (\lambda^2 + 4)  + c_3 \lambda + c_4  )w +  (c_7 \lambda + c_8)v ,\]
implying $c_2 = c_3 = c_4 = c_7 = c_8 = 0$.  Thus we have $Z^2 = 0$.  

The proof for $W^2 = 0$ is completely analogous from the fact that $W^2.v = W^2.w = 0$.

For the proof of $ZW - Y = 0$, we note that since by Lemma \ref{lower-order-lemY} we have $ZW \in O_2(V) + F_4(\mathbf{1})$, but also $ZW \in I_1$, there exists $f_4(\mathbf{1}) \in F_4({\bf 1})$ of the form (\ref{general-f_5}), such that $ZW \equiv_2 f_4(\mathbf{1})$.  Thus acting on $v = \alpha(-1)^2v_\lambda \in M(1, \lambda)$ by the zero modes given by Lemma \ref{o-mode-lemma} and Corollary \ref{level-one-cor},  and using Eqns.\ (\ref{f_5.v}) and (\ref{f_5.w}), we have that for all $\lambda \in \mathbb{C}$ 
\[0  \ = \ Z.0  \  =  \  ZW.v \ = \ f_4(\mathbf{1}).v \ = \  8 c_1 v + (c_5 \lambda + c_6) w,\]
implying $c_1 = c_5 = c_ 6  = 0$.  Then acting on  $w = \alpha(-2) v_\lambda$, we have
\[
w = Z.v  \  =  \ ZW.w \ = \ f_4(\mathbf{1}).w \ = \  (c_2 (\lambda^2 + 4)  + c_3 \lambda + c_4  )w +  (c_7 \lambda + c_8)v, \]
implying $c_2 = c_3 = c_7 = c_8  = 0$ and $c_4 = 1$.  Thus we have $ZW = Y$.

Finally, since by Lemma \ref{lower-order-lemY} we have $(x^2 - y)^3 \in O_2(V) + F_4(\mathbf{1})$, there exists $f_4(\mathbf{1}) \in F_4(\mathbf{1})$ such that $(x^2 - y)^3 \equiv_2 f_4(\mathbf{1})$ and $(x^2 - y)^3 + f_4(\mathbf{1}) \in I_2 \subset I_1$.   But since $(x^2 - y)^3 \notin I_1$ this time $f_4(\mathbf{1}) \notin I_1$ and thus will not be of the form  (\ref{general-f_5}).   Rather since $(x^2 - y)^3 + f_4(\mathbf{1}) \in I_2 \subset I_1$, and we already have $(x^2 - y + 4)Y, Y^2 - Y  \in I_2$,  we have that there exist $c_1, \dots, c_8 \in \mathbb{C}$ such that 
\begin{eqnarray}\label{cubed-relation}
0 &\equiv_2& (x^2 - y)^3 + f_4(\mathbf{1}) \\
&\equiv_2& (x^2-y + c_1)(x^2 - y)(x^2 - y + 2) 
+ (c_2 y + c_3x + c_4 )Y  \nonumber \\
& & \quad + \,  (c_5x + c_6)Z  + ( c_7x + c_8)W . \nonumber
\end{eqnarray}

Acting on $v = \alpha(-1)^2v_\lambda \in M(1, \lambda)$ by both sides of Eqn.\ (\ref{cubed-relation}), we have that for all $\lambda \in \mathbb{C}$ 
\begin{eqnarray*}
0 &=& (\lambda^2 - (\lambda^2 + 4) + c_1)(\lambda^2 - (\lambda^2 + 4)) (\lambda^2 - (\lambda^2 + 4) + 2)v + ( c_5\lambda + c_6)w\\
&=& (-4 + c_1 )(-4) (-2)v + ( c_5\lambda + c_6)w
\end{eqnarray*}
implying $c_5 = c_ 6 = 0$ and $c_1 = 4$.  Then acting on  $w = \alpha(-2) v_\lambda$, we have
\begin{eqnarray*}
0 &=& (\lambda^2 - (\lambda^2 + 4) + 4)(\lambda^2 - (\lambda^2 + 4)) (\lambda^2 - (\lambda^2 + 4) + 2)w + (c_2(\lambda^2 + 4)  + c_3\lambda + c_4 )w \\
& & \quad + \, (c_7 \lambda + c_8)v  \\
&=& (c_2(\lambda^2 + 4)  + c_3\lambda + c_4 )w + (c_7 \lambda + c_8)v,  
\end{eqnarray*}
implying $c_2 = c_3 = c_4 = c_7 = c_8  = 0$.  Thus we have
\[0 \equiv_2 (x^2 - y)^3 + f_4(\mathbf{1})  \equiv_2 (x^2 - y)(x^2 - y + 2)(x^2 - y + 4).\]
\end{proof}

\section{The structure of $A_2(M_a(1))$}\label{structure-section}

In this section, we give the main theorem of the paper which states that the relations determined by Propositions \ref{first-relations-prop}, \ref{first-commutators-prop}, \ref{second-relations-prop}, \ref{last-commutators-prop},  and \ref{last-relations-prop} give all the relations in $R = \mathbb{C}[x,y]\langle Y, Z, W \rangle$ and thus generate $I_2$ such that $A_2(M_a(1)) \cong R/I_2$.  Then we give some simplified realizations of  the algebra $A_2(M_a(1))$.

\begin{thm}\label{A_2-theorem}
  Let   $\mathbb{C}[x, y]\langle Y, Z, W \rangle$ denote the algebra generated over $\mathbb{C}$ by the two commuting variables $x$ and $y$, and three non-commuting variables $Y, Z$, and $W$.  Let $I_2$ be the ideal generated by the polynomials
\begin{eqnarray}\label{final-relations-first} 
& (x^2-y)(x^2-y+2)(x^2-y+4), \ (x^2 - y + 4)Y, \ (x^2 - y + 4)Z, \ (x^2 - y + 4)W, &\\
  &Y^2 - Y,  \ Z^2, \ W^2,  \ ZY, \ YW, \  ZW - Y, & \label{final-relations-middle}\\
 &  YZ- ZY - Z, \ YW - WY + W, \ ZW-WZ +  \frac{1}{8}(x^2 - y)(x^2 - y + 2)  - 2 Y  .& \label{final-relations-last}
   \end{eqnarray}

Then we have the following isomorphism of algebras  
\begin{equation}\label{final-iso}
A_2(M_a(1)) \cong \mathbb{C}[x, y]\langle Y, Z, W \rangle/I_2 \\
\end{equation}
under the identification 
\begin{equation}\label{identification1}
\alpha(-1){\bf 1} + O_2(M_a(1)) \longleftrightarrow x + I_2, \quad \alpha(-1)^2{\bf 1} + O_2(M_a(1)) \longleftrightarrow y + I_2,
\end{equation}
\begin{equation}
\alpha(-1)\alpha(-4){\bf 1} + O_2(M_a(1)) \longleftrightarrow \tilde y + I_2, \quad \alpha(-1)^2\alpha(-4){\bf 1} + O_2(M_a(1)) \longleftrightarrow z + I_2,
\end{equation}
\begin{equation}
\alpha(-1)\alpha(-4)^2{\bf 1} + O_2(M_a(1)) \longleftrightarrow \tilde z + I_2. \label{identification3}
\end{equation}
and the change of variables 
\begin{equation}
Y = \frac{1}{12} (x^2 - 2y - \tilde y), \  Z =  \frac{1}{32} ( x^3 + 2x\tilde y + \tilde z), \ W = -\frac{1}{40}  \left(2z + \tilde z + 2xy - 2x \tilde y - 3x^3\right).
\end{equation}

Furthermore
\begin{eqnarray}
A_2(M_a(1)) & \cong& \mathbb{C}[x]\oplus \mathbb{C}[x]\oplus (\mathbb{C}[x]\otimes M_2(\mathbb{C}))  \label{second-characterization-level-two}  \\
& \cong & A_1(M_a(1))\oplus (\mathbb{C}[x]\otimes M_2(\mathbb{C})),
\label{third-characterization-level-two} 
\end{eqnarray} 
where $M_2(\mathbb{C})$ denotes the algebra of $2\times 2$ complex matrices. 
\end{thm}

\begin{proof} From Proposition \ref{use-level-one-prop}  and Corollary  \ref{level-one-cor}, we have that $A_2(V) \cong \mathbb{C}[x,y] \langle Y, Z, W \rangle/I_2$ for some ideal $I_2$.   From Propositions Propositions \ref{first-relations-prop}, \ref{first-commutators-prop}, \ref{second-relations-prop},  \ref{last-commutators-prop}, and \ref{last-relations-prop},  we have that $I_2$ contains the polynomials (\ref{final-relations-first})--(\ref{final-relations-last}).  Noting as well that since the zero mode actions of $x, y, Y, Z, W$ given in Lemma \ref{o-mode-lemma} (see also Corollary \ref{level-one-cor}) imply that $(x^2 - y) 
(x^2-y+ 2), (x^2 - y + 4), Y, Y-1, Z, W \notin I_2$, we have that $I_2$ is equal to the ideal generated by these polynomials (\ref{final-relations-first})--(\ref{final-relations-last}), giving (\ref{final-iso}).

To prove the isomorphisms (\ref{second-characterization-level-two}) and (\ref{third-characterization-level-two}), we note that the factors of the polynomial $p(x,y) = (x^2-y)(x^2-y+ 2)(x^2-y+ 4)$ are relatively prime and thus 
\[R/p(x,y) \cong R/(x^2-y) \oplus R/(x^2 - y + 2) \oplus R/(x^2 - y + 4).\]  
Then by the Correspondence Theorem (aka the Fourth Isomorphism Theorem) from ring theory, and the fact that the other polynomials in (\ref{final-relations-first}) are in $I_2$, we have that letting $I(Y,Z, W)$ be the ideal generated by the polynomial relations in (\ref{final-relations-middle}) and (\ref{final-relations-last}) 
\begin{eqnarray*}
A_2(M_a(1)) &\cong&  R/I_2 \\
&\cong& R/((x^2 - y), Y, Z, W)  \oplus R/((x^2 - y + 2), Y, Z, W)  \\
& & \quad \oplus R/((x^2 - y + 4),  I(Y, Z, W)) \\
&\cong & \mathbb{C}[x,y]/(x^2 - y) \oplus  \mathbb{C}[x,y]/(x^2 - y + 2)   \oplus  R/((x^2 - y + 4), I(Y,Z,W))\nonumber  \\
&\cong & A_1(M_a(1))  \oplus  R/((x^2 - y + 4), I(Y,Z,W) ) .\nonumber  
\end{eqnarray*}

It remains to show that $R/((x^2 - y + 4), I(Y,Z,W)) \cong \mathbb{C}[x] \otimes M_2( \mathbb{C})$, which follows from 
\begin{eqnarray*}
\lefteqn{R/((x^2 - y + 4), I(Y,Z,W)) }\\
&\cong& \mathbb{C}[x,y] \langle Y, Z, W \rangle /((x^2 - y + 4), Y^2 - Y,  Z^2, \ W^2,  ZY, YW,  ZW - Y,  \\
& & \quad YZ- ZY - Z, YW - WY + W,  ZW-WZ +  \frac{1}{8}(x^2 - y)(x^2 - y + 2)  - 2 Y )\\
 &\cong& \mathbb{C}[x] \langle Y, Z, W \rangle /( Y^2 - Y,  Z^2, \ W^2,  ZY, YW,  ZW - Y,   YZ- ZY - Z,  \\
& & \quad  \hspace{2.5in}  YW - WY + W,  ZW-WZ + 1  - 2 Y )\\
&\cong& \mathbb{C}[x] \otimes \mathbb{C}  \langle Y, Z, W \rangle /( Y^2 - Y,  Z^2, \ W^2,  ZY, YW,  ZW - Y,   YZ - Z,  \\
& & \quad \hspace{3.5in} WY - W,  WZ - 1  + Y )\\
&\cong& \mathbb{C}[x] \otimes M_2( \mathbb{C}),  
\end{eqnarray*}
where the last isomorphism follows from the fact that the surjective algebra homomorphism given by
\begin{eqnarray*}
\varphi :  \mathbb{C} \langle Y, Z, W \rangle &\longrightarrow & M_2(\mathbb{C}) \\
Y & \mapsto & \left[ \begin{array}{cc}
0 & 0\\
0 & 1
\end{array} \right]\\
Z & \mapsto &  \left[\begin{array}{cc}
0 & 0\\
1 & 0
\end{array} \right] \\
W & \mapsto &  \left[\begin{array}{cc}
0 & 1\\
0 & 0
\end{array} \right] 
\end{eqnarray*} 
has 
\[\mathrm{ker} \, \varphi = (Y^2 - Y, \ Z^2, \ W^2, ZY, \ YW, \ ZW - Y, \ YZ - Z, \ WY - W, \  WZ - 1 + Y).\]
\end{proof}

\section{Conjecture for $A_n(M_a(1))$ for $n>2$}

We have the following conjecture from \cite{AB-general-n} for the general structure of the higher level Zhu algebras $A_n(M_a((1))$, which has been proven for $n = 1$ in  \cite{BVY-Heisenberg} and for $n= 2$ in this paper.  

\begin{conj}
\[ A_n(M_a(1)) \cong A_{n-1}(M_a(1)) \oplus \left( \mathbb{C}[x] \otimes M_{p(n)}(\mathbb{C})\right) \]
where $p(n)$ denotes the number of  unordered partitions of $n$ into nonnegative integers, and $M_{p(n)}(\mathbb{C})$ denotes the algebra of $p(n) \times p(n)$ matrices.  
\end{conj}

\appendix

\section{Calculations  needed to determine the higher order terms of the relations satisfied by the generators of $A_2(M_a(1))$}

Here we collect several necessary computations that are too long to include in the body of the paper, and are only used for the proof of Lemma \ref{lower-order-lemY} giving the highest order terms in the relations for the generators $x, y, Y, Z, W$ for $A_2(M_a(1))$.   

We first prove that several expressions are in $O_2(M_a(1))$, mainly using Eqn.\ \eqref{reduce-equation-Heisenberg-level2} in Corollary \ref{level-two-corollary}.  These expressions given in Corollary \ref{circlecor} below are helpful for the subsequent lemmas  proved in this appendix and needed for the proof of Lemma \ref{lower-order-lemY}.  Here again we are using the notation as in (\ref{ABC-notation}):
\begin{equation*}
A=\alpha(-1)+\alpha(-2),\hspace{.1in} B=\alpha(-2)+\alpha(-3),\hspace{.1in} C=\alpha(-3)+\alpha(-4), \hspace{.1in}D=\alpha(-4)+\alpha(-5).
\end{equation*}

\begin{cor}\label{circlecor}
For any  $v\in M_a(1)$,
\begin{eqnarray}\label{coreq1}
B(B+C)(C+D)v &\sim_2 & 0,\\
\label{coreq2}
A(B+C)(C+D)v &\sim_2 & 0,\\
\label{coreq3}
C(B+C)^2v &\sim_2 & 0, \\
%\label{coreq4}
%D(B+C)^2v &\sim_2 & 0, {\color{red}\mbox{*********.}}\\
\label{coreq5}
C(A+B)(B+C)v &\sim_2 & 0,\\
\label{coreq6}
B(A + B)(C + D)v  &\sim_2 &  0,\\
\label{coreq7}
AB(C+D)v &\sim_2 & 0,\\
\label{coreq8}
AB(A+B)(B+C)v &\sim_2 & 0,\\
\label{coreq9} 
B^2(C+D)v &\sim_2 & 0 .
\end{eqnarray}
\end{cor}

\begin{proof} 
Let $v\in M_a(1)$. We first first prove Eqn.\  \eqref{coreq1}. By \eqref{O21}, $C(C+D)v\sim_20$, and thus 
\begin{eqnarray*}
B(B+C)(C+D)v\sim_2(B+C)^2(C+D)v.
\end{eqnarray*}
Applying \eqref{O23}, we have $(B+C)^2(C+D)v\sim_2-2B(C+D)^2\sim_20$, where the last equivalence comes from \eqref{O??}, and this proves (\ref{coreq1}).

We next prove \eqref{coreq2}. By \eqref{coreq1}, we have
\begin{eqnarray*}
A(B+C)(C+D)v\sim_2(A+B)(B+C)(C+D)v.
\end{eqnarray*}
Applying \eqref{O24}, we have
$(A+B)(B+C)(C+D)v\sim_2-A(C+D)^2v\sim_2 0,$ where this last equivalence comes from applying \eqref{O??}, and this proves (\ref{coreq2}).

By \eqref{O23}, we have
\begin{eqnarray*}
C(B+C)^2v\sim_2-2CB(C+D)v\sim_2 0,
\end{eqnarray*}
where the last equivalence comes from \eqref{O21}, which proves \eqref{coreq3}.

By \eqref{O24} and then \eqref{O21},  we  have 
\[C(A+B)(B+C)v\sim_2-CA(C+D)v\sim_20,\]
proving \eqref{coreq5}.

By \eqref{O21}, then \eqref{O24}, and finally \eqref{O??}, we have
\begin{eqnarray*}
B(A + B)(C + D)v \sim_2 (B+C)(A+B)(C+D)v \sim_2 -A(C+D)^2v \sim_2 0,
\end{eqnarray*}
proving \eqref{coreq6}.

By  \eqref{O21}, then \eqref{coreq2},  we have that 
\begin{equation*}AB(C+D)v\sim_2A(B+C)(C+D)v \sim_2 0 ,
\end{equation*}
proving  \eqref{coreq7}.

By \eqref{O24}, then \eqref{O21}, and finally \eqref{coreq2}, we have 
\begin{eqnarray*}
AB(A+B)(B+C)v \sim_2 - A^2(C+D)Bv \sim_2 -A^2(B+C)(C+D)v \sim_2   0 , 
\end{eqnarray*}
proving \eqref{coreq8}.

And finally,  by \eqref{O21}, then  \eqref{O23}, and then \eqref{O??}, we have 
\begin{equation*}B^2(C+D)v \sim_2 (B+C)^2(C+D)v \sim_2 - 2B (C+D)^2v \sim_2 0,
\end{equation*}
 proving \eqref{coreq9}.  
\end{proof}

\begin{lem}\label{appendixlem1}
Given $v\in F_{r}({\bf 1})   \subset M_a(1)$,  
\[(x^2- y)*_2v    \sim_2    -2 (3A(B+C)+ 3B(C+D)+ B(A+B))v+ f_{r}({\bf 1}), \] 
for some $f_{r}({\bf 1})\in F_{r}({\bf 1})$.
\end{lem}

\begin{proof}
Writing Eqn.\ (\ref{multiplication-one-one}) in terms of $C$ and $D$, using the associativity of $*_2$, and applying Eqns.\ (\ref{O21}) and (\ref{O??}), we have
\begin{eqnarray}\label{x^2-CD}
x^2*_2v &=& \big(10\alpha(-3) +15\alpha(-4) + 6\alpha(-5) \big)^2v  \\
%&=& \big( \alpha(-3) + 3C + 6 (C + D) \big)^2 v\nonumber \\
&=& \big(\alpha(-3)^2 + 6 \alpha(-3) C + 12 \alpha(-3) (C+D) + 9C^2 + 36 C(C+D) \nonumber \\
& & \quad  +\,  (C+D)^2 \big)v \nonumber \\
&\sim_2 &  \big(\alpha(-3)^2 + 6 \alpha(-3) C + 12 \alpha(-3) (C+D) + 9C^2 \big)v .\nonumber 
\end{eqnarray}

%{\color{violet} By Eqn.\ (\ref{multiplication-one-one}), then algebraic manipulation of the terms, and then applying  \eqref{O??}, and  \eqref{O22},}  we have that for any $v\in M_a(1)$, 
%\begin{eqnarray*}
%x^2*_2v &=& \big(10\alpha(-3) +15\alpha(-4) + 6\alpha(-5) \big)^2v\\
%&=& \Big(6 \big(\alpha(-3) + 2\alpha(-4) +  \alpha(-5) \big)+3 \big(\alpha(-3)+\alpha(-4) \big)+\alpha(-3) \Big)^2v\\
%&=& 36 \big(\alpha(-3)+2\alpha(-4)+\alpha(-5)\big)^2v+36\big(\alpha(-3)+\alpha(-4)\big)\big(\alpha(-3)+2\alpha(-4)\\
%& & \quad + \, \alpha(-5) \big)v  + 12\big(\alpha(-3)+2\alpha(-4)+\alpha(-5)\big)\alpha(-3)+9\big(\alpha(-3)+\alpha(-4)\big)^2v\\
%& & \quad  + \,  6  \big(\alpha(-3)+\alpha(-4)\big)\alpha(-3)v +\alpha(-3)^2v\\
%&\sim_2 & 12\big(\alpha(-3)+2\alpha(-4)+\alpha(-5)\big)\alpha(-3)+9\big(\alpha(-3)+\alpha(-4)\big)^2v\\
%& & \quad  + \,  6  \big(\alpha(-3)+\alpha(-4)\big)\alpha(-3)v+\alpha(-3)^2v\\
%&=& 28\alpha(-3)^2v+48\alpha(-3)\alpha(-4)v+12\alpha(-3)\alpha(-5)v+9\alpha(-4)^2v.
%\end{eqnarray*}}

Using Eqn.\ (\ref{multiplication-one-cor-formula-level2-expllicit}) giving the formula for left multiplication by  $y = \alpha(-1)^2 {\bf 1}$, we have that for any $v \in F_r({\bf 1})$,
\begin{eqnarray*}
y*_2v &=&  \sum_{j =0}^{4}  \sum_{\stackrel{p_1, p_2 \in \mathbb{N}}{p_1 + p_2  =  j} }   \left( \binom{4}{ 2 - j}  - 3  \binom{4}{3 - j} + 6  \binom{4 }{4 - j}   \right) \alpha(-p_1-1)  \alpha(-p_2-1)   \, v   \\
& & \quad + \,  f_{r}({\bf 1})  \\
%&=& \big( 6-12 + 6 \big) \alpha(-1)^2 v + \big(4- 18 + 24\big) 2 \alpha(-1) \alpha(-2)v + \big(1-12+36\big) \big(\alpha(-2)^2 \\
%& & \quad + \,  2 \alpha(-1) \alpha(-3) \big)v  + \big(-3 + 24\big) \big(2 \alpha(-1) \alpha(-4) + 2 \alpha (-2) \alpha(-3) \big)v \\
%& & \quad + \,  6 \big(\alpha(-3)^2 + 2 \alpha(-1) \alpha(-5) + 2 \alpha(-2) \alpha(-4) \big)v  +     f_{r}({\bf 1}) \\ 
&=& 20 \alpha(-1) \alpha(-2)v + 25 \big(\alpha(-2)^2  +  2 \alpha(-1) \alpha(-3) \big)v  + 21 \big(2 \alpha(-1) \alpha(-4) \\
& & \quad + \,  2 \alpha (-2) \alpha(-3) \big)v  +   6 \big(\alpha(-3)^2 + 2 \alpha(-1) \alpha(-5) + 2 \alpha(-2) \alpha(-4) \big)v  +  f_{r}({\bf 1}) \\
%&=&\Big( 20\alpha(-1)\alpha(-2) + 25\alpha(-2)^2 + 50\alpha(-1)\alpha(-3) + 42\alpha(-1)\alpha(-4) \\
%& & \quad + \,  42\alpha(-2)\alpha(-3) + 6\alpha(-3)^2 + 12\alpha(-1)\alpha(-5)+12\alpha(-2)\alpha(-4)\Big)v+  f_{r}({\bf 1}) \\
&=& \big( 20AB + 30 AC + 5B^2 + 12 AD -18BC 
 -12 \alpha(-2)D +18 \alpha(-3) C  \\
 & & \quad  + \,  \alpha(-3)^2 \big)v+   f_{r}({\bf 1}) 
\end{eqnarray*}
for some $f_{r}({\bf 1})\in F_{r}(\bf{1})$.
Combining the expressions for left multiplication by $x$ and $y$,  rearranging terms, using \eqref{O24}, rearranging, using \eqref{O23}, and rearranging again, we have
\begin{eqnarray*}
\lefteqn{(x^2-y)*_2v}\\
&  \sim_2 &   \big(\alpha(-3)^2 + 6 \alpha(-3) C + 12 \alpha(-3) (C+D) + 9C^2 -20AB - 30AC \\
& & \quad -\,  5B^2 - 12 AD  + 18 BC + 12 \alpha(-2) D - 18 \alpha(-3) C - \alpha(-3)^2  \big)v  - f_{r} ({\bf 1}) \\
%&=&  \big(  12BD + 9C^2 -20AB - 30AC - 5B^2 - 12 AD  + 18 BC   \big)v  -  f_{r} ({\bf 1}) \\
&=&(-8A(B+C)-12A(B+2C+D)+2AC-5B^2+18BC+12BD\\
& & \quad  + \, 9C^2)v  -  f_{r}({\bf 1}) \\
&  \sim_2 & (-8A(B+C)+12B(B+C)+2AC-5B^2+18BC+12BD+9C^2)v- f_{r}({\bf 1}) \\
&=&(-8A(B+C)+12B^2+12BC+2AC+9(B+C)^2-14B^2+12BD)v-   f_{r}({\bf 1}) \\
&\sim_2&(-8A(B+C)+12B(C+D)+2AC-18B(C+D)-2B^2)v-  f_{r}({\bf 1}) \\
&=&(-6A(B+C)-6B(C+D)-2B(A+B))v-  f_{r}({\bf 1}),
\end{eqnarray*}
giving the result, where $f_{r}({\bf 1})$ is relabeled as  $-f_{r}({\bf 1})$.  
\end{proof}

\begin{lem}\label{x^2tildey}
Given $v\in F_{r}(\bf {1})$, 
\[(x^2+\tilde y)*_2v \sim_2  (16A(C+D)+4C(A+B)-42B(C+D)-12C(B+C))v+f_{r}({\bf 1}),\] 
where $f_{r}({\bf 1})\in F_{r}({\bf 1})$.
\end{lem}

\begin{proof}
By the multiplication formula  given by Eqn.\ (\ref{multiplication-formula-infinity-level2}), with $t = 4$, $i_1 = i_4  = 1 $, $i_2=i_3 = 0$, $p_1 = p_2 = p_3 = 1$, $p_4 = 2$, and $r = 5$, we have for all $v\in V$,
\begin{eqnarray*}
\lefteqn{\tilde y *_2v =   \alpha(-4) \alpha(-1) {\bf 1} *_2 v} \\
&=& \sum_{m=0}^2 \sum_{j=-m}^2 (-1)^m\binom{m+2}{2}\binom{7}{j+m}\displaystyle\sum_{\substack{k_1,k_{2}\in \mathbb{N}\\
k_1+k_{2}=2-j}} \! \! {}_{\circ}^{\circ}P_{i_1,\dots, i_{4}}(k_1, k_{2}) {}_{\circ}^{\circ} v +g_{i_1,\dots, i_{4}}(v) ,
\end{eqnarray*}
where by definition (\ref{define-P}),  in this case $P_{i_1,\dots, i_{4}}(k_1, k_{2}) = \binom{k_2+3}{3} \alpha(-k_1-1)\alpha(-k_2-4)$, and  $g_{i_1,\dots, i_{4}}(v)$ is a linear combination of elements of the form  ${}_{\circ}^{\circ} \alpha(q_1) \alpha(q_{2}){}_{\circ}^{\circ}v$ for $q_1, q_{2} \in \mathbb{Z}$ with at least one of the  $q$'s nonnegative.   Note that this then means that when acting on $v \in F_r( {\bf 1})$, we will have $g_{i_1,\dots, i_{4}}(v) \in F_r({\bf 1})$.   Thus we have that for some $f_r({\bf 1}) \in F_r ({\bf 1})$, and organizing by partitions of $j$
\begin{eqnarray*}
\lefteqn{\tilde y *_2v }\\
&=& \sum_{m=0}^2 \sum_{j=-m}^2 (-1)^m\binom{m+2}{2}\binom{7}{j+m}\displaystyle\sum_{\substack{k_1,k_{2}\in \mathbb{N}\\
k_1+k_{2}=2-j}} \! \! {}_{\circ}^{\circ} \binom{k_2+3}{3} \alpha(-k_1-1)\alpha(-k_2-4) {}_{\circ}^{\circ} v +f_{r}({\bf 1}) ,\\
&=& \Big(210 \alpha(-1)\alpha(-8) + 120\alpha(-2)\alpha(-7) + 60 \alpha(-3)\alpha(-6) + 30\alpha(-4)\alpha(-5)\\
& & \quad + \, 780\alpha(-1)\alpha(-7)+390\alpha(-2)\alpha(-6) +156\alpha(-3)\alpha(-5) + 39\alpha(-4)^2 \\
& & \quad + \, 1060\alpha(-1)\alpha(-6)+424\alpha(-2)\alpha(-5)+106\alpha(-3)\alpha(-4) + 616\alpha(-1)\alpha(-5) \\
& & \quad + \, 154\alpha(-2)\alpha(-4)+126\alpha(-1)\alpha(-4)\Big)v+f_r(\bf{1}).
\end{eqnarray*}

Applying the recursion (\ref{recursion-level-2}) from Corollary \ref{level-two-corollary},  and simplifying, we have that  for all $v\in V$, 
\begin{eqnarray}\label{multiply-tildey}
\lefteqn{\tilde y*_2v} \\
&\sim_2 & \Big(-210 \alpha(-1) \big(6\alpha(-3) + 15 \alpha(-4) + 10 \alpha(-5) \big) + 120\alpha(-2) \big( 3 \alpha(-3) \nonumber \\
& & \quad + \,  8 \alpha(-4) +  6 \alpha(-5) \big) - 60 \alpha(-3)\big( \alpha(-3) + 3 \alpha(-4) + 3 \alpha(-5)  \big)  \nonumber \\
& & \quad + \, 30\alpha(-4)\alpha(-5) + 780\alpha(-1) \big(3 \alpha(-3) + 8 \alpha(-4) + 6 \alpha(-5)   \big) \nonumber \\
& & \quad - \,  390\alpha(-2)\big( \alpha(-3) + 3 \alpha(-4) +  3 \alpha(-5)  \big)   +156\alpha(-3)\alpha(-5) \nonumber \\
& & \quad  +  \,39\alpha(-4)^2  -  1060\alpha(-1)\big( \alpha(-3) + 3 \alpha(-4)  +  3 \alpha(-5)    \big) \nonumber \\
& & \quad + \, 424\alpha(-2)\alpha(-5)  + 106\alpha(-3)\alpha(-4) + 616\alpha(-1)\alpha(-5) \nonumber
 \\
& & \quad   + \,  154\alpha(-2)\alpha(-4) + 126\alpha(-1)\alpha(-4)\Big)v+f_r({\bf 1}) \nonumber \\
%&=& \Big( 20\alpha(-1)\alpha(-3) + 36\alpha(-1)\alpha(-4) + 16\alpha(-1)\alpha(-5) - 30\alpha(-2)\alpha(-3)  \nonumber \\
%& & \quad - \,  56\alpha(-2)\alpha(-4) -  26 \alpha(-2)\alpha(-5) - 60\alpha(-3)^2- 74\alpha(-3)\alpha(-4) \nonumber  \\
%& & \quad - \,  24\alpha(-3)\alpha(-5) + 39\alpha(-4)^2 + 30\alpha(-4)\alpha(-5)\Big)v + f_{r}({\bf 1}) \nonumber \\
&=&  \Big( 20AC - 50 B C + 16AD  - 42 BD  - \alpha(-3)^2  + 9C^2  - 18 \alpha(-3) C   \nonumber \\
& & \quad  + \,   30CD  -  12 \alpha(-3) D \Big)v + f_{r}({\bf 1})  , \nonumber
\end{eqnarray}
for some $f_{r}({\bf 1})\in F_r({\bf 1}).$ 
Then using the formula for multiplication by $x^2$  given by Eqn.\ (\ref{x^2-CD}),  simplifying, rearranging, using Eqn.\ (\ref{O21}), and rearranging again, we have for all $v\in V$,
\begin{eqnarray*}
\lefteqn{(x^2+\tilde y) *_2 v } \\
&\sim_2 & \big(\alpha(-3)^2 + 6 \alpha(-3) C + 12 \alpha(-3) (C+D) + 9C^2 + 20AC - 50 B C + 16AD  - 42 BD  \\
& & \quad   - \,  \alpha(-3)^2  + 9C^2  - 18 \alpha(-3) C   +  30CD  -    12 \alpha(-3) D \Big)v + f_{r}(\bf{1}) \\
%&=&  {\color{purple} \big(  18C^2 + 20AC - 50 B C + 16AD  - 42 BD     +  30CD   \Big)v + f_{r}(\bf{1}) }\\
%&=& {\color{purple} REWRITING HERE!!} \\
%& {\color{purple} \sim_2} &\Big(28\alpha(-3)^2+48\alpha(-3)\alpha(-4)+12\alpha(-3)\alpha(-5)+9\alpha(-4)^2 {\color{violet} + } 20\alpha(-1)\alpha(-3)\\
%& & \quad + \,36\alpha(-1)\alpha(-4) + 16\alpha(-1)\alpha(-5) - 30\alpha(-2)\alpha(-3) - 56\alpha(-2)\alpha(-4)\\
%& &\quad - \, 26 \alpha(-2)\alpha(-5) - 60\alpha(-3)^2 - 74\alpha(-3)\alpha(-4) - 24\alpha(-3)\alpha(-5) + 39\alpha(-4)^2\\
%& &\quad + \, 30\alpha(-4)\alpha(-5)\Big)v+f_{r}(\bf{1})\\
%&=&\big(-32\alpha(-3)^2 - 26\alpha(-3)\alpha(-4) - 12\alpha(-3)\alpha(-5) + 48\alpha(-4)^2 + 30\alpha(-4)\alpha(-5)\\
%& & \quad + \, 20\alpha(-1)\alpha(-3) + 36\alpha(-1)\alpha(-4) + 16\alpha(-1)\alpha(-5) - 30\alpha(-2)\alpha(-3)\\
%& & \quad - \, 56\alpha(-2)\alpha(-4) - 26\alpha(-2)\alpha(-5)\big)v + f_{r}(\bf{1})\\
%&=&(20AC+16AD-50BC-42BD+18C^2+30CD)v+f_{r}(\bf{1})\\
&=&(16A(C+D) + 4AC -42B(C+D)-8BC+30C(C+D)-12C^2)v+f_{r}(\bf{1})\\
&\sim_2&(16A(C+D) + 4AC -42B(C+D)-8BC -12C^2)v+f_{r}(\bf{1})\\
&=&(16A(C+D)+4C(A+B)-42B(C+D)-12C(B+C))v+f_{r}(\bf{1}),
\end{eqnarray*}
giving the result. 
\end{proof}

\begin{lem}\label{C^3}
Given $v\in F_{r}(\bf {1})$, we have $C^3v\sim_2 f_{r+1}({\bf 1})$, where $f_{r+1}({\bf 1})\in F_{r+1}({\bf 1})$.  
\end{lem}

\begin{proof}
By the definition of the circle product Eqn.\ (\ref{define-circ}) and Lemma \ref{Y+-lemma},
\begin{eqnarray*}
\lefteqn{\alpha(-3)^3{\bf 1}\circ_2 v \ = \   \res_x \frac{(1+x)^{11}}{x^6} Y(\alpha(-3)^3 {\bf 1}, x)v }\\
&=& \res_x \frac{(1+x)^{2}}{x^6} \nord (  (1+x)^3 Y^-(\alpha(-3){\bf 1} , x) + (1+ x)^3 Y^+(\alpha(-3) {\bf 1} , x) )^3 \nord v\\
&\sim_2& \res_x \frac{(1+x)^{2}}{x^6} \nord \Big(  (1+x)^3 Y^-(\alpha(-3) {\bf 1} , x) + \alpha(-3)+\frac{3Cx}{1+x}+\frac{6(C+D)x^2}{(1+x)^2} \Big)^3 \nord v \\
&=& \res_x \frac{1}{x^6} \Big( p_1(x) Y^-(\alpha(-3) {\bf 1}, x)  + p_2(x) Y^-(\alpha(-3) {\bf 1}, x)^2  +  p_3(x) Y^-(\alpha(-3) {\bf 1}, x)^3  \\
& & \quad + \, (1+x)^2 \Big(\alpha(-3)+\frac{3Cx}{1+x} +  \frac{6(C+D)x^2}{(1+x)^2}\Big)^3 \Big)v,\\
\end{eqnarray*}
where $p_1, p_2, p_3$ are polynomials in $x$ of degree 5, 8, and 11, respectively, with the coefficients of $p_1$ involving $\alpha(-m_1)\alpha(-m_2)$, coefficients of $p_2$ involving $\alpha(-m)$, and coefficients of $p_3$ constants.   Therefore 
\[
\alpha(-3)^3{\bf 1}\circ_2 v
\sim_2 \res_x \frac{(1+x)^2}{x^6}  \Big(\alpha(-3)+\frac{3Cx}{1+x} +  \frac{6(C+D)x^2}{(1+x)^2}\Big)^3 v + f_{r+1}({\bf 1}), \]
for some $f_{r+1} ({\bf 1}) \in F_{r+1} ( {\bf 1})$,  where the $f_{r+1}$ term comes from the terms involving as many as two negative modes (from $p_1$) and one positive mode from the singular parts of $Y$ acting on $v \in F_r({\bf 1})$, subtracting a negative mode and adding up to 2 negative modes.  Then by Eqns.\ (\ref{O21}) and (\ref{O??}), we have
\[ \alpha(-3)^3{\bf 1}\circ_2 v \sim_2 27C^3v +  f_{r+1}({\bf 1}).\]
Therefore for any $v\in F_r({\bf 1})$, we have $C^3v \sim_2  f_{r+1}({\bf 1})$ for some $f_{r+1}({\bf 1})\in F_{r+1}({\bf1})$.
\end{proof}

As with Lemma \ref{x^2tildey}, and now using Lemma \ref{C^3}, we proceed to prove the additional relations up to terms in $F_k({\bf 1})$ for some $k \in \mathbb{N}$, for expressions involving $x, y, \tilde y, z, \tilde z$ multiplying a $v \in F_r({\bf 1})$ in $A_2(M_a(1))$ as follows:

\begin{lem}\label{zxtildey}
Given $v \in F_{r}({\bf 1})$, 
\begin{multline*}
(\tilde z+x \tilde y)*_2v    \sim_2    \big(-4\alpha(-3)C \big((A+B)-3(B+C) \big) - 4C^2(3A-10B) \\
 \quad + \, \alpha(-3)(C+D)(-16A+42B) \big)v+  f_{r+1} ({\bf 1}),
\end{multline*}
where $f_{r+1}({\bf 1})\in F_{r+1}({\bf 1})$. 
\end{lem}

\begin{proof} 
From Eqn.\ \eqref{multiplication-one-one} and Lemma \ref{Y+-lemma}, in terms of the $A, B, C, D$ notation,  we have
\begin{eqnarray*} x *_2 v & = & \big(10\alpha(-3) + 15 \alpha(-4) + 6\alpha(-5) \big)v  = \big(\alpha(-3) + 3C + 6(C+ D) \big) v ,\\
Y^+(\alpha(-1){\mathbf{1}},x)v &\sim_2& \Big( \alpha(-1)+\alpha(-2)x+\alpha(-3)\frac{x^2}{1+x}+C\frac{x^3}{(1+x)^2} +  (C+D) \frac{x^4}{(1+x)^3} \Big)v,\\
Y^+(\alpha(-4){\mathbf{1}},x)v &\sim_2& \Big(-\alpha(-3) \frac{1}{(1+x)^4}+C  \frac{(1-3x)}{(1+x)^5} +  (C+D) \frac{4x-6x^2}{(1+x)^6}\Big) v.
\end{eqnarray*}																								
Thus by associativity of $*_2$, the definition of $*_2$, collecting terms involving some positive modes acting on $v$ into $F_{r+1}({\bf 1})$, and simplifying, we have that for all $v \in F_r({\bf 1})$
\begin{eqnarray*}
\lefteqn{(\tilde z + x\tilde y) v = (\tilde z+x*_2 \tilde y)*_2v = \tilde z*_2v+ x *_2  ( \tilde y*_2v )} \\
&\sim_2& \res_x\Big(\frac{1}{x^3}-\frac{3}{x^4}+\frac{6}{x^5}\Big)(1+x)^{11}Y(\alpha(-1)\alpha(-4)^2{\bf 1}, x)v\\
& & \quad +\, \big(\alpha(-3) + 3C + 6(C+ D) \big) \res_x\Big(\frac{1}{x^3}-\frac{3}{x^4}+\frac{6}{x^5}\Big)(1+x)^7Y(\alpha(-1)\alpha(-4){\bf 1}, x)v\\
%&=&   \res_x\Big(\frac{1}{x^3}-\frac{3}{x^4}+\frac{6}{x^5}\Big)(1+x)^{11}\nord \Big(  \big(Y^-(\alpha(-1) {\bf 1}, x) + Y^+(\alpha(-1){\bf 1}, x) \big) \big( Y^-(\alpha(-4){\bf 1}, x) \\
%& & \quad + \,  Y^+(\alpha(-4){\bf 1}, x \big)^2  +  \big(\alpha(-3) + 3C + 6(C+ D) \big) (1+x)^7   \big(Y^-(\alpha(-1) {\bf 1}, x) \\
%& & \quad + \,  Y^+(\alpha(-1){\bf 1}, x) \big) \big( Y^-(\alpha(-4){\bf 1}, x) +   Y^+(\alpha(-4){\bf 1}, x \big)   \Big) \nord v \\
&=& \res_x\Big(\frac{1}{x^3}-\frac{3}{x^4}+\frac{6}{x^5}\Big) \Big( (1+x)^{11} Y^+(\alpha(-1) {\bf 1}, x) \big( Y^+(\alpha(-4) {\bf 1} , x ) \big)^2 \\
& & \quad + \,  (1+ x)^7  \big(\alpha(-3) + 3C + 6(C+ D) \big) Y^+(\alpha(-1) {\bf 1}, x ) Y^+ (\alpha(-4) {\bf 1} , x) \Big) v +   f_{r+1}( {\bf 1})  \\
%&=& \res_x\Big(\frac{1}{x^3}-\frac{3}{x^4}+\frac{6}{x^5}\Big) (1+x)^{7} Y^+(\alpha(-1) {\bf 1}, x)  Y^+(\alpha(-4) {\bf 1} , x )  \Big(    (1+x)^4 Y^+
%( \alpha(-4) {\bf 1} , x )  \\
%& & \quad + \,  \big(\alpha(-3) + 3C + 6(C+ D) \big) \Big)  v + f_{r+1}( {\bf 1}) \\
%&=& \res_x\Big(\frac{1}{x^3}-\frac{3}{x^4}+\frac{6}{x^5}\Big) (1+x)^{7}  \Big( \alpha(-1)+\alpha(-2)x+\alpha(-3)\frac{x^2}{1+x}+C\frac{x^3}{(1+x)^2} \\
%& & \quad + \,   (C+D) \frac{x^4}{(1+x)^3} \Big)   \Big(-\alpha(-3) \frac{1}{(1+x)^4}+C  \frac{1-3x}{(1+x)^5} +  (C+D) \frac{4x-6x^2}{(1+x)^6}\Big) \\
%& & \quad \cdot   \Big(    (1+x)^4  \Big(-\alpha(-3) \frac{1}{(1+x)^4}+C  \frac{1-3x}{(1+x)^5} +  (C+D) \frac{4x-6x^2}{(1+x)^6}\Big)   \\
%& & \quad + \,   \big(\alpha(-3) + 3C + 6(C+ D) \big)  \Big)  v + f_{r+1}( {\bf 1}) \\
%&=& \res_x\Big(\frac{1}{x^3}-\frac{3}{x^4}+\frac{6}{x^5}\Big) (1+x)^{7}  \Big( \alpha(-1)+\alpha(-2)x+\alpha(-3)\frac{x^2}{1+x}+C\frac{x^3}{(1+x)^2} \\
%& & \quad + \,   (C+D) \frac{x^4}{(1+x)^3} \Big)   \Big(-\alpha(-3) \frac{1}{(1+x)^4}+C  \frac{1-3x}{(1+x)^5} +  (C+D) \frac{4x-6x^2}{(1+x)^6}\Big) \\
%& & \quad \cdot   \Big(   C  \frac{1-3x}{1+x} +  (C+D) \frac{4x-6x^2}{(1+x)^2}    +   3C + 6(C+ D) \Big)  v + f_{r+1}( {\bf 1}) \\
&= &\res_x\Big(\frac{1}{x^3}-\frac{3}{x^4}+\frac{6}{x^5}\Big) (1+x)^{2}  \Big( \alpha(-1)+\alpha(-2)x+\alpha(-3)\frac{x^2}{1+x}+C\frac{x^3}{(1+x)^2} \\
& & \quad + \,   (C+D) \frac{x^4}{(1+x)^3} \Big)   \Big(-\alpha(-3) +C  \frac{1-3x}{1+x} +  (C+D) \frac{4x-6x^2}{(1+x)^2}\Big) \\
& & \quad \cdot   \Big(   C(1-3x)   +  (C+D) \frac{4x- 6x^2}{1+x}  +3C(1+x) + 6(C+D) (1+x)   \Big)  v + f_{r+1}( {\bf 1}).
\end{eqnarray*}
Then using Eqns.\ (\ref{O21}), (\ref{O??}), and Lemma \ref{C^3}, and then simplifying, we have 
\begin{eqnarray*}
\lefteqn{(\tilde z + x\tilde y) v }\\
&\sim_2 &\res_x\Big(\frac{1}{x^3}-\frac{3}{x^4}+\frac{6}{x^5}\Big) (1+x)^{2}  \Big( \alpha(-1)+\alpha(-2)x+\alpha(-3)\frac{x^2}{1+x}+C\frac{x^3}{(1+x)^2}  \Big)  \\
& & \quad \cdot  \Big(-\alpha(-3) +C  \frac{1-3x}{1+x}\Big)    \Big(   4C  +  (C+D) \frac{6+16x}{1+x}     \Big)  v + f_{r+1}( {\bf 1}) \\
&\sim_2& \res_x\Big(\frac{1}{x^3}-\frac{3}{x^4}+\frac{6}{x^5}\Big) (1+x)^{2} 
\Big( -4 \alpha(-1) \alpha(-3) C - \alpha(-1) \alpha(-3) (C+ D)  \frac{6+16x}{1+x}  \\
& & \quad +\, 4 \alpha(-1) C^2 \frac{1-3x}{1+x} - 4\alpha(-2) \alpha(-3) Cx  - \alpha(-2) \alpha(-3) (C+D)  \frac{x(6+16x)}{1+x}   \\
& & \quad  +\,  4 \alpha(-2) C^2  \frac{x(1-3x)}{1+x} - 4\alpha(-3)^2C \frac{x^2}{1+x} - \alpha(-3)^2 (C+D) \frac{x^2(6+16x)}{(1+x)^2} \\
& & \quad + \, 4\alpha(-3) C^2 \frac{x^2(1-3x)}{(1+x)^2} - 4\alpha(-3) C^2 \frac{x^3} {(1+x)^2} \Big) v +  g_{r+1}( {\bf 1}) \\
%&=& \Big( -4 \alpha(-1) \alpha(-3) C - 16 \alpha(-1) \alpha(-3) (C+D) - 12 \alpha(-1) C^2 + 4 \alpha(-2) \alpha(-3) C  \\
%& & \quad   + \,  26 \alpha(-2) \alpha(-3) (C+ D) + 28 \alpha(-2) C^2 + 8\alpha(-3)^2C + 42 \alpha(-3)^2(C+D) \\
%& & \quad  + \,  40 \alpha(-3) C^2 + 12 \alpha(-3) C^2  \Big) v +  g_{r+1}( {\bf 1})\\
% &=& \Big( -4 \alpha(-3) A C - 16  \alpha(-3) A(C+D) - 12 A C^2 + 8 \alpha(-3) BC    + 40 B C^2 \\
% & & \quad + \,  42 \alpha(-3)B(C+D)  + 12 \alpha(-3) C^2  \Big) v +  g_{r+1}( {\bf 1}) \\
&=& \Big( -4 \alpha(-3) (A+B)C  + \alpha(-3) (-16 A + 42 B) (C+D)  +  12 \alpha(-3) (B+C)C  \\
& & \quad    - \, 4( 3A -10B)  C^2     \Big) v +  g_{r+1}( {\bf 1}) 
\end{eqnarray*}
for some $f_{r+1}({\bf 1}), g_{r+1}({\bf 1})  \in F_{r+1}({\bf 1})$, which gives the result. 
\end{proof}

\begin{lem}\label{ztildexcube}
Given $v\in F_{r}({\bf 1})$, 
\begin{multline*}
(\tilde z-x^3)*_2v    \equiv_2    C^2 \big(-24(A+B)+88(B+C) \big) v +\alpha(-3) \big(84B(C+D)\\
+ 24C(B+C) -8C(A+B)-32A(C+D) \big)v + f_{r+1}({\bf 1}) ,
\end{multline*}
where $f_{r+1}({\bf 1})\in F_{r+1}({\bf 1})$.  
\end{lem}

\begin{proof}
By Lemma \ref{powersalpha1-lemma}, we have that $x^3  \equiv_2  x *_2  \alpha(-4)^2{\bf 1}$.  Using this, associativity of $*_2$,  Eqn.\   \eqref{multiplication-one-one}, Lemma \ref{Y+-lemma}, collecting terms involving some positive modes in the singular part of $Y$ acting on $v$ into $F_{r+1}({\bf 1})$, and using Eqns. (\ref{O21}) and (\ref{O??}), and Lemma \ref{C^3}, we have
\begin{eqnarray*}
\lefteqn{ (\tilde z-x^3)*_2v  \equiv_2  \alpha(-1) \alpha(-4)^2 {\bf 1}  *_2 v -  x *_2 (\alpha(-4)^2 {\bf 1} *_2 v)    } \\
&=&  \res_x\Big(\frac{1}{x^3}-\frac{3}{x^4}+\frac{6}{x^5}\Big)(1+x)^{11}Y(\alpha(-1)\alpha(-4)^2{\bf 1}, x)v\\
& & \quad - \, \big(\alpha(-3) + 3C + 6(C+ D) \big) \res_x\Big(\frac{1}{x^3}-\frac{3}{x^4}+\frac{6}{x^5}\Big)(1+x)^{10}Y(\alpha(-4)^2{\bf 1}, x)v\\
%&=& \res_x \Big(\frac{1}{x^3}-\frac{3}{x^4}+\frac{6}{x^5}\Big) (1+x)^{10} \Big( (1+x) Y^+(\alpha(-1) {\bf 1}, x)  Y^+(\alpha(-4){\bf 1}, x)^2 \\
%& & \quad - \, \big(\alpha(-3) + 3C + 6(C+ D) \big) Y^+(\alpha(-4){\bf 1}, x)^2 \Big) v + f_{r+1}({\bf 1})  \\
&=& \res_x \Big(\frac{1}{x^3}-\frac{3}{x^4}+\frac{6}{x^5}\Big) \Big( (1+x) Y^+(\alpha(-1) {\bf 1}, x)  - \big(\alpha(-3) + 3C + 6(C+ D) \big) \Big)\\
& & \quad \cdot  \Big( (1+x)^5Y^+(\alpha(-4){\bf 1}, x) \Big)^2  v + f_{r+1}({\bf 1}) \\
&\sim_2 & \res_x \Big(\frac{1}{x^3}-\frac{3}{x^4}+\frac{6}{x^5}\Big)  \Big( \alpha(-1) (1+x)  + \alpha(-2) x(1+x) + \alpha(-3) x^2 + C\frac{x^3}{1+x} \\
& & \quad + \, (C+D) \frac{x^4}{(1+x)^2}  - \alpha(-3) - 3C - 6(C+ D)  \Big) \Big( - \alpha(-3) (1+x) \\
& & \quad  + \,  C(1-3x) + (C+D) \frac{4x-6x^2}{1+x} \Big)^2  v + f_{r+1}({\bf 1}) \\
&\sim_2 & \res_x \Big(\frac{1}{x^3}-\frac{3}{x^4}+\frac{6}{x^5}\Big)  \Big( \alpha(-1) (1+x)  + \alpha(-2) x(1+x) + \alpha(-3) (x^2 - 1)  \\
& & \quad  +\,  C\frac{x^3}{1+x} + (C+D) \frac{x^4}{(1+x)^2}  - 3C - 6(C+ D)  \Big) \Big(  \alpha(-3)^2  (1+x)^2 \\
& & \quad - \,  2\alpha(-3) C(1+x)(1-3x) - 2\alpha(-3) (C+D) (4x - 6x^2) + C^2(1-3x)^2  \Big)  v \\
& & \quad + \,  f_{r+1}({\bf 1}) \\
&\sim_2 & \res_x \Big(\frac{1}{x^3}-\frac{3}{x^4}+\frac{6}{x^5}\Big)  \Big( \alpha(-1)\alpha(-3)^2 (1+x)^3 - 2 \alpha(-1) \alpha(-3)C (1+x)^2 (1-3x) \\
& & \quad - \,  2 \alpha(-1) \alpha(-3) (C+D) (4x-6x^2)(1+x) +   \alpha(-1) C^2 (1-3x)^2 (1+x) \\
& & \quad + \, \alpha(-2)\alpha(-3)^2 x(1+x)^3 - 2 \alpha(-2) \alpha(-3)Cx (1+x)^2 (1-3x) \\
& & \quad - \,  2 \alpha(-2) \alpha(-3) (C+D)x (4x-6x^2)(1+x) +   \alpha(-2) C^2 x(1-3x)^2 (1+x) \\
& & \quad + \, \alpha(-3)^3 (x^2 - 1) (1+ x)^2 - 2 \alpha(-3)^2 C (x^2 - 1) (1+x) (1-3x)  \\
& & \quad - \, 2\alpha(-3)^2 (C+D) (x^2 - 1) (4x - 6x^2) + \alpha(-3) C^2 (x^2 - 1) (1- 3x)^2 
 \\																		
 & & \quad + \, \alpha(-3)^2 C x^3 (1+x) - 2 \alpha(-3) C^2 x^3(1-3x) + \alpha(-3)^2 (C+D) x^4  \\
 & & \quad - \, 3\alpha(-3)^2 C (1+x)^2 + 6 \alpha(-3) C^2 (1+x)(1-3x) - 6 \alpha(-3)^2(C+D)  (1+x)^2  \Big)  v \\
& & \quad + \, g_{r+1}({\bf 1})  \\
&=& \Big(  - 8 \alpha(-1) \alpha(-3)C - 32 \alpha(-1) \alpha(-3) (C+D) -24   \alpha(-1) C^2  + 8 \alpha(-2) \alpha(-3)C \\
& & \quad + \, 52  \alpha(-2) \alpha(-3) (C+D) + 40  \alpha(-2) C^2  + 16  \alpha(-3)^2 C  + 84 \alpha(-3)^2 (C+D)  \\
& & \quad + \,  64 \alpha(-3) C^2  + 3 \alpha(-3)^2 C  + 42 \alpha(-3) C^2 +  6 \alpha(-3)^2 (C+D)   -  3\alpha(-3)^2 C  \\
& & \quad - \, 18 \alpha(-3) C^2 - 6 \alpha(-3)^2(C+D)  \Big)  v +  g_{r+1}({\bf 1}) \\
%&=& \Big(  - 8 \alpha(-1) \alpha(-3)C - 32 \alpha(-1) \alpha(-3) (C+D) -24   \alpha(-1) C^2 + 8 \alpha(-2) \alpha(-3)C \\
%& & \quad + \,  52  \alpha(-2) \alpha(-3) (C+D) + 40  \alpha(-2) C^2 + 16  \alpha(-3)^2 C  + 84\alpha(-3)^2 (C+D)  \\
%& & \quad + \,  88 \alpha(-3) C^2     \Big)  v +  g_{r+1}({\bf 1}) \\
%&=& \Big(  - 8  \alpha(-3)AC  - 32 \ \alpha(-3) A (C+D)  -24   A C^2 + 16  \alpha(-3)BC \\
%& & \quad + \,    84  \alpha(-3) B (C+D) + 64 B C^2   + 24 \alpha(-3) C^2     \Big)  v +  g_{r+1}({\bf 1}) \\
%&=& \alpha(-3) \big( -8AC - 32A(C+D) + 16BC + 84B(C+D) + 24 C^2 \big)v \\
%& & \quad + \, C^2 \big(-24A + 64B\big)v + g_{r+1} ({\bf 1}) \\
&=&  \alpha(-3) \big( -8C(A+B)  - 32A(C+D) + 24C(B+C)  + 84B(C+D) \big)v \\
& & \quad + \, C^2 \big(-24(A+B) + 88(B+ C) \big)v + g_{r+1} ({\bf 1}) \\
\end{eqnarray*}
for some $f_{r+1}({\bf 1}), g_{r+1} ({\bf 1}) \in F_{r+1}({\bf 1})$, which gives the result. 
\end{proof}

\begin{lem}\label{zxy} 
 For $v \in F_r({\bf 1})$
\begin{eqnarray*}
(z+x y)*_2v
&\sim_2 & \alpha(-1) \big(8(A+B)C +32A(C+D) -24(B+C)C -84B(C+D)\big)v\\
& &\quad + \, \big(-8A^2C-42A^2(C+D) +24(A+B)BC +48AC^2\big)v +  f_{r+1}({\bf 1}),
\end{eqnarray*} 
for some  $f_{r+1}({\bf{1}})\in F_{r+1}({\bf 1}).$  
\end{lem}

\begin{proof}
By commutativity of $y = 2 \omega$, associativity of $*_2$, the definition of $*_2$,  Lemma \ref{Y+-lemma}, collecting terms involving some positive modes in the singular part of $Y$ acting on $v$ into $F_{r+1}({\bf 1})$, and using Eqn.\ \eqref{multiplication-one-one}, Eqn.\ (\ref{O21}), Eqn.\ (\ref{O??}), Lemma \ref{C^3}, Eqn.\ (\ref{coreq9}), and Eqn.\ (\ref{coreq7}), we have that for $v \in F_r({\bf 1})$, 
\begin{eqnarray*}
\lefteqn{(z+xy) *_2 v }\\
&=& (z + x *_2 y) *_2 v = z *_2 v + y *_2 (x *_2 v ) = \alpha(-1)^2 \alpha(-4) {\bf 1} *_2 v + \alpha(-1)^2 {\bf 1}  *_2 ( x  *_2 v)  \\
&=& \res_x \Big(\frac{1}{x^3}-\frac{3}{x^4}+\frac{6}{x^5}\Big)(1+x)^{8} Y^+(\alpha(-1)^2 {\bf 1} , x)Y^+(\alpha(-4) {\bf 1} , x)v  \\
& & \quad + \,    \res_x \Big(\frac{1}{x^3}-\frac{3}{x^4}+\frac{6}{x^5}\Big)(1+x)^{4} Y^+(\alpha(-1)^2 {\bf 1} , x)  \big( \alpha(-3) + 3C  +  6(C+D) \big) v \\
& & \quad + \, f_{r+1}({\bf 1})  \\
%&=& \res_x \Big(\frac{1}{x^3}-\frac{3}{x^4}+\frac{6}{x^5}\Big)(1+x)^4 Y^+(\alpha(-1)^2 {\bf 1} , x)\Big( (1+x)^4 Y^+(\alpha(-4) {\bf 1} , x) +  \big( \alpha(-3) \\
%& & \quad  + \,  3C  +  6(C+D) \big) \Big) v +  f_{r+1}({\bf 1})\\
%&\sim_2& \res_x \Big(\frac{1}{x^3}-\frac{3}{x^4}+\frac{6}{x^5}\Big)(1+x)^4 \Big( \alpha(-1) + \alpha(-2)x + \alpha(-3) \frac{x^2}{1+x} + C \frac{x^3} {(1+x)^2} \\
%& & \quad + \,  ( C + D) \frac{x^4} {(1+x)^3} \Big)^2 \Big( -\alpha(-3) + C \frac{(1-3x)}{(1+x)} 
%+ (C+D) \frac{(4x - 6x^2)}{(1+x)^2}  \\
%& & \quad  + \,   \big( \alpha(-3)  + 3C  +  6(C+D) \big) \Big) v +  f_{r+1}({\bf 1})\\
&=& \res_x \Big(\frac{1}{x^3}-\frac{3}{x^4}+\frac{6}{x^5}\Big) \Big( \alpha(-1)(1+x)^2  + \alpha(-2)x(1+
x)^2  + \alpha(-3)x^2(1+x)  + C x^3 \\
& & \quad + \,  ( C + D) \frac{x^4} {1+x} \Big)^2 \Big(  C \frac{4}{(1+x)} 
+ (C+D) \frac{(6 + 16x)}{(1+x)^2} \Big) v +  f_{r+1}({\bf 1})\\
&\sim_2& \res_x \Big(\frac{1}{x^3}-\frac{3}{x^4}+\frac{6}{x^5}\Big) \Big( \alpha(-1)(1+x)^2  + \alpha(-2)x(1+
x)^2  + \alpha(-3)x^2(1+x)  \\
& & \quad  +\,  C x^3 \Big)^2 \Big(  C \frac{4}{(1+x)}  + (C+D) \frac{(6 + 16x)}{(1+x)^2} \Big) v +  f_{r+1}({\bf 1})\\
&=&  \res_x \Big(\frac{1}{x^3}-\frac{3}{x^4}+\frac{6}{x^5}\Big) \Big( \alpha(-1)(1+x)  + Ax(1+x) + B x^2(1+x)  \\
& & \quad + \,  C x^3 \Big)^2 \Big(  C \frac{4}{(1+x)} 
+ (C+D) \frac{(6 + 16x)}{(1+x)^2} \Big) v +  f_{r+1}({\bf 1})\\
&\sim_2& \res_x \Big(\frac{1}{x^3}-\frac{3}{x^4}+\frac{6}{x^5}\Big) \Big( \alpha(-1)^2 (1+x)^2  + 2 \alpha(-1)A  x  (1+x)^2 \\
& & \quad +\, 2 \alpha(-1) B x^2 (1+x)^2 + 2\alpha(-1) C x^3 (1+x)  +  A^2 x^2 (1+x)^2  +  2AB  x^3 (1+x)^2 \\
& & \quad  + \, 2A C x^4 (1+x) +  B^2 x^4 (1+x)^2 + 2 B C x^5 (1+ x) \Big)  \Big(  C \frac{4}{(1+x)} 
\\
& & \quad  + \, (C+D) \frac{(6+ 16x)}{(1+x)^2}  \Big) v + g_{r+1}({\bf 1})\\
&\sim_2& \res_x \Big(\frac{1}{x^3}-\frac{3}{x^4}+\frac{6}{x^5}\Big) \Big( 4\alpha(-1)^2C (1+x) 
+ \alpha(-1)^2(C+D)  (6+ 16x)  \\
& & \quad   +  \, 8 \alpha(-1) AC x  (1+x) + 2 \alpha(-1) A (C+D) x  (6+ 16x)  +  8 \alpha(-1) B C x^2 (1+x) \\
& & \quad  +  \, 2 \alpha(-1) B (C+D) x^2 (6+ 16x) + 8\alpha(-1) C^2 x^3   + 4A^2 C x^2 (1+ x) \\
& & \quad  + \,  A^2 (C+D) x^2(6+ 16 x) + 8 ABCx^3(1+x) + AB(C+D) x^3(6 + 16x)  \\
& & \quad + \,  8AC^2 x^4 + 4B^2C x^4(1+x)  + 8BC^2x^5  \Big)  v + g_{r+1}({\bf 1})\\
%&=& \Big( 8 \alpha(-1) AC + 32 \alpha(-1) A (C+D)  -16 \alpha(-1) B C  -  84 \alpha(-1) B (C+D)  \\
%& & \quad  - \, 24 \alpha(-1) C^2   -8 A^2 C - 42A^2 (C+D)  + 24 ABC +  78 AB(C+D)  +  48AC^2  \\
%& & \quad  + \, 24 B^2C   \Big)  v + g_{r+1}({\bf 1}) \\
%&=& \alpha(-1) \Big( 8 AC + 32  A (C+D)  -16  B C  -  84 B (C+D)   -  24  C^2 \Big) v + \Big(  -8 A^2 C \\
%& & \quad - \, 42A^2 (C+D)  + 24 ABC +  78 AB(C+D)  +  48AC^2  + 24 B^2C   \Big)  v + g_{r+1}({\bf 1}) \\
&=& \alpha(-1) \Big( 8 (A+ B)C + 32  A (C+D)  -24 ( B+ C)  C  -  84 B (C+D)  \Big) v \\
& & \quad  + \,  \Big(  -8 A^2 C - 42A^2 (C+D)  + 24 (A+ B) BC +  78 AB(C+D)  +  48AC^2    \Big)  v + g_{r+1}({\bf 1}) \\
&\sim_2 &  \alpha(-1) \Big( 8 (A+ B)C + 32  A (C+D)  -24 ( B+ C)  C  -  84 B (C+D)  \Big) v \\
& & \quad  + \,  \Big(  -8 A^2 C - 42A^2 (C+D)  + 24 (A+ B) BC  +  48AC^2    \Big)  v + g_{r+1}({\bf 1}) ,
\end{eqnarray*}
for some $f_{r+1}({\bf 1}), g_{r+1} ({\bf 1}) \in F_{r+1}({\bf 1})$, which gives the result. 
\end{proof}

Recall the variables $Y, Z$, and $W$  at level $n=2$ as defined in Section \ref{zero-modes-section} given by  
\begin{eqnarray*}
Y &=& \frac{1}{12} (x^2 - 2y- \tilde y)\\
Z &=& \frac{1}{32} ( x^3 + 2x\tilde y + \tilde z)  \\
W &=& -\frac{1}{40}  \left(2z+ \tilde z+ 2xy-2x \tilde y-3x^3\right). 
\end{eqnarray*}

We now determine formulas for these variables necessary to determine the highest degree relations for these generators in $A_2(V) = \mathbb{C}[x,y]\langle Y, Z, W \rangle/I_2$ as given in Lemma \ref{lower-order-lemY}. 

\begin{lem}\label{Ymult}
For $v \in F_r({\bf 1})$, we have 
\[Y*_2v \sim_2 \frac{1}{2}B(C+D)v + f_{r}({\bf 1})\]
for $f_{r}({\bf 1}) \in F_{r}( {\bf 1})$.  
\end{lem}

\begin{proof}
Applying Lemmas \ref{appendixlem1} and \ref{x^2tildey}, we have that for $v \in F_r({\bf 1})$,  
\begin{eqnarray*}
Y*_2v &=& \frac{1}{12} (x^2 - 2y- \tilde y)*_2 v =\frac{1}{12}(2(x^2-y)-(x^2+\tilde y))*_2 v\\
&=& \frac{1}{12}\Big(-12A(B+C)-12B(C+D)-4B(A+B)-16A(C+D)\\
&&\quad - \, 4C(A+B)+42B(C+D)+12C(B+C)\Big)v +  f_{r}(\bf{1}) \\
&=&\frac{1}{12}\Big(30B(C+D)-4(A+B)(B+C)-16A(C+D)\\
&&-12A(B+C)+12C(B+C)\Big)v+ f_{r}(\bf{1}),
\end{eqnarray*} 
for some $f_{r}({\bf 1})\in F_{r}({\bf 1}).$  Applying Eqn.\   \eqref{O24}, we have that $-4(A+B)(B+C)v \sim_2 4A(C+D)v$, so that
\[Y*_2v\sim_2\frac{1}{12}(30B(C+D)-12A(B+2C+D)+12C(B+C))v+  f_{r}({\bf{1}}).\]
Next, we apply \eqref{O24}, to obtain $-12A(B+2C+D)v  \sim_2 12B(B+C)$ so that 
\begin{eqnarray*}
Y*_2v &\sim_2& \frac{1}{12}(30B(C+D)+12B(B+C)+12C(B+C))  + f_{r} ({\bf 1})\\
&=& \frac{1}{12}(30B(C+D)+12(B+C)^2)v+ f_{r}(\bf{1}).
\end{eqnarray*}
 Finally, by Eqn.\ \eqref{O23}, $12(B+C)^2v \sim_2  -24B(C+D)v$, so that 
\[Y*_2v\sim_2\frac{1}{2}B(C+D)v+f_{r}(\bf{1}).\]
\end{proof}

\begin{lem}\label{Wmult}
For $v \in F_r ({\bf 1})$, we have 
\[W *_2v \equiv_2 \frac{1}{2}A^2(C+D)v+ f_{r+1}(\bf{1}),\]
 for some $f_{r+1}({\bf 1})\in F_{r+1}({\bf 1})$.
\end{lem}

\begin{proof}
We  substitute in  Eqn.\ (\ref{multiplication-one-one}), Lemmas \ref{x^2tildey},   \ref{ztildexcube}, and  \ref{zxy}, then  expanding and using \eqref{O21}, \eqref{O??}, then simplifying and using $\alpha(-1)-\alpha(-3)=A-B$,  then applying Eqns.\ \eqref{coreq7} and \eqref{coreq9}, Lemma \ref{C^3},  and finally, Eqn.\ \eqref{coreq3}, we have 
\begin{eqnarray*}
W *_2v &=& -\frac{1}{40}(2z+\tilde z+2xy - 2 x \tilde y -3x^3)*_2v\\
&= &-\frac{1}{40}(2(z+xy)+(\tilde z-x^3)-2x  *_2 (x^2+\tilde y))*_2v\\
&\equiv_2& -\frac{1}{40} \Big(2\big( \alpha(-1) \big(8(A+B)C +32A(C+D) -24(B+C)C-84B(C+D)\big)\\
& & \quad  - \, 8A^2C-42A^2(C+D) +24BC(A+B)+48AC^2  \big)+ C^2 \big(-24(A+B)\\
& & \quad + \, 88(B+C) \big)   + \alpha(-3) \big(84B(C+D) + 24C(B+C) -8C(A+B) \\
& & \quad  - \,  32A(C+D) \big) - 2\big( \alpha(-3) + 3C + 6(C+D) \big) \big( 16A(C+D)+4C(A+B)\\
& & \quad - \, 42B(C+D)-12C(B+C)\big) \Big)v +   f_{r+1}({\bf 1}) \\
%&\sim_2 & -\frac{1}{40} \Big(16 \alpha(-1) (A+B)C + 64 \alpha(-1) A(C+D) -48 \alpha(-1) (B+C)C \\
%& & \quad - \, 168 \alpha(-1) B(C+D) - 16 A^2C-84A^2(C+D) +48BC(A+B) + 96AC^2 \\
%& & \quad   - \, 24C^2(A+B)+88C^2 (B+C)   + 84 \alpha(-3) B(C+D) + 24 \alpha(-3) C(B+C)\\
%& & \quad  - \, 8 \alpha(-3) C(A+B) - 32 \alpha(-3) A(C+D)  -  32 \alpha(-3)  A(C+D) \\
%& & \quad - \, 8 \alpha(-3) C(A+B) + 84 \alpha(-3)  B(C+D) + 24  \alpha(-3) C(B+C) \\
%& & \quad - \,  24C^2(A+B) + 72C^2(B+C)   \Big)v +  f_{r+1}({\bf 1}) \\
%&=& 
&\sim_2 & -\frac{1}{40} \Big(16 (A-B) (A+B)C+ 64 (A-B) A(C+D) -48 (A-B) (B+C)C \\
& & \quad - \, 168 (A-B) B(C+D) - 16 A^2C- 84A^2(C+D) +48BC(A+B)  \\
& & \quad +  \,  96AC^2 - 48C^2(A+B)+  160 C^2 (B+C)    \Big)v +  f_{r+1}({\bf 1}) \\
&\sim_2 & -\frac{1}{40} \Big( 80 B^2C    - 20 A^2(C+D) +  160 C^2 (B+C)    \Big)v +  f_{r+1}({\bf 1}) \\
&\sim_2 & -\frac{1}{40} \Big(  - 20 A^2(C+D) +  80 C( B+C)^2  \Big)v +  g_{r+1}({\bf 1}) \\
&\sim_2 & -\frac{1}{40} \Big(  - 20 A^2(C+D) \Big)v +  g_{r+1}({\bf 1}) ,
\end{eqnarray*} 
for some $f_{r+1}({\bf 1}), g_{r+1} ({\bf 1}) \in F_{r+1}({\bf 1})$, which gives the result. 
\end{proof}

\begin{lem}\label{Zmult}
Given $v\in F_r({\bf 1})$, 
\[ Z *_2v \sim_2   \frac{1}{2} BC^2v +  f_{r+1} ({\bf1}),\] 
for some $f_{r+1}({\bf1})\in F_{r+1}({\bf 1})$. 
\end{lem}

\begin{proof}
Applying Lemmas \ref{x^2tildey} and \ref{zxtildey}, along with \eqref{multiplication-one-one}, then applying Eqns.\ \eqref{O21} and \eqref{O??}, and simplifying,  and applying Lemma \ref{C^3}, we have
\begin{eqnarray*}
Z *_2 v &=& \frac{1}{32}(x^3+x\tilde y+x\tilde y+\tilde z)*_2v=\frac{1}{32}(x(x^2+\tilde y)+x\tilde y+\tilde z)v\\
& \sim_2  &\frac{1}{32}\Big( \big(\alpha(-3)+3C+6(C+D) \big) \big(16A(C+D)+4C(A+B)-42B(C+D)\\
& &\quad - \, 12C(B+C)  \big)  -4\alpha(-3)C((A+B)-3(B+C)) -4C^2(3A-10B)  \\
& & \quad + \, \alpha(-3)(C+D)(-16A+42B)\Big)v+  f_{r+1}({\bf 1}) \\
&\sim_2 & \frac{1}{32}\Big(16\alpha(-3)A(C+D)+  4\alpha(-3) C(A+B) + 12 C^2(A+B) \\
& & \quad - \,  42\alpha(-3) B(C+D) -  \big(\alpha(-3)+3C\big) 12C(B+C) \\
& & \quad - \,  4\alpha(-3)C((A+B)-3(B+C)) -4C^2(3A-10B)  \\
& & \quad + \, \alpha(-3)(C+D)(-16A+42B)\Big)v + f_{r+1}({\bf 1}) \\
&=& \frac{1}{32}\Big(    - 36C^3  + 16BC^2  \Big)v+ f_{r+1}({\bf 1}) \\
&\sim_2& \frac{1}{2}  BC^2 v+ g_{r+1}({\bf 1}) \\
\end{eqnarray*} 
for some  $f_{r+1}({\bf 1}), g_{r+1} ({\bf 1}) \in F_{r+1}({\bf 1})$.
\end{proof}

\end{document}